%% file: article.tex
\documentclass[11pt, a4paper]{amsart}

\usepackage{microtype}
\usepackage{times}
\AtBeginDocument{
  \DeclareSymbolFont{AMSb}{U}{msb}{m}{n}
  \DeclareSymbolFontAlphabet{\mathbb}{AMSb}
}

\usepackage[dvipsnames]{xcolor}
\usepackage{graphicx}
\usepackage{tikz}
\usepackage{subfig}

\graphicspath{ {plots/} {svg-inkscape/} }

\usepackage[a4paper, pdftex, left=2cm, top=2cm, right=2cm, bottom=2cm]{geometry}
\usepackage[final]{pdfpages}
\usepackage{changepage}
\usepackage[foot]{amsaddr}

\usepackage{url}
\usepackage{hyperref}
\hypersetup{
    breaklinks=true,
    bookmarksopen=true,
    pdftitle={Adaptive quadratures for nonlinear approximation of low-dimensional PDEs using smooth neural networks}, % Change here
    pdfauthor={Alexandre Magueresse, Santiago Badia}, % Change here
    pdfsubject={adaptive quadrature, integration, PDE, neural networks}, % Change here
    pdfkeywords={adaptive quadrature, integration, PDE, neural networks}, % Change here
    colorlinks=true,
    linkcolor=blue,
    citecolor=blue,
    filecolor=blue,
    urlcolor=blue
}
\newcommand{\fig}[1]{Fig.~\ref{fig:#1}}
\newcommand{\tab}[1]{Tab.~\ref{tab:#1}}
\newcommand{\eq}[1]{Eq.~\eqref{eq:#1}}
\newcommand{\alg}[1]{Alg.~\ref{alg:#1}}
\newcommand{\prop}[1]{Prop.~\ref{prop:#1}}
\newcommand{\lem}[1]{Lem.~\ref{lem:#1}}
\newcommand{\sect}[1]{Sect.~\ref{sect:#1}}
\newcommand{\subsect}[1]{Subsect.~\ref{subsect:#1}}
\newcommand{\app}[1]{Apdx.~\ref{app:#1}}

\usepackage{csquotes}
\usepackage[backend=biber, defernumbers=false, maxbibnames=5, style=numeric-comp, sorting=none, isbn=false, bibencoding=utf8, safeinputenc, url=false, doi=true, giveninits=true]{biblatex}

\usepackage{float}
\usepackage{framed}
\usepackage{booktabs}
\usepackage{multirow}
\usepackage{algorithm}
\usepackage{algpseudocode}
\makeatletter
\algdef{S}[IF]{IfNoThen}[1]{\algorithmicif\ #1}
\makeatother
\usepackage[inline]{enumitem}

\newtheorem{lemma}{Lemma}
\newtheorem{proposition}{Proposition}
\newtheorem{definition}{Definition}

\usepackage{amsmath, amsfonts, amssymb, amscd}
\usepackage{mathtools}

\newcommand{\bm}[1]{\boldsymbol{#1}}

\let\epsilon\varepsilon
\let\phi\varphi

\newcommand{\NN}{\mathbb{N}}
\newcommand{\RR}{\mathbb{R}}

\newcommand{\oo}[2]{\left]#1, #2\right[}
\newcommand{\oc}[2]{\left]#1, #2\right]}
\newcommand{\cc}[2]{\left[#1, #2\right]}
\newcommand{\co}[2]{\left[#1, #2\right[}
\newcommand{\card}[1]{\left\vert #1 \right\vert}

\newcommand{\D}{\,\text{d}}

\newcommand{\abs}[1]{\left| #1 \right|}
\newcommand{\aabs}[1]{\left\| #1 \right\|}
\newcommand{\scal}[2]{\left\langle #1, #2 \right\rangle}

\newcommand{\floor}[1]{\left\lfloor #1 \right\rfloor}

\DeclareMathOperator{\ReLU}{ReLU}
\DeclareMathOperator{\abse}{ReLU_{\begingroup\mathgroup-1 \varepsilon\endgroup}}
\DeclareMathOperator{\sinc}{sinc}
\DeclareMathOperator{\erf}{erf}
\DeclareMathOperator{\lip}{Lip}
\newcommand{\rhoe}{\rho_\epsilon}

\usepackage{acronym}
\acrodef{pde}[PDE]{partial differential equation}
\acrodef{mc}[MC]{Monte Carlo}
\acrodef{aq}[AQ]{Adaptive Quadrature}
\acrodef{fem}[FEM]{Finite Element Method}
\acrodef{nn}[NN]{Neural Network}
\acrodef{pinn}[PINN]{Physics-Informed Neural Network}
\acrodef{cpwl}[CPWL]{Continuous Piecewise Linear}

\title[Adaptive quadratures for neural networks]{Adaptive quadratures for nonlinear approximation of low-dimensional PDEs using smooth neural networks}
\date{\today}

\author[A. Magueresse]{Alexandre Magueresse$^{\dagger\ast}$}
\address{$^\dagger$School of mathematics\\Monash university\\Clayton\\Victoria 3800\\Australia}
\email{alexandre.magueresse@monash.edu}
\author[S. Badia]{Santiago Badia$^{\dagger \sharp}$}
\address{$^\sharp$ Centre Internacional de M\`etodes Num\`erics a l'Enginyeria, Campus Nord, 08034 Barcelona, Spain.}
\email{santiago.badia@monash.edu}
\thanks{$^\ast$Corresponding author}

\begin{document}

\begin{abstract}
  \input{src/0.abstract}
\end{abstract}

\maketitle

\section{Introduction}
\label{sect:introduction}
\input{src/1.introduction}

\section{Preliminaries}
\label{sect:preliminaries}
\input{src/2.preliminaries}

\section{Regularisation of neural networks}
\label{sect:regularisation}
\input{src/3.regularisation}

\section{Global approximation of an activation function by a CPWL function}
\label{sect:cpwlisation}
\input{src/4.cpwlisation}

\section{Adaptive meshes and quadratures for neural networks}
\label{sect:adaptivity}
\input{src/5.decomposition}

\section{Numerical experiments}
\label{sect:experiments}
\input{src/6.experiments}

\section{Conclusion}
\label{sect:conclusion}
\input{src/7.conclusion}

\section*{Acknowledgements}
\label{sect:acknowledgements}
\input{src/8.acknowledgements}

\appendix

\section{Proofs}
\label{app:proofs}
\input{src/9.proofs}

\section{Algorithms}
\label{app:algorithms}
\input{src/10.algorithms}

\printbibliography

\end{document}

%% file: src/0.abstract.tex
Physics-informed neural networks (PINNs) and their variants have recently emerged as alternatives to traditional partial differential equation (PDE) solvers, but little literature has focused on devising accurate numerical integration methods for neural networks (NNs), which is essential for getting accurate solutions. In this work, we propose adaptive quadratures for the accurate integration of neural networks and apply them to loss functions appearing in low-dimensional PDE discretisations. We show that at opposite ends of the spectrum, continuous piecewise linear (CPWL) activation functions enable one to bound the integration error, while smooth activations ease the convergence of the optimisation problem. We strike a balance by considering a CPWL approximation of a smooth activation function. The CPWL activation is used to obtain an adaptive decomposition of the domain into regions where the network is almost linear, and we derive an adaptive global quadrature from this mesh. The loss function is then obtained by evaluating the smooth network (together with other quantities, e.g., the forcing term) at the quadrature points. We propose a method to approximate a class of smooth activations by CPWL functions and show that it has a quadratic convergence rate. We then derive an upper bound for the overall integration error of our proposed adaptive quadrature. The benefits of our quadrature are evaluated on a strong and weak formulation of the Poisson equation in dimensions one and two. Our numerical experiments suggest that compared to Monte-Carlo integration, our adaptive quadrature makes the convergence of NNs quicker and more robust to parameter initialisation while needing significantly fewer integration points and keeping similar training times.

%% file: src/1.introduction.tex
In the last years, the use of \acp{nn} to approximate \acp{pde} has been widely explored as an alternative to standard numerical methods such as the \ac{fem}. The idea behind these methods is to represent the numerical solution of a \ac{pde} as the realisation of a \ac{nn}. The parameters of the \ac{nn} are then optimised to minimise a form of the residual of the \ac{pde} or an energy functional. The original \ac{pinn} method \cite{raissi2019} considers the minimisation of the \ac{pde} in the strong form. Other frameworks recast the \ac{pde} as a variational problem and minimise the weak energy functional, like in Variational \acp{pinn} (VPINNs, \cite{kharazmi2019}), the Deep Ritz Method (DRM, \cite{yu2018}) and the Deep Nitsche Method (DNM, \cite{liao2019}).

\subsection*{Comparison between the FEM and NNs}

\Acp{nn} activated by the $\ReLU$ function (Rectified Linear Unit) constitute a valuable special case, as they can emulate linear finite elements spaces on a mesh that depends on the parameters of the network \cite{he2018}. This enables one to draw close links between the well-established theory of the \ac{fem} and the recently introduced neural frameworks \cite{opschoor2020}.

There are additional benefits to using the ReLU activation function, such as addressing the issue of exploding or vanishing gradients and being computationally cheap to evaluate. However, the numerical experiments in \cite{adcock2021} suggest that the theoretical approximation bounds for $\ReLU$ networks are not met in practice: even though there exist network configurations that realise approximations of a given accuracy, the loss landscape is rough because of the low regularity of $\ReLU$. As a result, gradient-based optimisation algorithms are largely dependent on the initial parameters of the network and may fail to converge to a satisfying local minimum. It was also shown in \cite{hayou2019} that training a network with a smooth activation function helps the network converge.

Compared to the \ac{fem}, \acp{nn} can be differentiated thanks to automatic differentiation rather than relying on a numerical scheme (and ultimately a discretisation of the domain). Another advantage of \acp{nn} is that they are intrinsically adaptive, in the sense that contrary to traditional \ac{pde} solvers that rely on a fixed mesh and fixed shape functions, \acp{nn} can perform nonlinear approximation. Finally, the machine learning approach proposes a unified setting to handle forward and inverse problems that can involve highly nonlinear differential operators. However, properly enforcing initial and boundary conditions remains one of the major challenges for \acp{nn} \cite{berrone2022}, especially when the domain has a complex geometry.

\subsection*{Numerical integration of NNs}

Current efforts in the literature are mainly focused on developing novel model architectures or solving increasingly complex \acp{pde}, while there are comparatively few results on the convergence, consistency, and stability of \acp{pinn} and related methods. These are the three main ingredients that provide theoretical guarantees on the approximation error, and they remain to be established for \acp{nn}.

In particular, the numerical quadratures that are used to evaluate the loss function have mostly been investigated experimentally, and the link between the integration error of the loss function, often referred to as the generalisation gap, and the convergence of the network is still not fully understood. The overwhelming majority of implementations of \acp{nn} solvers rely on \ac{mc} integration because it is relatively easy to implement. However, it is known that its convergence rate with respect to the number of integration points is far from optimal in low-dimensional settings. For example, the integration error decays as $n^{-2}$ for the trapezoidal rule in dimension one, whereas it is of the order of $n^{-1/2}$ for \ac{mc} regardless of the dimension. Furthermore, local behaviours of the function to integrate are generally poorly handled by \ac{mc}, because the points are sampled uniformly across the domain.

A few recent works have considered more advanced integration methods for \ac{nn} solvers, including Gaussian quadratures or quasi-\ac{mc} integration. Some studies contemplate resampling the \ac{mc} points during the training or adding more points where the strong, pointwise residual is larger. We refer the reader to \cite{wu2023} for a comprehensive review of sampling methods and pointwise residual-based resampling strategies applied to \acp{nn}. They show that the convergence of the model is significantly influenced by the choice of the numerical quadrature. A theoretical decay rate of the generalisation gap when the training points are carefully sampled was shown in \cite{longo2021}, while \cite{shin2020} proved the consistency of \acp{pinn} and related methods in the sample limit. Finally, the generalisation error was bounded with the sum of the training error and the numerical integration error in \cite{mishra2022}. However, this result requires strong assumptions on the solution and the numerical quadratures, and the constants involved in the bound are not explicit and may not be under control.

It was shown in \cite{rivera2022} that \enquote{quadrature errors can destroy the quality of the approximated solution when solving \acp{pde} using deep learning methods}. The authors explored the use of a \ac{cpwl} interpolation of the output of the network and observed better convergence compared to \ac{mc} integration. Their experiments were performed in dimension one, and the networks were interpolated on fixed, uniform meshes.

\subsection*{Motivations and contributions}

Our work aligns with current research in the theory and practice of \acp{nn} applied to approximating \acp{pde}. We seek to equip \acp{nn} with some of the tools of the \ac{fem} in terms of numerical integration, for low-dimensional ($d \leq 3$) problems arising from computational physics or engineering.

The space of \ac{cpwl} activation functions is of great interest for \acp{nn} owing to its structure of vector space and its stability by composition. This means that whenever a network is activated by a \ac{cpwl} function, its realisation is \ac{cpwl} too. Our idea is that provided that we can decompose the domain at hand into regions where the network is almost linear in each cell, we can integrate with high accuracy any integral that involves the network and its partial derivatives by relying on Gauss-like quadratures of a sufficiently high order. The linear and bilinear forms are decomposed on this mesh in a similar way as in the \ac{fem}, except that the underlying mesh is adaptive instead of fixed. We provide an algorithm to obtain such a decomposition of the domain based on the expression of the activation function and the parameters of the network.

Nevertheless, as will be shown in \sect{regularisation}, \ac{cpwl} activations are not smooth enough for the optimisation problem of a \ac{nn} approximating a \ac{pde} to be well-behaved. If the activation function is \ac{cpwl}, it needs to be regularised so that the network can be properly trained. In this work, all the networks that we train are activated by smooth functions. However, the quadrature points and the mesh are obtained from the projection of this smooth network onto a \ac{cpwl} space, when the activation function is replaced by a \ac{cpwl} approximation of the smooth activation. We design a procedure to obtain a \ac{cpwl} approximation of a smooth function on $\RR$ as a whole.

Our extensive numerical experiments demonstrate that our proposed integration method enables the \ac{nn} to converge faster and smoother, to lower generalisation errors compared to \ac{mc} integration, and to be more robust to the initialisation of its parameters while keeping similar training times.

\subsection*{Outline}

We recall the general setting of \ac{pinn} and its variants and introduce some notations related to the variational problem we consider in \sect{preliminaries}. In \sect{regularisation}, after explaining why \acp{nn} activated by \ac{cpwl} functions are not suitable to approximate \acp{pde}, we suggest several smooth approximations of $\ReLU$. In \sect{cpwlisation}, we introduce a class of activation functions that behave linearly at infinity, propose a method to approximate them by \ac{cpwl} functions on $\RR$ as a whole and prove its quadratic convergence rate. Then, in \sect{adaptivity}, we provide an algorithm to decompose the physical domain into convex regions where the \ac{nn} is almost linear. We also describe how to create Gaussian quadratures for convex polygons and analyse the integration error of our method. We present our numerical experiments in \sect{experiments}. We solve the Poisson equation in dimensions one and two on the reference segment and square, but also on a more complex domain, to show that our adaptive quadrature can adapt to arbitrary polygonal domains. We discuss our findings and conclude our work in \sect{conclusion}. In order to keep the article more readable, we decided to write the proofs of our lemmas and propositions in \app{proofs} and our algorithms in \app{algorithms}.

%% file: src/2.preliminaries.tex
Consider the boundary condition problem
\begin{equation}
    \left\{\begin{array}{cl}
        \mathcal{D} u = f & \text{ in } \Omega  \\
        \mathcal{B} u = g & \text{ on } \Gamma.
    \end{array}\right.,\label{eq:pde}
\end{equation}
where $\Omega$ is a domain in $\RR^d$ with boundary $\Gamma$, $\mathcal{D}$ is a domain differential operator, $f: \Omega \to \RR$ is a forcing term, $\mathcal{B}$ is a trace/flux boundary operator and $g: \Gamma \to \RR$ the corresponding prescribed value. We approximate the solution $u$ as the realisation of a \ac{nn} of a given architecture. We introduce some notations and recall the key principles of the \ac{pinn} framework in the rest of this section.

\subsection{Neural networks}

Our solution ansatz is a fully-connected, feed-forward \ac{nn}, which is obtained as the composition of linear maps and nonlinear activation functions applied element-wise. In this work, we separate the structure of a \ac{nn} from its realisation when the activation function and the values of the weights and biases are given. We describe the architecture of a network by a tuple $(n_0, \ldots, n_L)\in \NN^{(L+1)}$, where $L$ is the number of linear maps it is composed of, and for $1 \leq k \leq L$ the number of neurons on layer $k$ is $n_k$. We take $n_0 = d$ and since we only consider scalar-valued \acp{pde}, we have $n_L = 1$.

For each layer number $1 \leq k \leq L$, we write $\bm{\Theta}_k: \RR^{n_{k-1}} \to \RR^{n_k}$ the linear map at layer $k$, defined by $\bm{\Theta}_k \bm{x} = \bm{W}_k \bm{x} + \bm{b}_k$ for some weight matrix $\bm{W}_k \in \RR^{n_k \times n_{k-1}}$ and bias vector $\bm{b}_k \in \RR^{n_k}$. We apply the activation function $\rho: \RR \to \RR$ after every linear map except for the last one so as not to constrain the range of the model. The expression of the network is then
\[\mathcal{N}(\bm{\vartheta}, \rho) = \bm{\Theta}_L \circ \rho \circ \bm{\Theta}_{L-1} \circ \ldots \circ \rho \circ \bm{\Theta}_1,\]
where $\bm{\vartheta}$ stands for the collection of trainable parameters of the network $\bm{W}_k$ and $\bm{b}_k$. Although the activation functions could be different at each layer or even trainable, we apply the same, fixed activation function everywhere.

\subsection{Abstract setting for the loss function}

The loss function serves as a distance metric between the solution and the model. We write the loss in terms of the network itself to simplify notations, although it is to be understood as a function of its trainable parameters. Throughout this paper, we write $\scal{u}{v}_{\Omega} = \int_\Omega u v \D \Omega$ and $\scal{u}{v}_{\Gamma} = \int_{\Gamma} u v \D \Gamma$ the classical inner products on $L^2(\Omega)$ and $L^2(\Gamma)$. We write $\aabs{\cdot}_\Omega$ and $\aabs{\cdot}_\Gamma$ the corresponding $L^2$ norms.

\subsubsection{Strong formulation}

According to the original \ac{pinn} formulation \cite{raissi2019}, the network is trained to minimise the strong residual of \eq{pde}. The minimisation is performed in $L^2$, so we suppose that $\mathcal{D} u$ and $f$ are in $L^2(\Omega)$, and $\mathcal{B} u$ and $g$ are in $L^2(\Gamma)$, This ensures that the residual
\[\mathcal{R}(u) = \frac{1}{2} \aabs{\mathcal{D} u - f}_\Omega^2 + \frac{\beta}{2} \aabs{\mathcal{B} u - g}_\Gamma^2\]
is well-posed, where $\beta > 0$ is a weighting term for the boundary conditions.

\subsubsection{Weak formulation}

The \ac{fem} derives a weak form of \eq{pde} that we can write in the following terms:
\[\left\{\begin{array}{l}
        \text{Find } u \in U \text{ such that} \\
        \forall v \in V, \quad a(u, v) = \ell(v).
    \end{array}\right.\]
In this formulation, $a: U \times V \to \RR$ is a bilinear form, $\ell: V \to \RR$ is a linear form, and $U$ and $V$ are two functional spaces on $\Omega$ such that $a(u, v)$ and $\ell(v)$ are defined for all $u \in U$ and $v \in V$. When the bilinear form $a$ is symmetric, coercive, and continuous on $U \times V$, and $U = V$, this weak setting can be recast as the minimisation of the energy functional
\[\mathcal{J}(u) = \frac{1}{2} a(u, u) - \ell(u)\]
for $u \in U$. In our case, $U$ is the manifold of \acp{nn} with a given architecture defined on $\Omega$. One difference with the \ac{fem} setting is that it is difficult to strongly enforce essential boundary conditions on \ac{nn} spaces. Among other approaches, it is common to multiply the network by a function that vanishes on the boundary and add a lifting of the boundary condition \cite{sukumar2022}. Instead, we follow the penalty method and modify the residual: the boundary terms that come from integration by parts do not cancel and we suppose that it is possible to alter the bilinear form $a$ into $a_\beta$ and linear form $\ell$ into $\ell_\beta$ such that $a_\beta$ is still symmetric and continuous. To ensure the coercivity of the bilinear form and weakly enforce Dirichlet boundary conditions, we add the term $\beta \scal{u}{v}_\Omega$ in the bilinear form and $\beta \scal{g}{v}_\Gamma$ in the linear form. Provided that the penalty coefficient $\beta$ is large enough, $a_\beta$ can be made coercive \cite[196]{ern2021} and we consider the energy minimisation problem
\begin{equation}
    \label{eq:problem}
    \text{Find } u \in U \text{ such that for all } v \in U, \mathcal{J}_\beta(u) = \frac{1}{2} a_\beta(u, u) - \ell_\beta(u) \leq \mathcal{J}_\beta(v).
\end{equation}

We see that the strong problem is a special case of the weak formulation. Indeed, by expanding the inner products and rearranging the terms, we obtain
\[\mathcal{R}(u) = \frac{1}{2} \left(\scal{\mathcal{D} u}{\mathcal{D} u}_\Omega + \beta \scal{\mathcal{B} u}{\mathcal{B} u}_\Gamma\right) - \left(\scal{\mathcal{D} u}{f}_\Omega + \beta \scal{\mathcal{B} u}{g}_\Gamma\right) + C,\]
where $C$ is a constant that only depends on the data $f$ and $g$. In particular, the minimum of $\mathcal{R}$ does not depend on $C$ so we can disregard this constant in the context of optimisation. If $\mathcal{B}$ is a Dirichlet operator, this expression corresponds to the penalty method. After considering the functional derivative of $\mathcal{R}$ and enforcing it to be zero, we recognise the notations of the weak form of problem \eq{pde}, where the bilinear and linear form are defined as $a_\beta(u, v) = \scal{\mathcal{D} u}{\mathcal{D} v}_\Omega + \beta \scal{\mathcal{B} u}{\mathcal{B} v}_\Gamma$ and $\ell_\beta(v) = \scal{\mathcal{D} v}{f}_\Omega + \beta \scal{\mathcal{B} v}{g}_\Gamma$. In the remainder of this work, we keep the notations of the weak form as in \eq{problem} and simply write $\mathcal{J}$ the energy to be minimised.

\subsection{Evaluation of the loss function and optimisation}

The bilinear and linear forms are sums of integrals over the domain $\Omega$ and its boundary $\Gamma$. In the general case, these integrals have to be approximated because no closed-form expressions can be obtained, or they are numerically intractable. This is why most \ac{pinn} implementations and their variants approximate the loss function with \ac{mc} integration. The integration points can be resampled at a given frequency during training to prevent overfitting or to place more points where the pointwise residual is larger \cite{wu2023}. In this work, we are precisely designing a new integration method to replace \ac{mc}.

The trainable parameters of the network are randomly initialised and a gradient descent algorithm can be used to minimise $\mathcal{J}$ with respect to $\bm{\vartheta}$. The partial derivatives of $u$ with respect to the space variable $\bm{x}$, and the gradient of the loss function with respect to the parameters $\bm{\vartheta}$ are made available through automatic differentiation.

%% file: src/3.regularisation.tex
It was shown in \cite{he2018} that \acp{nn} activated by the $\ReLU$ activation function can emulate first-order Lagrange elements, thus advocating the use of $\ReLU$ networks to approximate \acp{pde}. In this section, we show that the convergence of \ac{cpwl} networks is slow and noisy whenever the energy functional $\mathcal{J}$ involves a spatial derivative of $u$. We then provide a reason for preferring smoother activations and suggest a few alternatives for the $\ReLU$ function.

\subsection{Limitations of CPWL activation functions}

When the activation is \ac{cpwl}, the \ac{nn} $u$ is continuous and weakly differentiable (in the sense of distributions), and its partial derivatives are piecewise constant and discontinuous. Even though derivatives of order two and higher will be evaluated as zero, only the first-order derivatives of the network make pointwise sense from a theoretical point of view, while derivatives of higher order do not.

Let us consider an illustrative example to clarify the kind of problems that arise when the energy functional involves a spatial derivative of $u$. We consider the Poisson problem $u'' = f''$ in $\Omega$ with $u = f$ on $\Gamma$, which we reframe into the following variational problem using the Nitsche method
\begin{align*}
    a(u, v) & = \scal{u'}{v'}_\Omega - \scal{u' \cdot n}{v}_\Gamma - \scal{u}{v' \cdot n}_\Gamma + \beta \scal{u}{v}_\Gamma \\
    \ell(v) & = \scal{-f''}{v}_\Omega - \scal{f}{v' \cdot n}_\Gamma + \beta \scal{f}{v}_\Gamma.
\end{align*}
Here $n$ stands for the outward-pointing unit normal to $\Gamma$. We consider the manufactured solution $f = \ReLU$ on $\Omega = \cc{-1}{+1}$. In this case, $f'' = \delta_0$ is the Dirac delta centered at zero, which makes the first integral in $\ell(v)$ equal to $-v(0)$. We consider a one-layer network $u(x) = f(x + c)$ to approximate the solution to this problem, and we wish to recover $c = 0$. The energy functional has the following expression:
\[\mathcal{J}(c) = \frac{1}{2} \left\{\begin{array}{ll}
        0                                 & \text{if } c < -1            \\
        \beta c^2 + \abs{c} - (\beta - 1) & \text{if } -1 \leq c \leq +1 \\
        2 \beta c^2 - 2 (\beta - 1) c     & \text{if } c > +1
    \end{array}\right.\]
We verify that $\mathcal{J}$ shows a discontinuity at $-1$. Besides, $\mathcal{J}$ is decreasing on $\cc{-1}{0}$ and increasing on $\cc{0}{+1}$. We compute $\mathcal{J}(0) = -\frac{1}{2} (\beta - 1)$, which is negative whenever $\beta > 1$. This proves that $\mathcal{J}$ has a global minimum at $c = 0$. However, the gradient of $\mathcal{J}$ is discontinuous at $0$, being equal to $-1/2$ on the left and $+1/2$ on the right. Because $\ReLU$ is not differentiable at zero, any gradient-based optimiser will suffer from oscillations, implying slow and noisy convergence.

This example only involved one unknown, but one would already have to choose a low learning rate and apply a decay on the learning rate to make the network converge to an acceptable solution. One can imagine that when the optimisation problem involves a whole network, these oscillations can interfere with and severely hinder convergence. Ensuring that the activation function is at least $\mathcal{C}^1$ enables one to interchange the derivative and integral signs, and if it is $\mathcal{C}^2$ then the gradient of the loss function is continuous.

\subsection{Regularisation of the activation function}

The regularity of a \ac{nn} is entirely dictated by that of its activation function. In order to make the energy functional smoother while keeping the motivation of \ac{cpwl}, our idea is to approximate the \ac{cpwl} function by a smoother equivalent and replace every occurrence of the \ac{cpwl} activation by its smoother counterpart in the network.

Our approach is similar to the \enquote{canonical smoothings} of $\ReLU$ networks introduced in \cite{dong2022}. We take $\ReLU$ as an example and point out several families of smooth equivalents. In each case, we introduce the normalised variable $\bar{x} = x / \gamma$, where $\gamma$ is a constant that only depends on the choice of the family.
\begin{itemize}
    \item If we rewrite $\ReLU(x) = \frac{1}{2}(x + \abs{x})$, we see that the lack of regularity comes from the absolute value being non-differentiable at zero. We can thus find a smooth equivalent of the absolute value and replace it in this expression. We set $\gamma = 2 \epsilon$ and define $\rhoe$ as
          \[\rhoe(x) = \frac{1}{2} \left(x + \gamma \sqrt{1 + \bar{x}^2}\right).\]
    \item The first derivative of $\ReLU$ is discontinuous at zero. We can replace it with any continuous sigmoid function and obtain a smoothing of $\ReLU$ by integrating it. We set $\gamma = \frac{2}{\ln 2} \epsilon$ and define $\rhoe$ as
          \[\rhoe(x) = \frac{1}{2} \left(x + \gamma \ln(2 \cosh(\bar{x}))\right).\]
    \item The second derivative of $\ReLU$ is the Dirac delta at zero. We can draw from the well-established theory of mollifiers to obtain a smooth Dirac and integrate it twice to build a regularised version of $\ReLU$. We set $\gamma = 2 \sqrt{\pi} \epsilon$ and define $\rhoe$ as
          \[\rhoe(x) = \frac{1}{2} \left(x + x \erf(\bar{x}) + \frac{\gamma}{\sqrt{\pi}} \exp(-\bar{x}^2) \right).\]
\end{itemize}

\begin{figure}
    \centering
    \subfloat[$\rhoe$\label{fig:reg_zero}]{
        \includegraphics[width=0.31\linewidth]{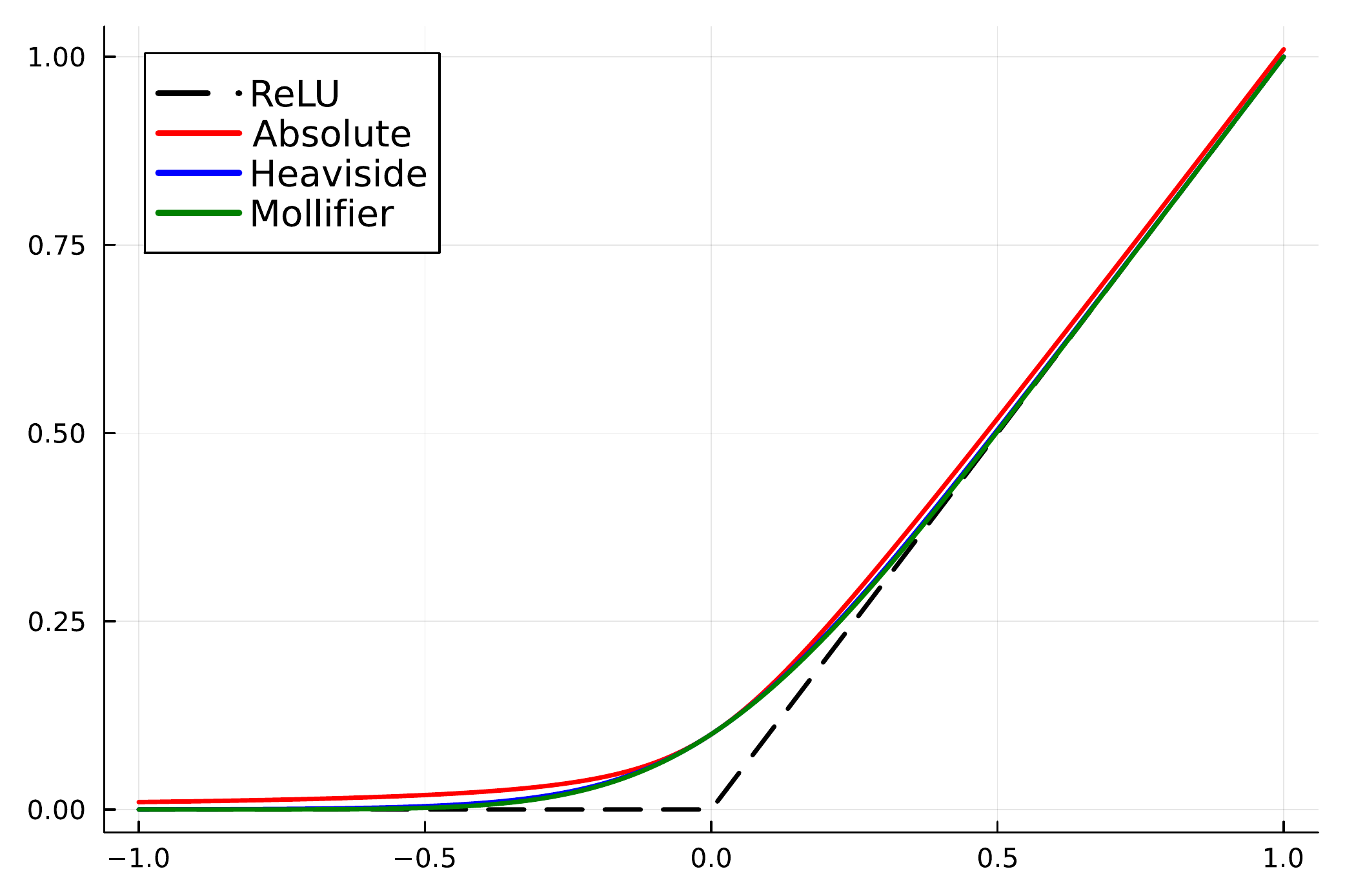}
    }
    \hfill
    \subfloat[$\rhoe'$\label{fig:reg_one}]{
        \includegraphics[width=0.31\linewidth]{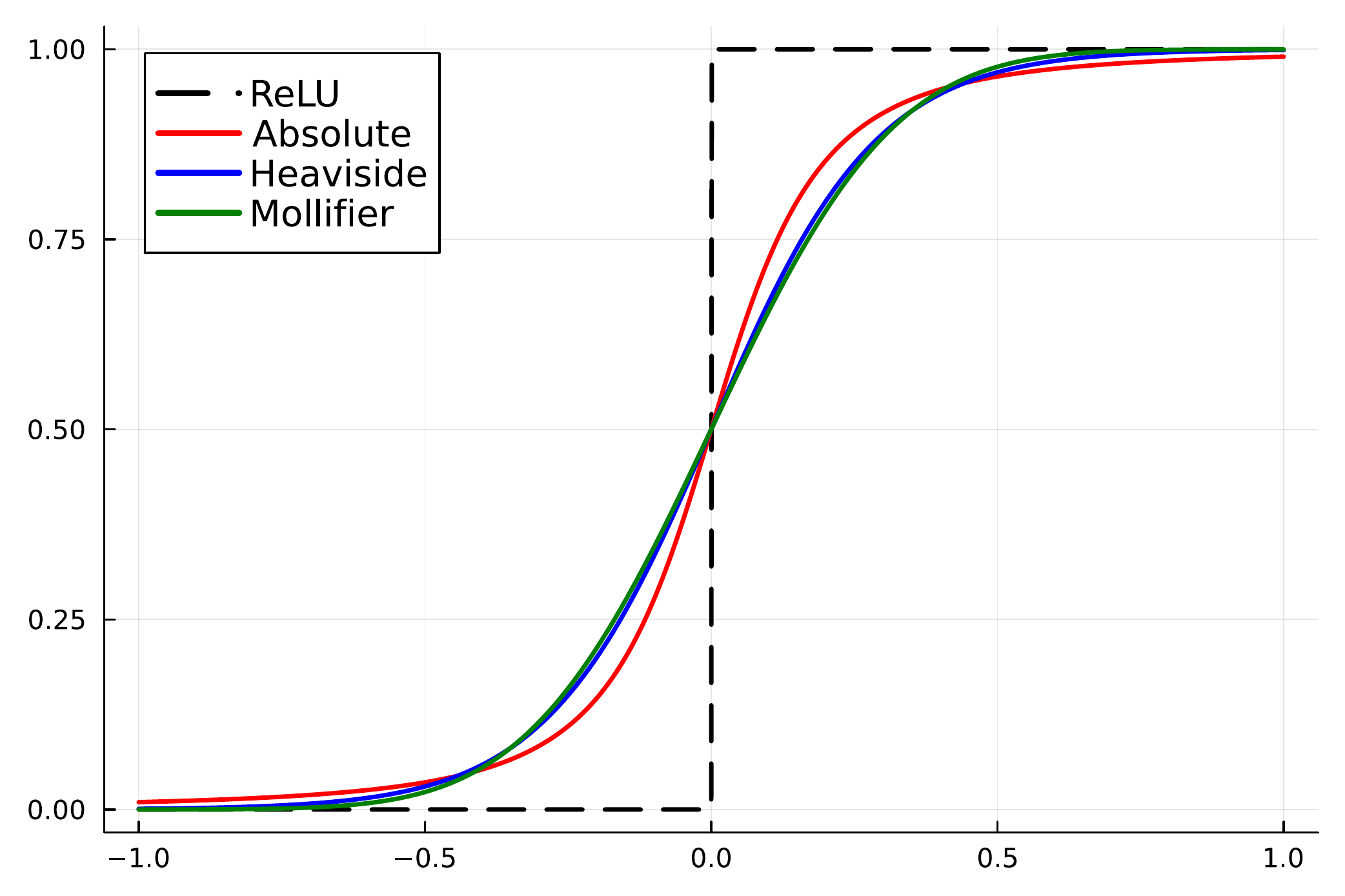}
    }
    \hfill
    \subfloat[$\rhoe''$\label{fig:reg_two}]{
        \includegraphics[width=0.31\linewidth]{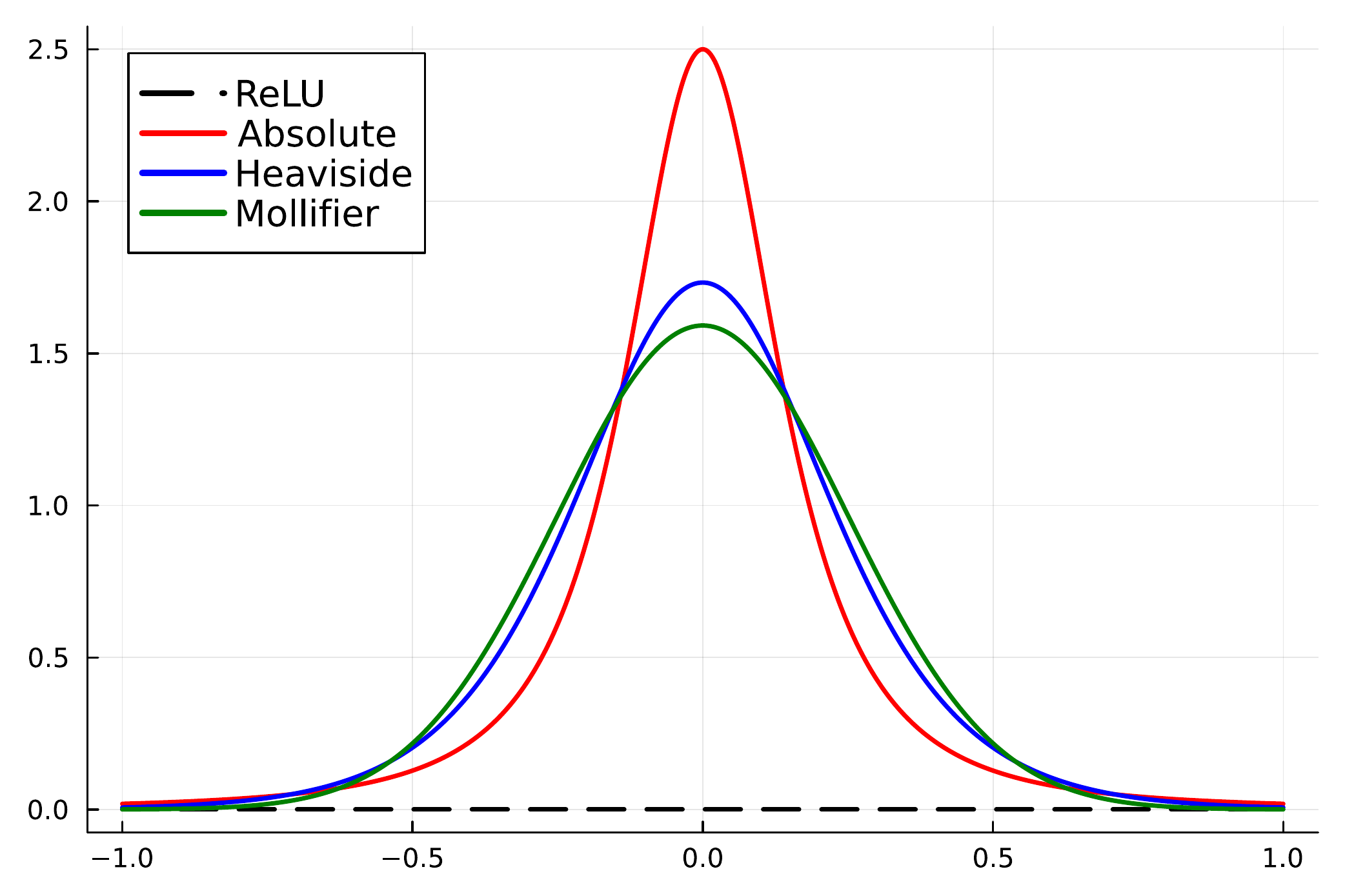}
    }
    \caption{Examples of regularising families for the $\ReLU$ function and corresponding first and second derivatives.}
    \label{fig:regularisation}
\end{figure}

The three classes of smooth equivalents of $\ReLU$ introduced above, together with their first and second derivatives, are illustrated on \fig{regularisation}. It can be easily shown that for all $\epsilon > 0$, the three smooth equivalents above have the following properties:
\begin{enumerate*}[label=(\itshape\roman*)]
    \item the function $\rhoe$ belongs to $\mathcal{C}^\infty(\RR)$;
    \item similarly to $\ReLU$, the function $\rhoe$ is convex, monotonically increasing on $\RR$, and it is symmetric around $y=-x$, i.e. for all $x \in \RR$, it holds $\rhoe(x) - \rhoe(-x) = x$;
    \item the graph of $\rhoe$ lies above that of $\ReLU$; and
    \item the parameter $\epsilon$ controls the pointwise distance between $\rhoe$ and $\ReLU$, in the sense that $\aabs{\ReLU - \rhoe}_{L^\infty(\RR)} = \rhoe(0) = \epsilon$.
\end{enumerate*}

In particular, Lemma 2.5 in \cite{dong2022} applies to all the candidates above, and we refer to Proposition 2.6 in the same article for a bound of $\aabs{\mathcal{N}(\bm{\vartheta}, \ReLU) - \mathcal{N}(\bm{\vartheta}, \rho_\epsilon)}_{L^\infty(\Omega)}$ in terms of $\aabs{\ReLU - \rho_\epsilon}_{L^\infty(\RR)}$.

The computational cost of these functions must also be taken into account. For this reason, we choose to use the family obtained by the first method. We write it $\abse$ in the rest of this work.

%% file: src/4.cpwlisation.tex
In this section, we suppose that we are given a smooth activation $\rho$ and we consider the \ac{nn} $u = \mathcal{N}(\bm{\vartheta}, \rho)$. We draw inspiration from the \ac{fem} and we would like to construct a mesh of the domain such that the behaviour of the \ac{nn} is close to linear in each cell of the mesh. We design a method that approximates the smooth activation function $\rho$ by a \ac{cpwl} function $\pi[\rho]$ and bound the distance between the two functions.

\subsection{Assumptions on the activation function}

Since we have no prior bound on the parameters of the network, we need the approximation of $\rho$ by $\pi[\rho]$ to hold on $\RR$ as a whole. Since $\pi[\rho]$ is going to be piecewise linear, we need $\rho$ to behave linearly at infinity in order for $\rho - \pi[\rho]$ to be integrable on $\RR$ as a whole. Therefore, we require that $\rho$ be twice differentiable and that $\rho''$ vanish fast enough at infinity. As will be made clear in the proof of \lem{bound}, we also need $\rho$ to be analytic in the neighbourhood of the zeros of $\rho''$. Altogether, we restrict our approach to functions that belong to the functional spaces $\mathcal{A}^\alpha$ for $\alpha > 0$, defined as
\[
    \mathcal{A}^\alpha = \left\{\begin{array}{l}
        \rho: \RR \to \RR \text{ twice differentiable on } \RR \text{ such that}                          \\
        \rho'' \text{ and } x \mapsto \abs{x}^{2 + \alpha} \rho''(x) \text{ are bounded on } \RR \text{,} \\
        \rho \text{ is analytic in a neighbourhood of the zeros of } \rho''
    \end{array}\right\}.
\]
For convenience, we also introduce the union space $\mathcal{A} = \bigcup_{\alpha > 0} \mathcal{A}^\alpha$. The following lemma shows important properties of functions in $\mathcal{A}^\alpha$, which we use in the remainder of this paper.

\begin{lemma}[Some properties of functions in $\mathcal{A}^\alpha$]
    \label{lem:integrability}
    For all $\alpha > 0$ and $\rho \in \mathcal{A}^\alpha$,
    \begin{itemize}
        \item $\rho$ behaves linearly at infinity, i.e. there exist linear functions (slant asymptotes) $\rho_{+ \infty}$ and $\rho_{- \infty}$ such that $\rho - \rho_{\pm \infty}$ vanishes at $\pm \infty$.
        \item For all $p > 1/\alpha$, $\rho - \rho_{\pm \infty} \in L^p(\RR_\pm)$.
        \item For all $q > 1/(1 + \alpha)$, $\rho' - \rho_{\pm \infty}' \in L^q(\RR_\pm)$.
        \item For all $r > 1/(2 + \alpha)$, $\rho'' \in L^r(\RR)$.
    \end{itemize}
\end{lemma}
\begin{proof}
    See \ref{proof:lem:integrability}.
\end{proof}

\subsection{Approximation space}
\label{subsect:approximation_space}

In our experiments, we observed that the way we approximate the activation function by a \ac{cpwl} counterpart plays an important role in the performance of our adaptive quadrature. There are existing algorithms and heuristics to build $\pi[\rho]$ \cite{berjon2015} but we found that $\pi[\rho]$ needs to satisfy extra properties for our method to be most efficient in practice. In particular, we noticed that it is preferable that the \ac{cpwl} approximation of $\rho$ lie below $\rho$ on the regions where $\rho$ is convex, and above $\rho$ where $\rho$ is concave. To that aim, we wish to decompose $\RR$ into intervals where $\rho$ is either strictly convex, linear, or strictly concave. We can then approximate $\rho$ on each interval by connecting the tangents to $\rho$ at free points. More formally, we define the partition induced by $\rho$ as follows.

\begin{definition}[Partition induced by a function]
    Let $\rho \in \mathcal{A}$. We define the set $Z_0(\rho) = \{I \subset \RR, \rho''_{|I} = 0\}$ that contains the intervals (possibly isolated points) where $\rho''$ is zero and $Z(\rho) = \partial Z_0(\rho) \cup \{-\infty, +\infty\}$ the set of points where $\rho''$ is zero, excluding intervals but including $\pm \infty$. We write $z(\rho) = \card{Z(\rho)}$ and $\xi_1 < \ldots < \xi_{z(\rho)}$ the ordered elements of $Z(\rho)$. The \emph{partition induced by $\rho$} is defined as $P(\rho) = \{\oo{\xi_i}{\xi_{i+1}}, 1 \leq i \leq z(\rho) - 1\}$, that is the collection of intervals whose ends are two consecutive elements of $Z(\rho)$. The elements of $P(\rho)$ are called \emph{compatible intervals for $\rho$}. We finally introduce $1 \leq \kappa(\rho) \leq z(\rho) - 1$ the number of segments in $P(\rho)$ where $\rho$ is not linear.
\end{definition}

By way of example, the second derivative of $\abse$ is always positive so $Z_0(\abse) = \varnothing$, and we infer that $Z(\abse) = \{-\infty, +\infty\}$, $P(\abse) = \{\RR\}$, $z(\abse) = 2$ and $\kappa(\abse) = 1$. However, the second derivative of $\tanh$ vanishes at zero so $Z_0(\tanh) = \{0\}$, and it follows that $Z(\tanh) = \{-\infty, 0, +\infty\}$, $P(\tanh) = \{\oo{-\infty}{0}, \oo{0}{+\infty}\}$, $z(\tanh) = 3$ and $\kappa(\tanh) = 2$.

For $\rho \in \mathcal{A}$ and $x \in \RR$, we write $T[\rho, x]$ the tangent to $\rho$ at $x$ and we extend this notation to $x = \pm \infty$ by considering the slant asymptotes to $\rho$. From the construction of $P(\rho)$, it is easy to see that for any interval $I \in P(\rho)$ such that $\rho$ is not linear on $I$ and for all $x < y \in I$, the two tangents $T[\rho, x]$ and $T[\rho, y]$ intersect at a unique point in $\oo{x}{y}$ that we write $c(x, y)$. Let $T[\rho, x, y]$ denote the \ac{cpwl} function defined on $\oo{x}{y}$ as $T[\rho, x]$ on $\oc{x}{c(x,y)}$ and $T[\rho, y]$ on $\co{c(x,y)}{y}$.

We are now ready to define our approximation space. For $n \geq z(\rho)$, we define the space $\mathcal{A}_n(\rho)$ of \ac{cpwl} functions that connect tangents to $\rho$ at $Z(\rho)$ as well as $n - z(\rho)$ other free points that do not belong to $Z(\rho)$. More formally, let $\mathcal{S}_n$ denote the set of increasing sequences of length $n$ in $\RR \backslash Z(\rho)$. Next, let $\mathcal{S}_n(\rho)$ be the set of increasing sequences of length $n$ in $\RR$ obtained by concatenating the elements of $\mathcal{S}_{n - z(\rho)}$ with $Z(\rho)$. For $S \in \mathcal{S}_n(\rho)$ and $1 \leq k \leq n$, let $S_k$ be the $k$-th element of $S$. For convenience, we set the convention $S_0 = -\infty$ and $S_{n+1} = +\infty$. Using the extensions of $T$ and $c$ to $\pm \infty$, $\mathcal{A}_n(\rho)$ can be defined as
\[
    \mathcal{A}_n(\rho) = \left\{\begin{array}{l}
        \rho_n: \RR \to \RR \text{ such that there exists } S \in \mathcal{S}_n(\rho) \text{ and for all } 1 \leq k \leq n \text{,}  \\
        \text{the restriction of } \rho_n \text{ to } \oo{c(S_{k-1}, S_k)}{c(S_k, S_{k+1})} \text{ is } T[\rho, S_k]
    \end{array}\right\}.
\]
As an example, the function plotted in \fig{cpwl_tanh} belongs to $\mathcal{A}_7(\tanh)$. There are four free points because $z(\tanh) = 3$. The square points are fixed and they are located at $\pm \infty$ and zero.

\subsection{Convergence rate}

For $\rho \in \mathcal{A}$, $n \geq z(\rho)$ and $p \geq 2$, we write $\pi_{n, p}[\rho]$ the best approximation of $\rho$ in $\mathcal{A}_n(\rho)$ in the $L^p$ norm. In this paragraph, we bound the distance between $\rho$ and $\pi_{n, p}[\rho]$, and determine the convergence rate of this approximation. We start by finding an upper bound for the distance between $\rho$ and $T[\rho, x, y]$ for $x < y \in I$, where $I$ is a compatible interval for $\rho$.

\begin{lemma}[Distance between a function and its tangents]
    \label{lem:bound}
    Let $\alpha > 0$ and $\rho \in \mathcal{A}^\alpha$. Let $p \in \oo{1/\alpha}{\infty}$. There exist constants $A_p(\rho)$ and $B_p(\rho)$ such that for all interval $I$ compatible for $\rho$ and for all $J = \oo{x}{y} \subset I$, it holds
    \begin{align*}
        \aabs{\rho - T[\rho, x, y]}_{L^p(J)} \leq A_p(\rho) \aabs{\rho''}_{L^{p/(2p+1)}(J)}, \\
        \aabs{\rho' - T[\rho, x, y]'}_{L^p(J)} \leq B_q(\rho) \aabs{\rho''}_{L^{p/(p+1)}(J)}.
    \end{align*}
    These bounds extend to the case $p = \infty$: there exist constants $A_\infty(\rho)$ and $B_\infty(\rho)$ such that for all interval $I$ compatible for $\rho$ and for all $J = \oo{x}{y} \subset I$, it holds
    \begin{align*}
        \aabs{\rho - T[\rho, x, y]}_{L^\infty(J)} \leq A_\infty(\rho) \aabs{\rho''}_{L^{1/2}(J)}, \\
        \aabs{\rho' - T[\rho, x, y]'}_{L^\infty(J)} \leq B_\infty(\rho) \aabs{\rho''}_{L^1(J)}.
    \end{align*}
\end{lemma}
\begin{proof}
    See \ref{proof:lem:bound}.
\end{proof}

Thanks to this lemma, we can bound the distance between $\rho$ and $\pi_{n, p}[\rho]$ and prove a quadratic convergence for all $p$. We also prove a linear convergence rate for the distance between $\rho'$ and $\pi_{n, p}[\rho]'$.

\begin{proposition}[Best approximation in $\mathcal{A}_n$]
    \label{prop:cpwl}
    Let $\alpha > 0$ and $\rho \in \mathcal{A}^\alpha$. Let $n \geq z(\rho)$, $p \in \oo{1/\alpha}{\infty}$ and $q \geq p$. Using the constants from \lem{bound}, it holds
    \begin{align*}
        \aabs{\rho - \pi_{n, p}[\rho]}_{L^q(\RR)}   & \leq \frac{A_q(\rho) \kappa(\rho)^{(2q + 1)/q} \aabs{\rho''}_{L^{q/(2q+1)}(\RR)}}{n^2},                                                    \\
        \aabs{\rho' - \pi_{n, p}[\rho]'}_{L^q(\RR)} & \leq \frac{B_q(\rho) \kappa(\rho)^{(q+1)/q} \aabs{\rho''}_{L^{q/(2q+1)}(\RR)}^{(q+1)/(2q+1)} \aabs{\rho''}_{L^\infty(\RR)}^{q/(2q+1)}}{n}.
    \end{align*}
    These bounds extend to the cases $p < \infty$, $q = \infty$ and $p = q = \infty$: using the constants from \lem{bound}, it holds
    \begin{align*}
        \aabs{\rho - \pi_{n, p}[\rho]}_{L^\infty(\RR)}   & \leq \frac{A_\infty(\rho) \kappa(\rho)^2 \aabs{\rho''}_{L^{1/2}(\RR)}}{n^2},                                       \\
        \aabs{\rho' - \pi_{n, p}[\rho]'}_{L^\infty(\RR)} & \leq \frac{B_\infty(\rho) \kappa(\rho) \aabs{\rho''}_{L^{1/2}(\RR)}^{1/2} \aabs{\rho''}_{L^\infty(\RR)}^{1/2}}{n}.
    \end{align*}
\end{proposition}
\begin{proof}
    See \ref{proof:prop:cpwl}.
\end{proof}

\subsection{Numerical implementation}

In practice, we choose $p = 2$ and implement the least-square approximation of $\rho$ on $\mathcal{A}_n(\rho)$. We write $\pi_n$ instead of $\pi_{n, 2}$. Since we are searching for the best approximant within $\mathcal{A}_n[\rho]$, the unknowns to be determined are the locations of the tangents $(\xi_k)_{1 \leq k \leq n-z(\rho)}$. We write the total $L^2$ norm between $\rho$ and $\pi_n[\rho]$ in terms of $(\xi_k)$ and seek to minimise this quantity via a gradient descent algorithm.

We are interested in two activation functions: $\abse$ and $\tanh$. Both functions are symmetric around the origin, which allows us to restrict the optimisation domain. We also verify that both functions belong to $\mathcal{A}$.
\begin{itemize}
    \item The relationship $\abse(x) - \abse(-x) = x$ holds for all $x \in \RR$. We impose the same condition on $\pi_n[\abse]$ and verify that $\abse - \pi_n[\abse]$ is even. Since $\abse$ takes bounded values on $\RR_-$, we approximate $\abse$ by $\pi_n[\abse]$ on $\RR_-$ only and recover $\pi_n[\abse]$ as a whole by symmetry. The second derivative of $\abse$ is $x \mapsto \frac{1}{2 \gamma} (1 + \overline{x}^2)^{-3/2}$, which decays as $x^{-3}$. We verify that the second derivative of $\abse$ is never zero. This shows that $\abse \in \mathcal{A}^\alpha$ with $\alpha = 1$. The slant asymptotes of $\abse$ intersect, so we can approximate $\abse$ with $2$ pieces and this approximation is $\ReLU$.
    \item The $\tanh$ function is odd, so imposing that $\pi_n[\tanh]$ is odd enforces that $\tanh - \pi_n[\tanh]$ is even. We decide to perform the minimisation on $\RR_+$ and obtain an approximation of $\tanh$ on $\RR$ as a whole by symmetry. The second derivative of $\tanh$ is $x \mapsto -2 \tanh(x) (1 - \tanh(x)^2)$. This is asymptotically equivalent to $-8 \exp(-2x)$. In particular, for all $\alpha > 0$, $x \mapsto x^{2 + \alpha} \tanh''(x)$ is bounded on $\RR$. Moreover, $\tanh''(x) = 0$ if and only if $x = 0$, and $\tanh$ is analytic in a neighbourhood of zero (its Taylor series converges within a radius of $\pi / 2$). This shows that $\tanh \in \mathcal{A}^\alpha$ for all $\alpha > 0$. Since the slant asymptotes of $\tanh$ do not intersect, the minimum number of pieces to approximate $\tanh$ on $\RR$ is three and the corresponding \ac{cpwl} function is known as the hard $\tanh$ activation.
\end{itemize}

As an example, we plot $\pi_7[\tanh]$ in \fig{cpwl_tanh}. The convergence plot in $L^2$ norm of our method applied to both functions is shown in \fig{convergence_rate}, together with the theoretical convergence rate of $-2$. We verify that the convergence is asymptotically quadratic. We notice that the approximation of $\abse$ shows a preasymptotic behaviour in which the effective convergence rate is closer to $1.5$.

\begin{figure}
    \centering
    \subfloat[Visualisation of $\pi_7(\tanh)$.\label{fig:cpwl_tanh}]{
        \includegraphics[width=0.45\linewidth]{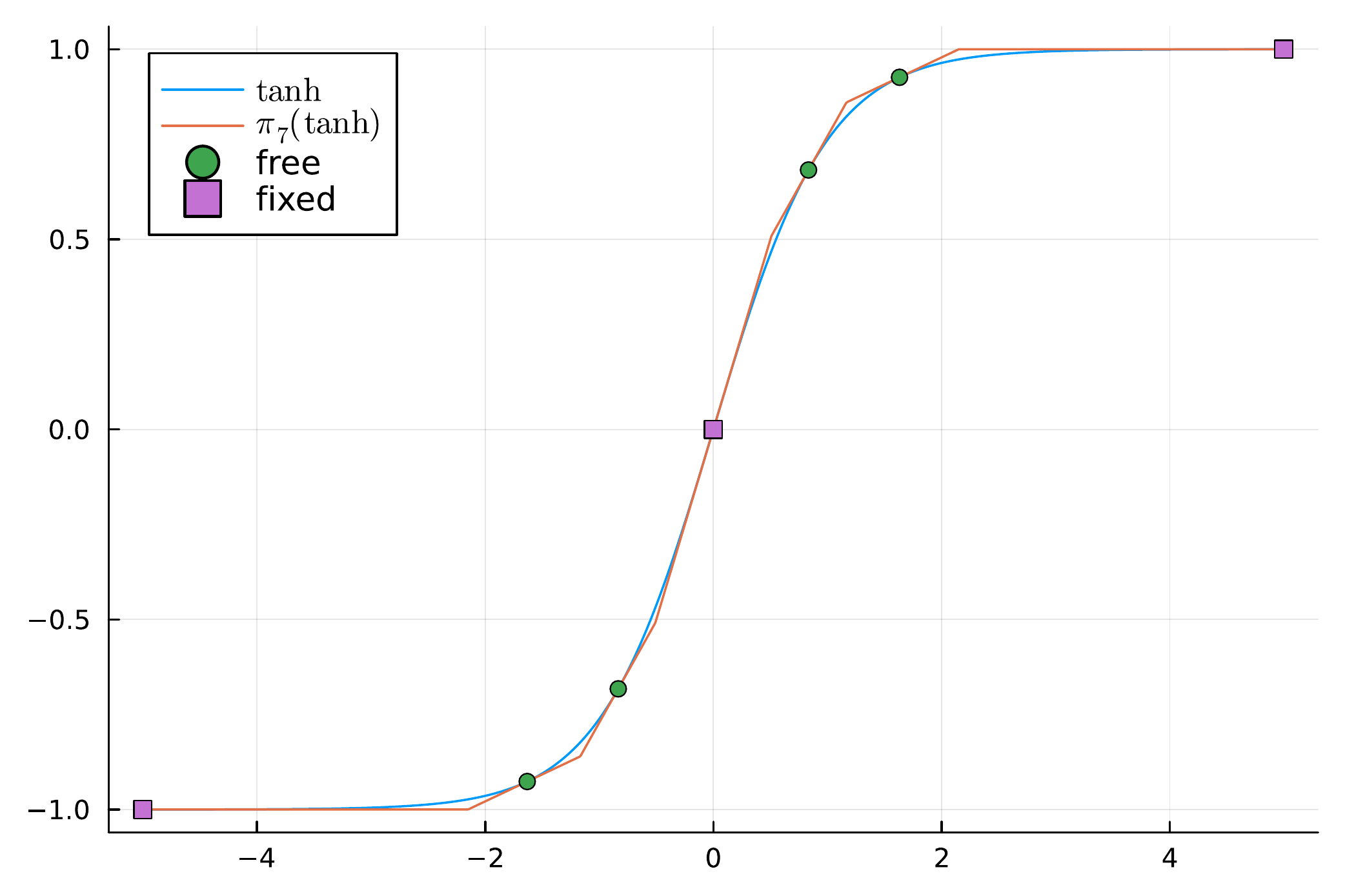}
    }
    \hfill
    \subfloat[Convergence plot in $L^2$ norm.\label{fig:convergence_rate}]{
        \includegraphics[width=0.45\linewidth]{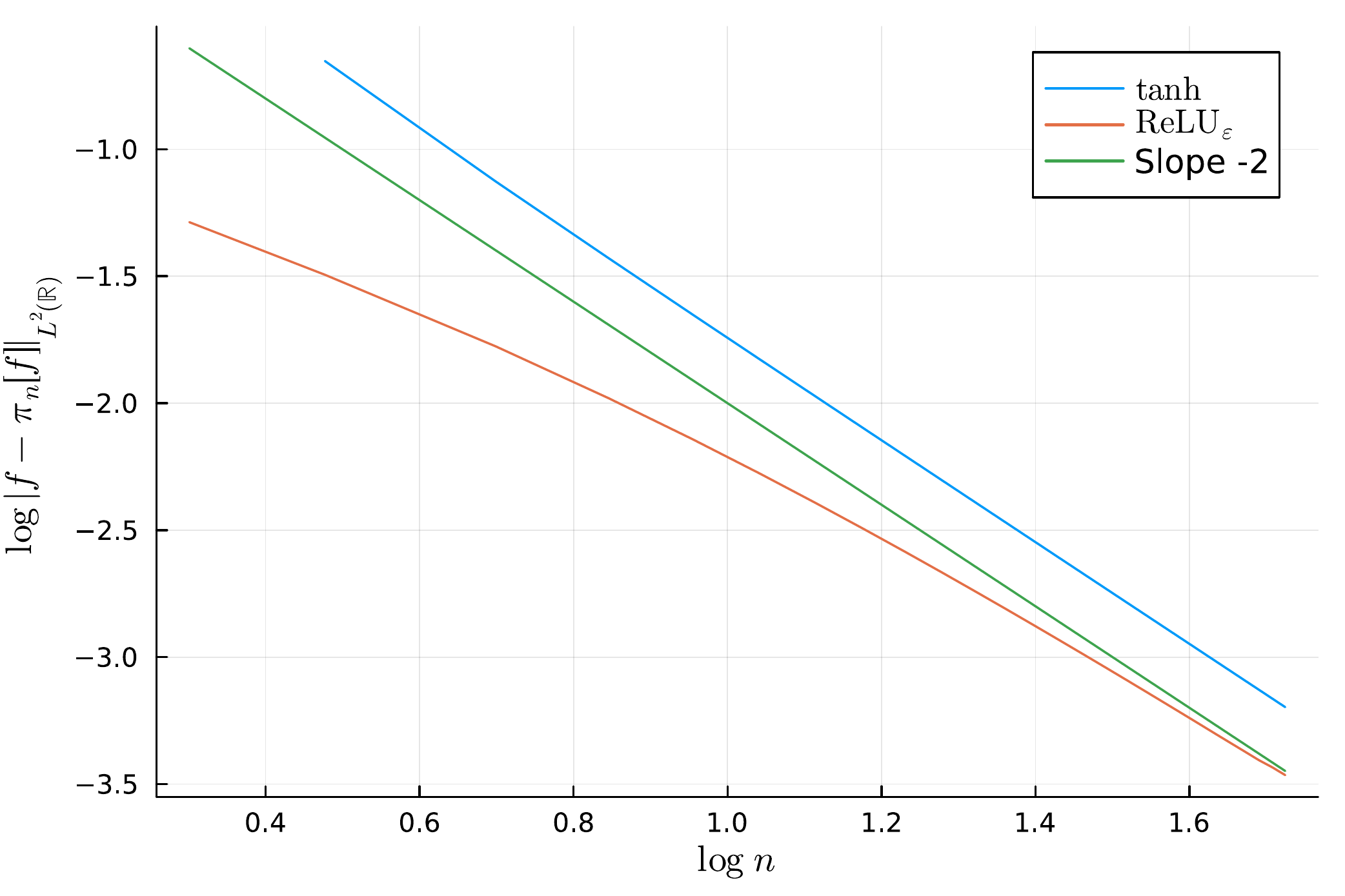}
    }
    \caption{Visualisation of $\pi_7(\tanh)$ and convergence plot of our proposed method in $L^2$ norm for $\abse$ and $\tanh$. On \protect\subref{fig:cpwl_tanh}, the square pink points are fixed, while the circle green points are chosen to minimise the $L^2$ norm of $\tanh - \pi_7[\tanh]$.}
    \label{fig:cpwl_convergence}
\end{figure}

%% file: src/5.decomposition.tex
In practice the \ac{cpwl} approximation of $\rho$ is going to be fixed, so we only keep the subscripts $n$ and $p$ in $\pi_{n, p}$ when we discuss the convergence at the end of this section. Let $u = \mathcal{N}(\bm{\vartheta}, \rho)$ and $\pi[u] = \mathcal{N}(\bm{\vartheta}, \pi[\rho])$ be two \acp{nn} sharing the same weights, activated by $\rho$ and $\pi[\rho]$ respectively. We are interested in recovering a maximal decomposition of $\Omega$ (resp. $\Gamma$) such that $\pi[u]$ is linear on these regions. We say that this decomposition is a mesh of $\Omega$ (resp. $\Gamma$) adapted to $\pi[u]$. Since $\pi[\rho]$ is fixed, this mesh depends only on $\bm{\vartheta}$. We introduce the notation $\tau_\Omega(\bm{\vartheta})$ and $\tau_\Gamma(\bm{\vartheta})$ to refer to these meshes, and $\tau(\bm{\vartheta})$ to refer to either of the two meshes. We equip each cell with a Gaussian quadrature to evaluate the loss function. Bounds for the integration error of our \ac{aq} are provided at the end of this section.

\subsection{Convexity of the mesh}

We observe that each neuron corresponds to a scalar-valued linear map followed by the composition by $\pi[\rho]$. The breakpoints of $\pi[\rho]$ define hyperplanes in the input space of each neuron. Indeed, the composition of a linear map $\bm{x} \mapsto \bm{W} \bm{x} + \bm{b}$ by $\pi[u]$ is \ac{cpwl} and the boundaries of the regions where the composition is linear correspond to the equations $\bm{w} \cdot \bm{x} + b = \xi$, where $\bm{w}$ can be any row vector of $\bm{W}$, $b$ is the coordinate of $\bm{b}$ corresponding to the same row, and $\xi$ can be any breakpoint of $\pi[\rho]$.

Intuitively, we can obtain the adapted meshes $\tau(\bm{\vartheta})$ by considering the hyperplanes attached to every neuron of the network. The cells of the mesh are the largest sets of points that lie within exactly one piece of $\pi[\rho]$ at each neuron. This is why these regions are also called the \enquote{activation patterns} of a \ac{nn}, as we can label them by the pieces they belong to at each neuron \cite{hanin2019}.

The implementation of an algorithm that, given the parameters $\bm{\vartheta}$ of a \ac{nn}, the linearisation $\pi[\rho]$ and the input space $\Omega$, outputs $\tau_\Omega(\bm{\vartheta})$ has to be robust to a large variety of corner cases, as in practice the cells can be of extremely small measures, or take skewed shapes. Besides, we are concerned with the complexity of this algorithm as it is meant to be run at every iteration (or at a given frequency) during the training phase. Fortunately enough, the only cells in the mesh that may not be convex must be in contact with the boundary.

\begin{lemma}[Convexity of the mesh]
    \label{lem:convexity}
    Let $u: \RR^d \to \RR$ be a \ac{nn} with weights $\bm{\vartheta}$ activated by a \ac{cpwl} function and $\Omega \subset \RR^d$. If a cell in $\tau_\Omega(\bm{\vartheta})$ is not convex, it has to intersect with the boundary of $\Omega$. In particular, if $\Omega$ is convex, all the cells of $\tau_\Omega(\bm{\vartheta})$ are convex.
\end{lemma}
\begin{proof}
    See \ref{proof:lem:convexity}.
\end{proof}

If $\Omega$ is not convex, we decompose $\Omega$ into a set of convex polytopes and construct an adapted mesh for each polytope independently. By \lem{convexity} these submeshes are convex. The mesh composed of all the cells of the submeshes is adapted to $u$ on $\Omega$ and contains convex cells only. Thus without loss of generality, we suppose that $\Omega$ is convex. Clipping a line by a polygon is made easier when the polygon is convex. Our method to build $\tau_\Omega(\bm{\vartheta})$ takes great advantage of the convexity of the mesh. Our algorithm is described in \alg{adaptiveMesh}.

\subsection{Representation of a linear region}
\label{subsect:representation}

Let $u_k$ denote the composition of the first $k$ layers of the network $u$, and $\tau^k(\bm{\vartheta})$ be the mesh adapted to $u_k$. We represent a linear region of $\tau^k(\bm{\vartheta})$ and its corresponding local expression by $(R, \bm{W}_R, \bm{b}_R)$, where $R \in \tau^k(\bm{\vartheta})$ is a cell of the mesh, $\bm{W}_R \in \RR^{n_k \times d}$ and $\bm{b}_R \in \RR^{n_k}$ are the coefficients of the restriction of $\pi[u_k]$ to $R$: $\pi[u_k]_{\left|R\right.}(\bm{x}) = \bm{W}_R \bm{x} + \bm{b}_R$.

\subsection{Initialisation of the algorithm}

We initialise the algorithm with the region $(\Omega, \bm{I}_{d}, \bm{0}_d)$, where $\bm{I}_d$ is the identity matrix of size $d$ and $\bm{0}_d$ is the zero vector of size $d$. The mesh $\{(\Omega, \bm{I}_{d}, \bm{0}_d)\}$ is adapted to $u_0$.

\subsection{Composition by a linear map}

Suppose that $(R, \bm{W}_R, \bm{b}_R)$ is adapted to $u_k$. The composition of $\pi[u_k]_{\left|R\right.}$ by the linear map $\bm{\Theta}_{k+1}$ remains linear on $R$, only the coefficients of the map are changed. They become $\bm{W}_{k+1} \bm{W}_R$ and $\bm{W}_{k+1} \bm{b}_R + \bm{b}_{k+1}$.

\subsection{Composition by an activation function}

Suppose that $(R, \bm{W}_R, \bm{b}_R)$ is adapted to $u_k$. We compose $\pi[u_k]_{\left|R\right.}$ by $\pi[\rho]$ componentwise. Let $\bm{w}$ be one of the row vectors of $\bm{W}_R$ and $b$ be the coordinate of $\bm{b}_R$ corresponding to the same row. Let $\xi$ be one of the breakpoints of $\pi[\rho]$. We need to find the pre-image of $\xi$ under the map $\bm{x} \mapsto \bm{w} \cdot \bm{x} + b$, which corresponds to a hyperplane in $R$. These hyperplanes have to be determined for all rows of $\bm{W}_R$ and all breakpoints of $\pi[\rho]$. We underline that the hyperplanes corresponding to a given row $\bm{w}$ of $\bm{W}_R$ are parallel, as they are level sets of $\bm{x} \mapsto \bm{w} \cdot \bm{x}$. We explain our method in detail in \app{algo_mesh}.

The process of mesh extraction in dimension two is illustrated in \fig{mesh_extraction} for a single-layer neural network. In \fig{mesh_1}, the plain lines depict two parallel hyperplanes corresponding to two different breakpoints of $\pi[\rho]$ for one of the output coordinates of the linear map. The dashed line indicates the direction vector given by the corresponding row vector of the linear map. \fig{mesh_2} pictures the hyperplanes resulting from the other output coordinates of the linear map, thus oriented in different directions. We highlight that since the output coordinates of the linear maps may have different ranges, they may not individually activate all the pieces of $\pi[\rho]$, and thus they can give rise to different numbers of hyperplanes. The clipping operation is shown in \fig{mesh_3}, and the intersections of the hyperplanes against themselves are displayed in \fig{mesh_4}. In this example, the initial cell would be cut into $13$ subcells.

\begin{figure}
    \centering
    \subfloat[\label{fig:mesh_1}]{
        % \includesvg[width=0.2\linewidth]{mesh_1.svg}
        \def\svgwidth{0.2\linewidth}
        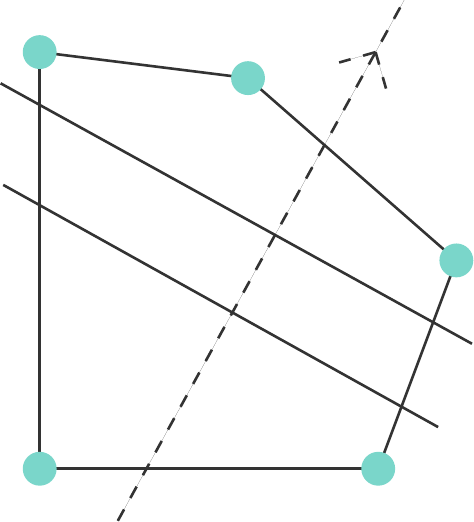
    }
    \hfill
    \subfloat[\label{fig:mesh_2}]{
        % \includesvg[width=0.2\linewidth]{mesh_2.svg}
        \def\svgwidth{0.2\linewidth}
        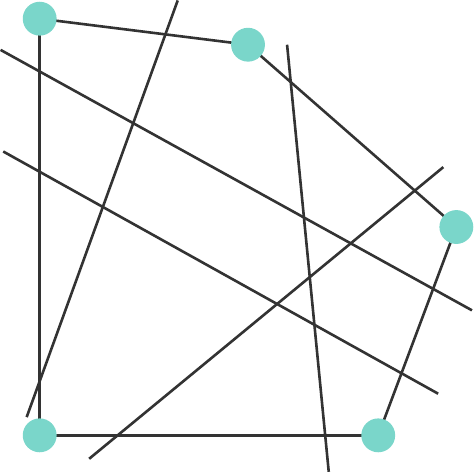
    }
    \hfill
    \subfloat[\label{fig:mesh_3}]{
        % \includesvg[width=0.2\linewidth]{mesh_3.svg}
        \def\svgwidth{0.2\linewidth}
        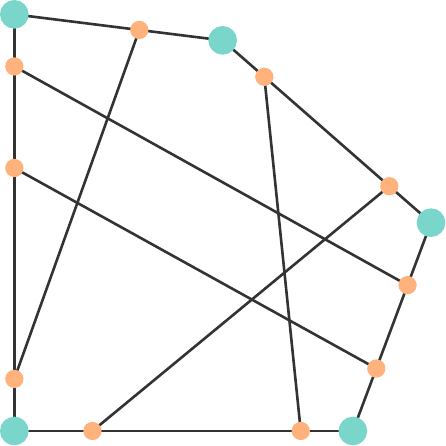
    }
    \hfill
    \subfloat[\label{fig:mesh_4}]{
        % \includesvg[width=0.2\linewidth]{mesh_4.svg}
        \def\svgwidth{0.2\linewidth}
        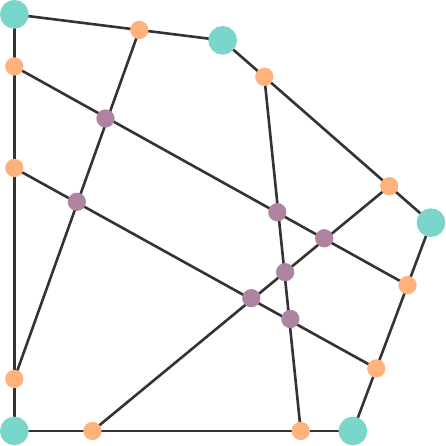
    }
    \caption{Example of a mesh extraction, corresponding to a neural network with architecture $(2, m)$ with $m \geq 4$, that is $u = \rho \circ \bm{\Theta}$, where $\bm{\Theta}: \bm{x} \mapsto \bm{W} \bm{x} + \bm{b}$, $\bm{W} \in \RR^{m \times 2}$ and $\bm{b} \in \RR^{m}$. \protect\subref{fig:mesh_1} Parallel hyperplanes associated with different breakpoints $\xi_1, \xi_2$ of $\pi[\rho]$ for one of the output coordinates of $\bm{\Theta}$ (they are orthogonal to one of the row vectors of $\bm{W}$). \protect\subref{fig:mesh_2} Hyperplanes corresponding to all the output coordinates of $\bm{\Theta}$. \protect\subref{fig:mesh_3} Clipping of the hyperplanes against the region boundary. \protect\subref{fig:mesh_4} Pairwise intersection of the hyperplanes.}
    \label{fig:mesh_extraction}
\end{figure}

\subsection{Gaussian quadratures for convex polygons}

We decompose the integrals in the linear and bilinear forms on the cells of $\tau_\Omega(\bm{\vartheta})$ and $\tau_\Gamma(\bm{\vartheta})$. In these cells, the terms that only depend on the network $\pi[u]$ and its spatial derivatives are polynomials. As a consequence, the linear and bilinear forms involving $\pi[u]$ can be computed exactly using Gaussian quadratures on segments in dimension one, and polygons in dimension two.

In dimension one, the cells of the mesh are segments. Gaussian quadrature rules are known and made available through plenty of libraries. However, in dimension two, the cells of the mesh can be arbitrary convex polygons. To handle the general case, one approach could consist in splitting each convex cell into triangles and then applying a known Gaussian quadrature for triangles. This approach is the least economical in terms of the number of quadrature points. At the opposite end of the spectrum, we could rely on Gaussian quadratures of reference $n$-gons. Still, the order of the mapping from the reference $n$-gon to the $n$-gon at hand is $n-2$, which makes the Jacobian of this mapping costly to evaluate. In this work, we strike a tradeoff and decompose each convex cell into a collection of triangles and convex quadrangles. We use a recursive algorithm to obtain this decomposition, as explained in \app{algo_polygon}. We refer to \cite{witherden2015} for numerically accurate symmetric quadrature rules for triangles and quadrangles.

\subsection{Alternatives}
\label{subsect:alternatives}

There exist computationally cheaper numerical integration methods based on evaluations of the integrands at the vertices of the mesh. Indeed, using Stokes theorem one can transform surface integrals on polygons into line integrals on their edges, and in turn, into pointwise evaluations at their vertices \cite{chin2015}. This approach requires knowing the local expression of $\pi[u]$, that is the coefficients of the linear interpolation of $u$ on each cell.

We have conducted preliminary experiments using this approach but we have observed that in addition to being numerically unstable, the overall cost including the interpolation step is not lower. Indeed, finding the best-fitting plane that passes through given points involves the inversion of a $3 \times 3$ system on each cell. The coefficients of these matrices are the integrals over each cell of the polynomials $x^2$, $xy$, $y^2$, $x$, $y$ and $1$. In many cases, the cells take skewed shapes so these matrices can be extremely ill-conditioned.

\subsection{Analysis of the integration error}

Our proposed adaptive quadrature consists in approximating $\mathcal{J}(u)$ by a quadrature on the cells of the mesh adapted to $\pi[u]$. We insist that our proposed \ac{aq} relies on the evaluation of $u$ and not $\pi[u]$. We now show how to bound the integration error of a bilinear form (the integration error of a linear form is bounded in a similar way). Let $\mathcal{I}, \mathcal{Q}: U \times U \to \RR$ be two bilinear forms. Here $\mathcal{Q}$ denotes the approximation of $\mathcal{I}$ by a numerical quadrature. We suppose that $u$ and $\pi[u]$ belong to $U$ and decompose the integration error into
\begin{align*}
    \abs{\mathcal{I}(u, u) - \mathcal{Q}(u, u)} & \leq \abs{\mathcal{I}(u - \pi[u], u)} + \abs{\mathcal{I}(\pi[u], u - \pi[u])} \\
                                                & + \abs{\mathcal{I}(\pi[u], \pi[u]) - \mathcal{Q}(\pi[u], \pi[u])}             \\
                                                & + \abs{\mathcal{Q}(u - \pi[u], \pi[u])} + \abs{\mathcal{Q}(u, u - \pi[u])}.
\end{align*}
The term $\abs{\mathcal{I}(\pi[u], \pi[u]) - \mathcal{Q}(\pi[u], \pi[u])}$ is the error incurred by any piecewise numerical integration method (e.g. it is the standard integration error in the FEM). In particular, one can consider piecewise polynomial approximations of physical parameters, forcing term, and boundary conditions up to a given order and use a Gaussian quadrature that cancels the numerical quadrature error of the bilinear form. We neglect this term in the remainder of this section. We further assume that $\mathcal{I}$ and $\mathcal{Q}$ are bounded: there exist norms $\aabs{\cdot}_{\mathcal{I}, 1}$, $\aabs{\cdot}_{\mathcal{I}, 2}$ on $U$ and a constant $C_{\mathcal{I}} > 0$ such that for all $(u, v) \in U \times U$, it holds $\abs{\mathcal{I}(u, v)} \leq C_{\mathcal{I}} \aabs{u}_{\mathcal{I}, 1} \aabs{v}_{\mathcal{I}, 2}$ and we adopt similar notations for $\mathcal{Q}$. It follows that
\begin{align*}
    \abs{\mathcal{I}(u, u) - \mathcal{Q}(u, u)} & \leq C_{\mathcal{I}} (\aabs{u - \pi[u]}_{\mathcal{I}, 1} \aabs{u}_{\mathcal{I}, 2} + \aabs{\pi[u]}_{\mathcal{I}, 1} \aabs{u - \pi[u]}_{\mathcal{I}, 2}) \\
                                                & + C_{\mathcal{Q}} (\aabs{u - \pi[u]}_{\mathcal{Q}, 1} \aabs{u}_{\mathcal{Q}, 2} + \aabs{\pi[u]}_{\mathcal{Q}, 1} \aabs{u - \pi[u]}_{\mathcal{Q}, 2}).
\end{align*}

We consider the two networks $u = \mathcal{N}(\bm{\vartheta}, \rho)$ and $\pi_{n, p}[u] = \mathcal{N}(\bm{\vartheta}, \pi_{n, p}[\rho])$, that only differ by their activation function. We show that we can bound $u - \pi_{n, p}[u]$ in terms of $\rho - \pi_{n, p}[\rho]$ in the following proposition.

\begin{proposition}[Distance between two neural networks]
    \label{prop:nns}
    Let $\Omega \subset \RR^d$ be a bounded domain, $\rho$ and $\sigma$ two continuous and almost everywhere differentiable functions such that $\rho - \sigma$, $\rho'$ and $\sigma'$ are bounded on $\RR$. We further assume that $\rho$ and $\rho'$ are Lipschitz continuous on $\RR$. Let $u_\rho = \mathcal{N}(\bm{\vartheta}, \rho)$ and $u_\sigma = \mathcal{N}(\bm{\vartheta}, \sigma)$ two \acp{nn} that only differ by their activation functions. There exist three constants $C_1$, $C_2$, $C_3$ such that
    \begin{align}
        \aabs{u_\rho - u_\sigma}_{L^\infty(\Omega)}                      & \leq C_1 \aabs{\rho - \sigma}_{L^\infty(\RR)},                                     \label{eq:nns_1}     \\
        \aabs{\nabla u_\rho - \nabla u_\sigma}_{L^\infty(\Omega, \RR^d)} & \leq C_2 \aabs{\rho - \sigma}_{L^\infty(\RR)} + C_3 \aabs{\rho' - \sigma'}_{L^\infty(\RR)}. \label{eq:nns_2}
    \end{align}
    These three constants depend on $\bm{\vartheta}$, $\aabs{\rho'}_{L^\infty(\RR)}$, $\aabs{\sigma'}_{L^\infty(\RR)}$, and the Lipschitz constants of $\rho$ and $\rho'$.
\end{proposition}
\begin{proof}
    See \ref{proof:prop:nns}.
\end{proof}

\subsubsection{Bounding zero- and first-order derivatives}

When we consider the weak formulation of the Poisson problem, both norms for $\mathcal{I}$ are $H^1$ norms, and those for $\mathcal{Q}$ are $W^{1, \infty}$ norms. \prop{cpwl} and \prop{nns} can be combined to bound $L^q$- and $W^{1, q}$-like norms of $u - \pi_{n, p}[u]$

For example, let $\alpha > 1/2$, $\rho \in \mathcal{A}_\alpha$, $p = 2 > 1/\alpha$ and $n \geq \kappa(\rho)$. We know that $\rho$ and $\pi_{n, 2}[\rho]$ are continuous and almost everywhere differentiable and their derivatives are bounded. Furthermore \prop{nns} showed that $\rho - \pi_{n, 2}[\rho]$ is bounded. Since $\rho'$ and $\rho''$ are bounded, we infer that $\rho$ and $\rho'$ are Lipschitz continuous. We use the fact that if $f$ is bounded, then it holds $\aabs{f}_{L^2(\Omega)} \leq \aabs{f}_{L^\infty(\Omega)} \card{\Omega}^{1/2}$. Combining \prop{cpwl} and \prop{nns}, we show the bounds
\begin{align*}
    \aabs{u - \pi_{n, 2}[u]}_{L^2(\Omega)} & \lesssim n^{-2}, & \aabs{\nabla u - \nabla \pi_{n, 2}[u]}_{L^2(\Omega)} & \lesssim n^{-1},
\end{align*}
where the term $n^{-2}$ has been neglected in the second bound. Altogether, we conclude that $\aabs{u - \pi_{n, 2}[u]}_{H^1(\Omega)}$ decays as $n^{-1}$.

\subsubsection{Bounding higher-order derivatives}

We now turn to the strong formulation of the Poisson problem. Since we use smooth activation functions, the strong formulation of the Poisson problem is well-posed for $u$. We reiterate that our \ac{aq} is performed on $u$ directly, not on $\pi_{n, p}[u]$. In this case, the norms of $\mathcal{I}$ and $\mathcal{Q}$ are $H^2$ and $W^{2, \infty}$ respectively. The decomposition of the integration error presented for the zero- and first-order derivatives cannot be adapted here because $\pi_{n, p}[\rho]$ is not smooth enough. To show similar integration bounds as above, one would need to approximate $\rho$ by a piecewise quadratic, globally $\mathcal{C}^1$ polynomial (e.g. a quadratic spline).

%% file: svg-inkscape/mesh_1_svg-tex.pdf_tex
%% Creator: Inkscape 1.1.2 (0a00cf5339, 2022-02-04), www.inkscape.org
%% PDF/EPS/PS + LaTeX output extension by Johan Engelen, 2010
%% Accompanies image file 'mesh_1_svg-tex.pdf' (pdf, eps, ps)
%%
%% To include the image in your LaTeX document, write
%%   \input{<filename>.pdf_tex}
%%  instead of
%%   \includegraphics{<filename>.pdf}
%% To scale the image, write
%%   \def\svgwidth{<desired width>}
%%   \input{<filename>.pdf_tex}
%%  instead of
%%   \includegraphics[width=<desired width>]{<filename>.pdf}
%%
%% Images with a different path to the parent latex file can
%% be accessed with the `import' package (which may need to be
%% installed) using
%%   \usepackage{import}
%% in the preamble, and then including the image with
%%   \import{<path to file>}{<filename>.pdf_tex}
%% Alternatively, one can specify
%%   \graphicspath{{<path to file>/}}
%% 
%% For more information, please see info/svg-inkscape on CTAN:
%%   http://tug.ctan.org/tex-archive/info/svg-inkscape
%%
\begingroup%
  \makeatletter%
  \providecommand\color[2][]{%
    \errmessage{(Inkscape) Color is used for the text in Inkscape, but the package 'color.sty' is not loaded}%
    \renewcommand\color[2][]{}%
  }%
  \providecommand\transparent[1]{%
    \errmessage{(Inkscape) Transparency is used (non-zero) for the text in Inkscape, but the package 'transparent.sty' is not loaded}%
    \renewcommand\transparent[1]{}%
  }%
  \providecommand\rotatebox[2]{#2}%
  \newcommand*\fsize{\dimexpr\f@size pt\relax}%
  \newcommand*\lineheight[1]{\fontsize{\fsize}{#1\fsize}\selectfont}%
  \ifx\svgwidth\undefined%
    \setlength{\unitlength}{136.30634573bp}%
    \ifx\svgscale\undefined%
      \relax%
    \else%
      \setlength{\unitlength}{\unitlength * \real{\svgscale}}%
    \fi%
  \else%
    \setlength{\unitlength}{\svgwidth}%
  \fi%
  \global\let\svgwidth\undefined%
  \global\let\svgscale\undefined%
  \makeatother%
  \begin{picture}(1,1.10178817)%
    \lineheight{1}%
    \setlength\tabcolsep{0pt}%
    \put(0,0){\includegraphics[width=\unitlength,page=1]{mesh_1_svg-tex.pdf}}%
    \put(0.49653848,0.30395298){\color[rgb]{0.2,0.2,0.2}\makebox(0,0)[lt]{\lineheight{1.25}\smash{\begin{tabular}[t]{l}\textbf{$\xi_1$}\end{tabular}}}}%
    \put(0.49653848,0.71662635){\color[rgb]{0.2,0.2,0.2}\makebox(0,0)[lt]{\lineheight{1.25}\smash{\begin{tabular}[t]{l}\textbf{$\xi_2$}\end{tabular}}}}%
  \end{picture}%
\endgroup%

%% file: svg-inkscape/mesh_2_svg-tex.pdf_tex
%% Creator: Inkscape 1.1.2 (0a00cf5339, 2022-02-04), www.inkscape.org
%% PDF/EPS/PS + LaTeX output extension by Johan Engelen, 2010
%% Accompanies image file 'mesh_2_svg-tex.pdf' (pdf, eps, ps)
%%
%% To include the image in your LaTeX document, write
%%   \input{<filename>.pdf_tex}
%%  instead of
%%   \includegraphics{<filename>.pdf}
%% To scale the image, write
%%   \def\svgwidth{<desired width>}
%%   \input{<filename>.pdf_tex}
%%  instead of
%%   \includegraphics[width=<desired width>]{<filename>.pdf}
%%
%% Images with a different path to the parent latex file can
%% be accessed with the `import' package (which may need to be
%% installed) using
%%   \usepackage{import}
%% in the preamble, and then including the image with
%%   \import{<path to file>}{<filename>.pdf_tex}
%% Alternatively, one can specify
%%   \graphicspath{{<path to file>/}}
%% 
%% For more information, please see info/svg-inkscape on CTAN:
%%   http://tug.ctan.org/tex-archive/info/svg-inkscape
%%
\begingroup%
  \makeatletter%
  \providecommand\color[2][]{%
    \errmessage{(Inkscape) Color is used for the text in Inkscape, but the package 'color.sty' is not loaded}%
    \renewcommand\color[2][]{}%
  }%
  \providecommand\transparent[1]{%
    \errmessage{(Inkscape) Transparency is used (non-zero) for the text in Inkscape, but the package 'transparent.sty' is not loaded}%
    \renewcommand\transparent[1]{}%
  }%
  \providecommand\rotatebox[2]{#2}%
  \newcommand*\fsize{\dimexpr\f@size pt\relax}%
  \newcommand*\lineheight[1]{\fontsize{\fsize}{#1\fsize}\selectfont}%
  \ifx\svgwidth\undefined%
    \setlength{\unitlength}{136.30634573bp}%
    \ifx\svgscale\undefined%
      \relax%
    \else%
      \setlength{\unitlength}{\unitlength * \real{\svgscale}}%
    \fi%
  \else%
    \setlength{\unitlength}{\svgwidth}%
  \fi%
  \global\let\svgwidth\undefined%
  \global\let\svgscale\undefined%
  \makeatother%
  \begin{picture}(1,0.99712315)%
    \lineheight{1}%
    \setlength\tabcolsep{0pt}%
    \put(0,0){\includegraphics[width=\unitlength,page=1]{mesh_2_svg-tex.pdf}}%
  \end{picture}%
\endgroup%

%% file: svg-inkscape/mesh_3_svg-tex.pdf_tex
%% Creator: Inkscape 1.1.2 (0a00cf5339, 2022-02-04), www.inkscape.org
%% PDF/EPS/PS + LaTeX output extension by Johan Engelen, 2010
%% Accompanies image file 'mesh_3_svg-tex.pdf' (pdf, eps, ps)
%%
%% To include the image in your LaTeX document, write
%%   \input{<filename>.pdf_tex}
%%  instead of
%%   \includegraphics{<filename>.pdf}
%% To scale the image, write
%%   \def\svgwidth{<desired width>}
%%   \input{<filename>.pdf_tex}
%%  instead of
%%   \includegraphics[width=<desired width>]{<filename>.pdf}
%%
%% Images with a different path to the parent latex file can
%% be accessed with the `import' package (which may need to be
%% installed) using
%%   \usepackage{import}
%% in the preamble, and then including the image with
%%   \import{<path to file>}{<filename>.pdf_tex}
%% Alternatively, one can specify
%%   \graphicspath{{<path to file>/}}
%% 
%% For more information, please see info/svg-inkscape on CTAN:
%%   http://tug.ctan.org/tex-archive/info/svg-inkscape
%%
\begingroup%
  \makeatletter%
  \providecommand\color[2][]{%
    \errmessage{(Inkscape) Color is used for the text in Inkscape, but the package 'color.sty' is not loaded}%
    \renewcommand\color[2][]{}%
  }%
  \providecommand\transparent[1]{%
    \errmessage{(Inkscape) Transparency is used (non-zero) for the text in Inkscape, but the package 'transparent.sty' is not loaded}%
    \renewcommand\transparent[1]{}%
  }%
  \providecommand\rotatebox[2]{#2}%
  \newcommand*\fsize{\dimexpr\f@size pt\relax}%
  \newcommand*\lineheight[1]{\fontsize{\fsize}{#1\fsize}\selectfont}%
  \ifx\svgwidth\undefined%
    \setlength{\unitlength}{128.25bp}%
    \ifx\svgscale\undefined%
      \relax%
    \else%
      \setlength{\unitlength}{\unitlength * \real{\svgscale}}%
    \fi%
  \else%
    \setlength{\unitlength}{\svgwidth}%
  \fi%
  \global\let\svgwidth\undefined%
  \global\let\svgscale\undefined%
  \makeatother%
  \begin{picture}(1,1)%
    \lineheight{1}%
    \setlength\tabcolsep{0pt}%
    \put(0,0){\includegraphics[width=\unitlength,page=1]{mesh_3_svg-tex.pdf}}%
  \end{picture}%
\endgroup%

%% file: svg-inkscape/mesh_4_svg-tex.pdf_tex
%% Creator: Inkscape 1.1.2 (0a00cf5339, 2022-02-04), www.inkscape.org
%% PDF/EPS/PS + LaTeX output extension by Johan Engelen, 2010
%% Accompanies image file 'mesh_4_svg-tex.pdf' (pdf, eps, ps)
%%
%% To include the image in your LaTeX document, write
%%   \input{<filename>.pdf_tex}
%%  instead of
%%   \includegraphics{<filename>.pdf}
%% To scale the image, write
%%   \def\svgwidth{<desired width>}
%%   \input{<filename>.pdf_tex}
%%  instead of
%%   \includegraphics[width=<desired width>]{<filename>.pdf}
%%
%% Images with a different path to the parent latex file can
%% be accessed with the `import' package (which may need to be
%% installed) using
%%   \usepackage{import}
%% in the preamble, and then including the image with
%%   \import{<path to file>}{<filename>.pdf_tex}
%% Alternatively, one can specify
%%   \graphicspath{{<path to file>/}}
%% 
%% For more information, please see info/svg-inkscape on CTAN:
%%   http://tug.ctan.org/tex-archive/info/svg-inkscape
%%
\begingroup%
  \makeatletter%
  \providecommand\color[2][]{%
    \errmessage{(Inkscape) Color is used for the text in Inkscape, but the package 'color.sty' is not loaded}%
    \renewcommand\color[2][]{}%
  }%
  \providecommand\transparent[1]{%
    \errmessage{(Inkscape) Transparency is used (non-zero) for the text in Inkscape, but the package 'transparent.sty' is not loaded}%
    \renewcommand\transparent[1]{}%
  }%
  \providecommand\rotatebox[2]{#2}%
  \newcommand*\fsize{\dimexpr\f@size pt\relax}%
  \newcommand*\lineheight[1]{\fontsize{\fsize}{#1\fsize}\selectfont}%
  \ifx\svgwidth\undefined%
    \setlength{\unitlength}{128.25bp}%
    \ifx\svgscale\undefined%
      \relax%
    \else%
      \setlength{\unitlength}{\unitlength * \real{\svgscale}}%
    \fi%
  \else%
    \setlength{\unitlength}{\svgwidth}%
  \fi%
  \global\let\svgwidth\undefined%
  \global\let\svgscale\undefined%
  \makeatother%
  \begin{picture}(1,1)%
    \lineheight{1}%
    \setlength\tabcolsep{0pt}%
    \put(0,0){\includegraphics[width=\unitlength,page=1]{mesh_4_svg-tex.pdf}}%
  \end{picture}%
\endgroup%

%% file: src/6.experiments.tex
We carry out thorough numerical experiments to validate our integration method. Choosing the best hyperparameters to train a \ac{nn} usually requires trying several combinations and picking the one that leads to the lowest generalisation error. In addition to the usual learning rate and the number of epochs, we also need to tune the design of the loss function by specifying the penalty coefficient $\beta$, and the integration method. In the case of \ac{mc} integration, we select the number of points in the domain and on the boundary, as well as the resampling frequency. Our adaptive quadrature introduces a new set of hyperparameters, namely the number of pieces in the \ac{cpwl} approximation of the activation function, the order of the quadrature in each cell as well as the remeshing frequency.

In this work, we are also concerned with reducing the computational budget involved with the training of a \ac{nn} to solve a \ac{pde}. For this reason, we consider a two-hidden layer feed-forward \ac{nn} with $10$ neurons on both layers. The number of trainable parameters is $141$ and $151$ in dimensions one and two respectively. The optimiser is chosen as ADAM and all the networks are trained for $5000$ epochs. To compensate for the low number of training iterations, we use a slightly higher learning rate compared to common practice in the literature. We set it to $10^{-2}$ or $10^{-3}$, whereas it seems to usually lie between $10^{-4}$ and $10^{-3}$. We specify the learning rate for each experiment in the sections below.

In our experiments, we compare models that have been trained with the same average number of integration points. In the case of a network trained with \ac{mc}, the number of integration points is fixed during training, but this is not the case for the method we propose here. In practice, we observe that the number of points starts at an intermediate value, then slightly decreases, and finally rises above the initial value. We still take the average number of integration points across the whole training as a reference because it is the average number of times the model is evaluated at each epoch and is a good indicator of the cost of a numerical method.

We compare \ac{aq} against \ac{mc} integration to solve the Poisson equation in dimensions one and two, both in the strong and weak form using the Nitsche formulation. We report the relative $L^2$ norm of the pointwise difference, defined as
\[E(u, \hat{u}) = \frac{\aabs{\hat{u} - u}_{L^2(\Omega)}}{\aabs{\hat{u}}_{L^2(\Omega)}}.\]
The loss function corresponding to each problem is introduced in \subsect{study_cases} and the manufactured solutions we consider are presented in \subsect{solutions}. We give an outline of our experiments and short conclusions for each of them.
\begin{itemize}
    \item We conduct a set of exploratory experiments in \subsect{exploration} to identify the best learning rate and sampling frequency, as well as to evaluate and compare the robustness of \ac{aq} and \ac{mc} to the number of integration points. We find that regardless of the integration method, the strong Poisson problem brings about higher levels of noise during the training phase compared to the weak Poisson formulation. In general, \ac{aq} can reach lower errors than \ac{mc} and reduce this noise to a certain extent while using fewer integration points. This is especially true for the weak Poisson problem, where the convergence is noticeably faster and smoother with \ac{aq}.
    \item In \subsect{initialisation}, we show that models trained with our proposed integration method are more robust to parameter initialisation compared to \ac{mc}. Furthermore, \ac{aq} consistently leads to lower errors than \ac{mc} even when it relies on fewer integration points.
    \item The next round of experiments in \subsect{reduction} shows that it is possible to reduce the number of integration points by merging small regions with their neighbours.
    \item Finally in \subsect{rhombi} we solve a Poisson equation on a slightly more complex domain. We find that for a similar number of integration points, \ac{aq} reduces the error of \ac{mc} by $70\%$.
\end{itemize}
We provide a summary of our experiments in \subsect{summary} and discuss limitations and possible extensions of our method in \subsect{discussion}.

In all the following tables, $N_\Omega$ and $N_\Gamma$ stand for the number of integration points in the domain and on the boundary, respectively. The letters $P$ and $O$ denote the number of pieces in the approximation of the activation function by a \ac{cpwl} function, and the order of the quadrature in each cell of the mesh adapted to the network.

%%%
\subsection{Study cases}
\label{subsect:study_cases}

In the following experiments, we consider two simple domains: the segment $\Omega_1 = \cc{-1}{+1}$ and the square $\Omega_2 = \cc{-1}{+1}^2$. We weakly enforce Dirichlet boundary conditions on the whole boundary of the domain. We introduce the bilinear and linear forms corresponding to each problem below.

\begin{description}
    \item[Strong formulation] Let $f: \Omega \to \RR$ and $g: \Gamma \to \RR$ be two continuous functions. We want to solve the equation $-\Delta u = f$ with the Dirichlet boundary condition $u_{\vert \Gamma} = g$. The original \ac{pinn} formulation corresponds to the strong form of the \ac{pde} and is associated with the loss
        \[J(u) = \frac{1}{2} \|\Delta u + f\|_\Omega^2 + \frac{\beta}{2} \|u - g\|_\Gamma^2,\]
        where $\beta > 0$ is a weight factor for the boundary condition. We transform the squared $L^2$ norms into inner products and obtain the following bilinear and linear forms defined on $U = V = H^2(\Omega)$
        \begin{align*}
            a(u, v) & = \scal{\Delta u}{\Delta v}_\Omega + \beta \scal{u}{v}_\Gamma, \\
            \ell(v) & = -\scal{f}{\Delta v}_\Omega + \beta \scal{g}{v}_\Gamma.
        \end{align*}
    \item[Weak formulation] We also solve a weak form of the Poisson equation and use the Nitsche method to obtain a symmetric bilinear form. In this setting, we only require $U = V = H^1(\Omega)$, $f \in H^{-1}(\Omega)$ and $g \in H^{-1/2}(\Gamma)$. The bilinear and linear forms are as follows
        \begin{align*}
            a(u, v) & = \scal{\nabla u}{\nabla v}_\Omega - \scal{\nabla u \cdot n}{v}_\Gamma - \scal{u}{\nabla v \cdot n}_\Gamma + \beta \scal{u}{v}_\Gamma, \\
            \ell(v) & = \scal{f}{v}_\Omega - \scal{g}{\nabla v \cdot n}_\Gamma + \beta \scal{g}{v}_\Gamma.
        \end{align*}
        Here, $n$ is the outward-pointing unit normal vector to $\Gamma$ and $\beta > 0$ is a coefficient large enough so that the bilinear form is coercive.
\end{description}

In all our experiments, the relative $L^2$ norm is evaluated with a quadrature of order $10$ on a $100$-cell uniform mesh in 1D and on a $100 \times 100$ uniform grid in 2D.

%%%
\subsection{Activation functions and forcing terms}
\label{subsect:solutions}

We solve the two problems we have just described with different forcing terms $f$. Throughout our experiments, all the problems that involve a given forcing term are solved using the same activation function. We report two groups of forcing terms and their corresponding activation in \tab{activations}. The notation $\sinc$ stands for the cardinal sine function $x \mapsto \sin(x)/x$, defined on $\RR \backslash \{0\}$ and continuously extended at zero by setting $\sinc(0) = 1$.

\begin{table}
    \centering
    \resizebox{0.5\linewidth}{!}{%
        \begin{tabular}{ccc}
            \toprule
            \multirow{2}{*}{Activation} & \multicolumn{2}{c}{Forcing term}                                   \\
            \cmidrule(l){2-3}
                                        & 1D                               & 2D                              \\
            \midrule
            $\abse$                     & $\sinc(3 \pi x)$                 & $\sinc(2 \pi x) \sinc(2 \pi y)$ \\
            $\tanh$                     & $\tanh(10 (x^2 - 0.5^2))$        & $\tanh(10 (x^2 + y^2 - 0.5^2))$ \\
            \bottomrule
        \end{tabular}
    }
    \caption{Forcing terms in dimensions one and two with their corresponding activation function.}
    \label{tab:activations}
\end{table}

%%%
\subsection{Relationship between integration method, learning rate, and sampling frequency}
\label{subsect:exploration}

In this initial experiment, we focus on solving the four one-dimensional problems ($\abse$/$\tanh$, weak/strong) so that exploring a large space of hyperparameters remains computationally affordable. All the networks are trained from the same seed, which means that they are initialised from the same set of weights and biases, and the \ac{mc} integration points are sampled in the same way across all experiments.

We study the connection between the learning rate, the frequency at which we resample the integration points, and the integration method. We set the learning rate to $10^{-2}$ or $10^{-3}$ and resample the integration points every $1$, $10$, $100$, $500$, $1000$ or $5000$ epochs. Since the networks are trained for $5000$ epochs, the last frequency corresponds to fixed integration points. We compare \ac{mc} with $50$ or $100$ points against \ac{aq} with several choices of quadrature order and number of pieces in $\pi[\rho]$. When the domain is one-dimensional, the boundary integrals reduce to evaluating the integrand at $\Gamma_1 = \{-1, +1\}$.

We also solve the weak Poisson problem in dimension two to further assess the effect of the integration method when the sampling frequency is fixed.

%%%
\subsubsection{Strong Poisson in 1D}

The relative $L^2$ norm after training the network for the two one-dimensional strong Poisson problems is shown in \tab{1D_pois_abs} and \tab{1D_pois_tanh}. We first observe that the sampling frequency and the number of integration points have a large influence on the final error. It is unclear whether increasing the number of integration points helps improve performance, as both trends appear for the two activation functions and integration methods.

Generally speaking, we note that \ac{aq} works best with high refresh rates (every $100$ or fewer epochs), especially when the number of pieces and the order of the quadrature are low. However, it is almost always the case that \ac{aq} reaches similar or better accuracies than \ac{mc} with fewer points. For instance, when the resampling rate is set to $10$ epochs and the learning rate to $10^{-2}$ for the $\abse$ activation function, a quadrature of order $2$ with $3$ pieces involves $57$ points on average and leads to a final error of $6.25 \times 10^{-3}$ while all of the \ac{mc} settings bring to higher errors. We notice that using a higher-order quadrature or more pieces in the \ac{cpwl} approximation of the activation function reduces the error more consistently than increasing the number of \ac{mc} points, even though it is not systematic in either case.

To illustrate the learning phase, we plot the learning curve and final pointwise error in two different settings in \fig{1_l2rel} and \fig{2_l2rel}. We remark that the levels of noise in the two scenarios are very different. In the first case, the convergence with \ac{aq} is less perturbed by noise than that with \ac{mc}, and the error seems to decay faster with \ac{aq} than with \ac{mc}. In the second case, the training suffers from high levels of noise for all integration settings, and this is a representative example of most of the networks solving the strong Poisson problem.

To summarise, the strong Poisson problem brings about high levels of noise during the training phase for both integration methods. Still, \ac{aq} can reach lower errors than \ac{mc} and reduce this noise to a certain extent while using fewer integration points.

\begin{table}
    \centering
    \resizebox{0.8\linewidth}{!}{%
        \begin{tabular}{ccccccccc}
            \toprule
            \multirow{2}{*}{$\eta$}     &                          & \multirow{2}{*}{$N_\Omega$ ($P$, $O$)} & \multicolumn{6}{c}{$\nu$}                                                                                                                         \\
            \cmidrule(l){4-9}
                                        &                          &                                        & $1$                       & $10$                  & $100$                 & $500$                 & $1000$                & $5000$                \\
            \midrule
            \multirow{12}{*}{$10^{-2}$} & \multirow{3}{*}{\ac{mc}} & $50$                                   & $2.65 \times 10^{-2}$     & $3.50 \times 10^{-2}$ & $6.81 \times 10^{-2}$ & $4.01 \times 10^{-1}$ & $1.76 \times 10^{-1}$ & $5.79 \times 10^{-2}$ \\
                                        &                          & $100$                                  & $2.97 \times 10^{-2}$     & $7.99 \times 10^{-2}$ & $3.18 \times 10^{-2}$ & $1.48 \times 10^{-2}$ & $4.24 \times 10^{-2}$ & $2.51 \times 10^{-1}$ \\
            \cmidrule(l){3-9}
                                        & \multirow{5}{*}{\ac{aq}} & $25$ ($2$, $2$)                        & $1.11 \times 10^{-1}$     & $1.05 \times 10^{-1}$ & $3.00 \times 10^{-2}$ & $2.39 \times 10^{-2}$ & $4.47 \times 10^{-2}$ & $1.97 \times 10^{-0}$ \\
                                        &                          & $42$ ($2$, $5$)                        & $5.13 \times 10^{-2}$     & $3.18 \times 10^{-2}$ & $9.43 \times 10^{-2}$ & $6.39 \times 10^{-2}$ & $1.52 \times 10^{-2}$ & $1.62 \times 10^{-0}$ \\
                                        &                          & $75$ ($2$, $10$)                       & $1.84 \times 10^{-3}$     & $2.26 \times 10^{-2}$ & $1.22 \times 10^{-2}$ & $3.83 \times 10^{-2}$ & $7.25 \times 10^{-3}$ & $4.32 \times 10^{-1}$ \\
                                        &                          & $57$ ($3$, $2$)                        & $2.55 \times 10^{-2}$     & $6.25 \times 10^{-3}$ & $1.70 \times 10^{-1}$ & $2.60 \times 10^{-1}$ & $2.46 \times 10^{-1}$ & $3.04 \times 10^{-2}$ \\
                                        &                          & $84$ ($3$, $5$)                        & $1.08 \times 10^{-2}$     & $3.81 \times 10^{-2}$ & $1.32 \times 10^{-2}$ & $1.76 \times 10^{-3}$ & $3.76 \times 10^{-2}$ & $2.65 \times 10^{-2}$ \\
                                        &                          & $86$ ($5$, $2$)                        & $7.56 \times 10^{-2}$     & $6.18 \times 10^{-3}$ & $6.66 \times 10^{-3}$ & $3.58 \times 10^{-2}$ & $1.04 \times 10^{-2}$ & $1.09 \times 10^{-0}$ \\
            \midrule
            \multirow{12}{*}{$10^{-3}$} & \multirow{3}{*}{\ac{mc}} & $50$                                   & $4.67 \times 10^{-2}$     & $1.15 \times 10^{-1}$ & $1.10 \times 10^{-1}$ & $1.06 \times 10^{-1}$ & $4.40 \times 10^{-1}$ & $5.77 \times 10^{-1}$ \\
                                        &                          & $100$                                  & $4.29 \times 10^{-2}$     & $1.17 \times 10^{-1}$ & $4.03 \times 10^{-2}$ & $3.18 \times 10^{-2}$ & $6.50 \times 10^{-3}$ & $1.44 \times 10^{-1}$ \\
            \cmidrule(l){3-9}
                                        & \multirow{5}{*}{\ac{aq}} & $28$ ($2$, $2$)                        & $1.91 \times 10^{-2}$     & $4.64 \times 10^{-2}$ & $6.49 \times 10^{-2}$ & $4.35 \times 10^{-2}$ & $7.09 \times 10^{-2}$ & $2.41 \times 10^{-0}$ \\
                                        &                          & $44$ ($2$, $5$)                        & $1.80 \times 10^{-2}$     & $6.67 \times 10^{-3}$ & $1.88 \times 10^{-2}$ & $2.49 \times 10^{-2}$ & $2.81 \times 10^{-2}$ & $1.85 \times 10^{-0}$ \\
                                        &                          & $78$ ($2$, $10$)                       & $1.01 \times 10^{-2}$     & $6.75 \times 10^{-3}$ & $2.23 \times 10^{-2}$ & $5.03 \times 10^{-3}$ & $5.74 \times 10^{-3}$ & $1.26 \times 10^{-0}$ \\
                                        &                          & $59$ ($3$, $2$)                        & $3.87 \times 10^{-2}$     & $3.84 \times 10^{-2}$ & $1.32 \times 10^{-2}$ & $7.94 \times 10^{-2}$ & $1.71 \times 10^{-1}$ & $1.33 \times 10^{-2}$ \\
                                        &                          & $86$ ($3$, $5$)                        & $1.40 \times 10^{-2}$     & $3.87 \times 10^{-2}$ & $1.61 \times 10^{-2}$ & $9.66 \times 10^{-3}$ & $6.76 \times 10^{-3}$ & $7.50 \times 10^{-3}$ \\
                                        &                          & $89$ ($5$, $2$)                        & $6.32 \times 10^{-3}$     & $1.63 \times 10^{-2}$ & $3.50 \times 10^{-2}$ & $6.45 \times 10^{-3}$ & $3.03 \times 10^{-2}$ & $1.38 \times 10^{-2}$ \\
            \bottomrule
        \end{tabular}
    }
    \caption{Comparison of the relative $L^2$ norm for the strong Poisson problem with the $\abse$ activation, depending on learning rate, resampling frequency, and integration method.}
    \label{tab:1D_pois_abs}
\end{table}

\begin{table}
    \centering
    \resizebox{0.8\linewidth}{!}{%
        \begin{tabular}{ccccccccc}
            \toprule
            \multirow{2}{*}{$\eta$}     &                          & \multirow{2}{*}{$N_\Omega$ ($P$, $O$)} & \multicolumn{6}{c}{$\nu$}                                                                                                                         \\
            \cmidrule(l){4-9}
                                        &                          &                                        & $1$                       & $10$                  & $100$                 & $500$                 & $1000$                & $5000$                \\
            \midrule
            \multirow{12}{*}{$10^{-2}$} & \multirow{3}{*}{\ac{mc}} & $50$                                   & $2.62 \times 10^{-2}$     & $2.81 \times 10^{-2}$ & $7.34 \times 10^{-3}$ & $5.55 \times 10^{-2}$ & $2.75 \times 10^{-2}$ & $4.61 \times 10^{-1}$ \\
                                        &                          & $100$                                  & $2.96 \times 10^{-2}$     & $1.17 \times 10^{-2}$ & $2.81 \times 10^{-2}$ & $1.72 \times 10^{-3}$ & $9.02 \times 10^{-3}$ & $4.71 \times 10^{-2}$ \\
            \cmidrule(l){3-9}
                                        & \multirow{5}{*}{\ac{aq}} & $38$ ($3$, $2$)                        & $2.86 \times 10^{-2}$     & $6.26 \times 10^{-3}$ & $5.95 \times 10^{-2}$ & $1.83 \times 10^{-2}$ & $4.11 \times 10^{-2}$ & $7.62 \times 10^{-0}$ \\
                                        &                          & $57$ ($3$, $5$)                        & $3.13 \times 10^{-3}$     & $8.46 \times 10^{-3}$ & $2.65 \times 10^{-3}$ & $2.19 \times 10^{-3}$ & $1.37 \times 10^{-3}$ & $1.54 \times 10^{-0}$ \\
                                        &                          & $118$ ($3$, $10$)                      & $1.73 \times 10^{-3}$     & $1.54 \times 10^{-4}$ & $1.74 \times 10^{-3}$ & $1.41 \times 10^{-2}$ & $6.11 \times 10^{-2}$ & $2.92 \times 10^{-0}$ \\
                                        &                          & $76$ ($5$, $2$)                        & $8.55 \times 10^{-4}$     & $2.06 \times 10^{-3}$ & $7.29 \times 10^{-3}$ & $1.72 \times 10^{-2}$ & $5.46 \times 10^{-3}$ & $5.06 \times 10^{-1}$ \\
                                        &                          & $118$ ($5$, $5$)                       & $4.62 \times 10^{-3}$     & $2.83 \times 10^{-3}$ & $4.58 \times 10^{-4}$ & $1.77 \times 10^{-3}$ & $4.02 \times 10^{-3}$ & $3.39 \times 10^{-1}$ \\
                                        &                          & $112$ ($7$, $2$)                       & $6.35 \times 10^{-3}$     & $2.97 \times 10^{-4}$ & $7.16 \times 10^{-3}$ & $4.23 \times 10^{-3}$ & $2.47 \times 10^{-2}$ & $3.97 \times 10^{-1}$ \\
            \midrule
            \multirow{12}{*}{$10^{-3}$} & \multirow{3}{*}{\ac{mc}} & $50$                                   & $8.51 \times 10^{-2}$     & $3.04 \times 10^{-2}$ & $9.91 \times 10^{-3}$ & $1.36 \times 10^{-1}$ & $2.62 \times 10^{-2}$ & $8.61 \times 10^{-2}$ \\
                                        &                          & $100$                                  & $5.87 \times 10^{-2}$     & $6.39 \times 10^{-2}$ & $5.65 \times 10^{-2}$ & $7.68 \times 10^{-3}$ & $1.79 \times 10^{-2}$ & $9.28 \times 10^{-2}$ \\
            \cmidrule(l){3-9}
                                        & \multirow{5}{*}{\ac{aq}} & $31$ ($3$, $2$)                        & $1.73 \times 10^{-0}$     & $8.18 \times 10^{-2}$ & $1.49 \times 10^{-2}$ & $1.64 \times 10^{-1}$ & $4.34 \times 10^{-0}$ & $7.94 \times 10^{-0}$ \\
                                        &                          & $49$ ($3$, $5$)                        & $3.42 \times 10^{-2}$     & $2.51 \times 10^{-2}$ & $2.82 \times 10^{-2}$ & $2.70 \times 10^{-2}$ & $2.23 \times 10^{-2}$ & $1.66 \times 10^{-0}$ \\
                                        &                          & $113$ ($3$, $10$)                      & $5.41 \times 10^{-3}$     & $1.57 \times 10^{-2}$ & $1.07 \times 10^{-2}$ & $5.37 \times 10^{-2}$ & $7.75 \times 10^{-3}$ & $1.56 \times 10^{-0}$ \\
                                        &                          & $75$ ($5$, $2$)                        & $2.83 \times 10^{-3}$     & $1.01 \times 10^{-2}$ & $8.73 \times 10^{-3}$ & $1.07 \times 10^{-3}$ & $4.03 \times 10^{-2}$ & $1.19 \times 10^{-0}$ \\
                                        &                          & $107$ ($5$, $5$)                       & $8.00 \times 10^{-3}$     & $9.14 \times 10^{-3}$ & $8.99 \times 10^{-3}$ & $4.01 \times 10^{-2}$ & $2.08 \times 10^{-3}$ & $1.04 \times 10^{-0}$ \\
                                        &                          & $107$ ($7$, $2$)                       & $1.09 \times 10^{-2}$     & $6.94 \times 10^{-3}$ & $4.83 \times 10^{-3}$ & $8.86 \times 10^{-3}$ & $8.29 \times 10^{-3}$ & $3.94 \times 10^{-1}$ \\
            \bottomrule
        \end{tabular}
    }
    \caption{Comparison of the relative $L^2$ norm for the strong Poisson problem with the $\tanh$ activation, depending on learning rate, resampling frequency, and integration method.}
    \label{tab:1D_pois_tanh}
\end{table}

\begin{figure}
    \centering
    \subfloat[Learning curve\label{fig:1_l2}]{
        \includegraphics[width=0.45\linewidth]{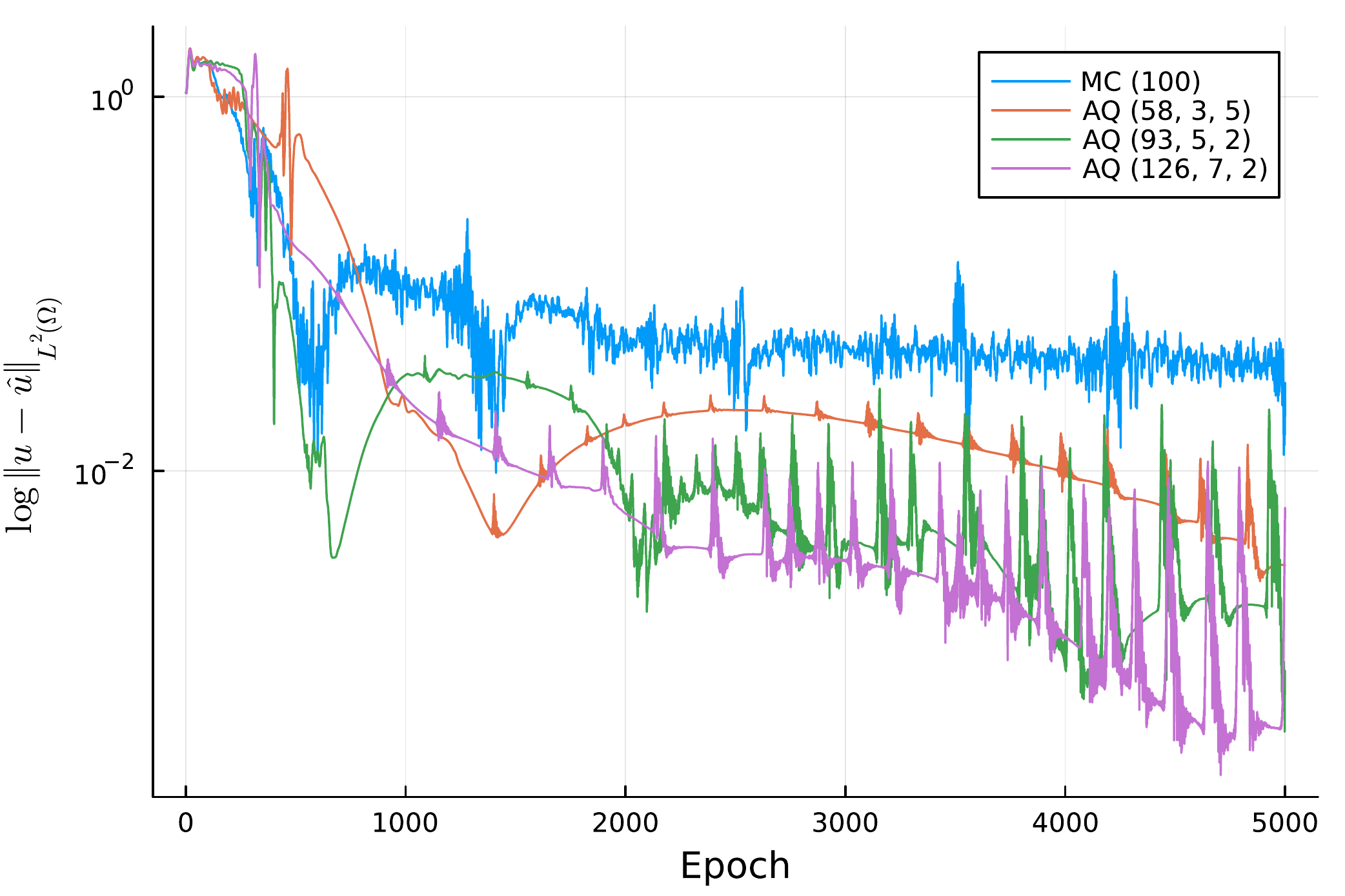}
    }
    \hfill
    \subfloat[Pointwise error\label{fig:1_rel}]{
        \includegraphics[width=0.45\linewidth]{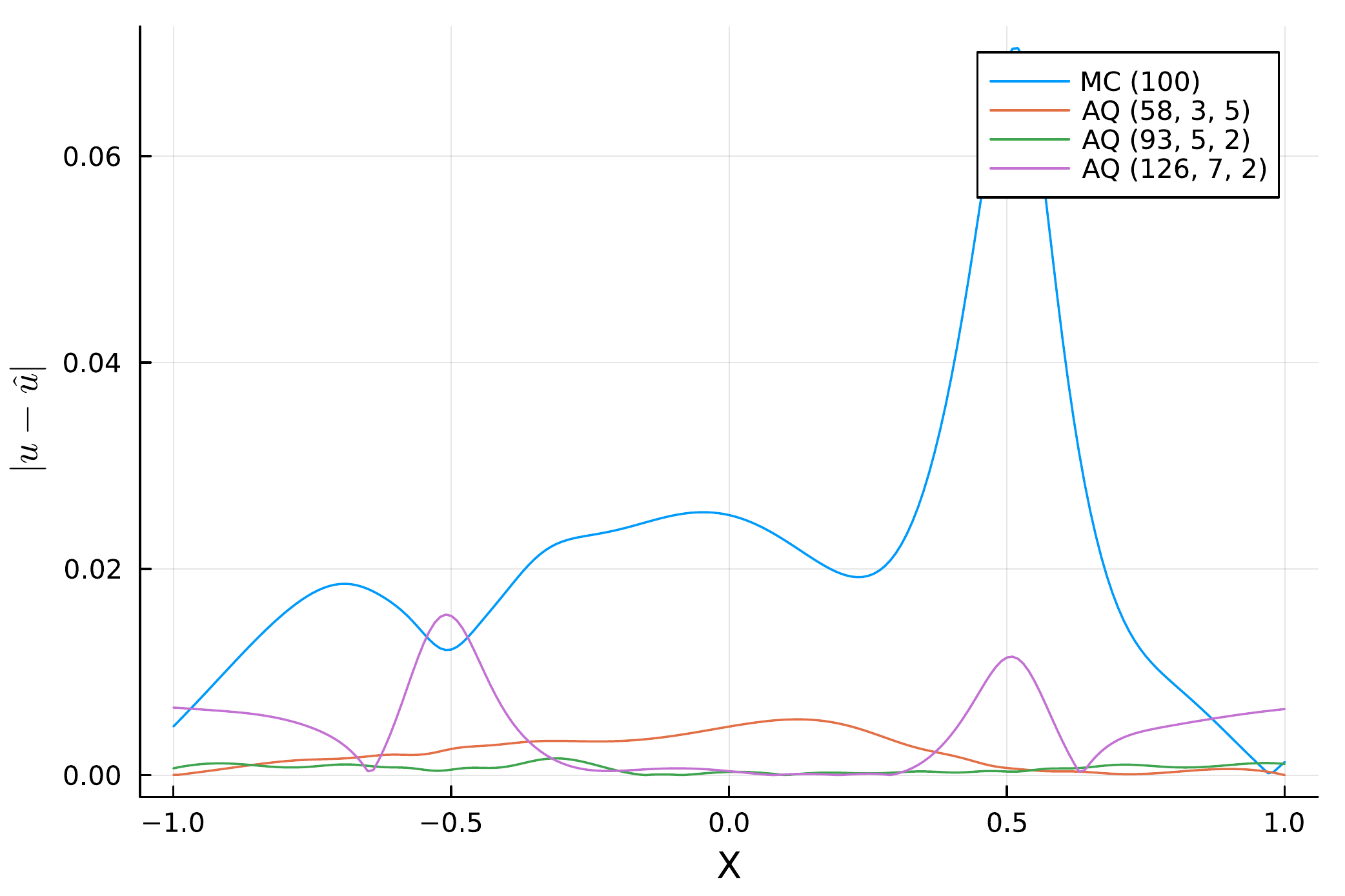}
    }
    \caption{Learning curve \protect\subref{fig:1_l2} and pointwise error \protect\subref{fig:1_rel} for the strong Poisson problem with the $\tanh$ activation. The learning rate is set to $10^{-2}$ and the points are resampled every $1$ epoch.}
    \label{fig:1_l2rel}
\end{figure}

\begin{figure}
    \centering
    \subfloat[Learning curve\label{fig:2_l2}]{
        \includegraphics[width=0.45\linewidth]{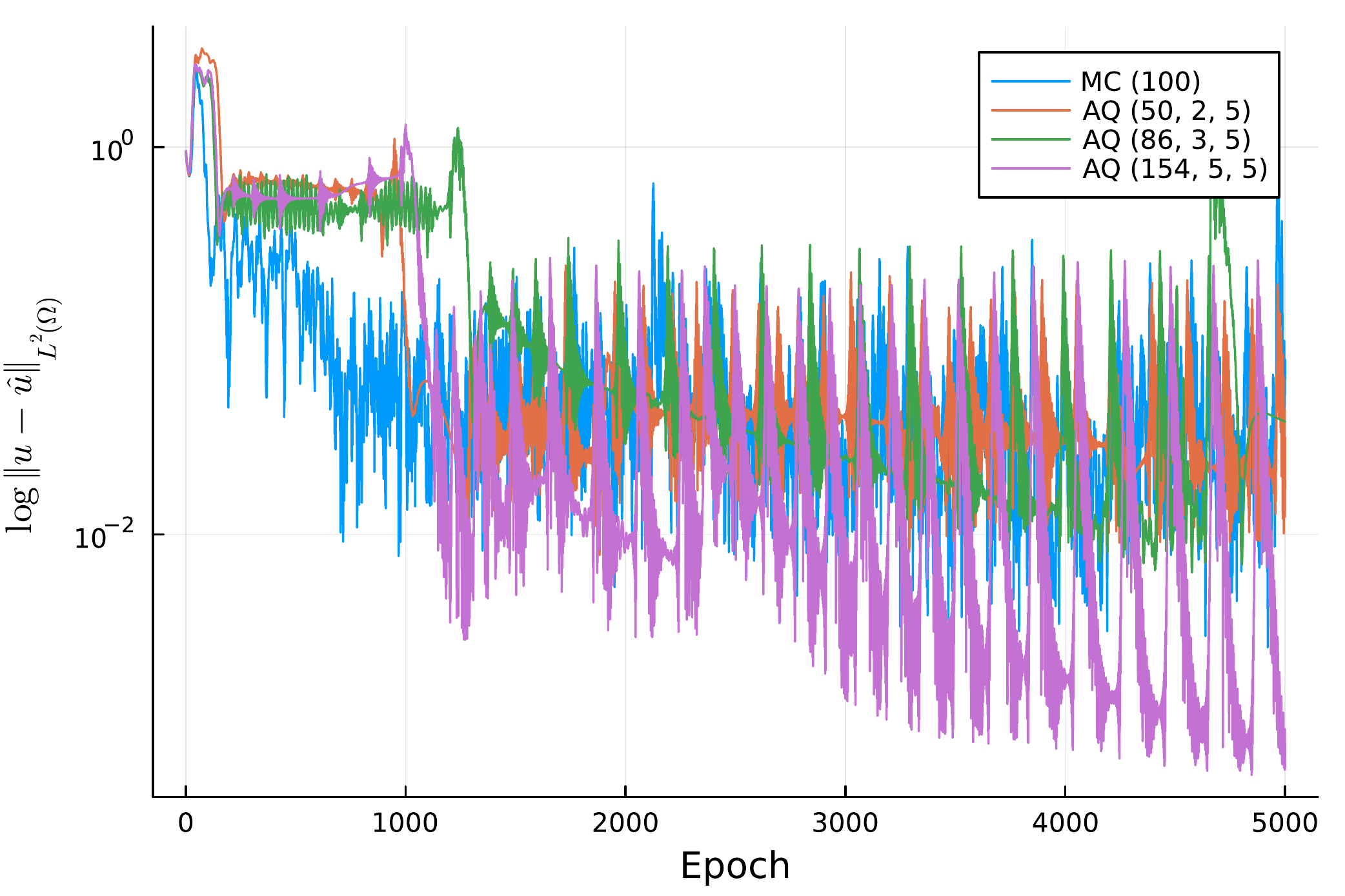}
    }
    \hfill
    \subfloat[Pointwise error\label{fig:2_rel}]{
        \includegraphics[width=0.45\linewidth]{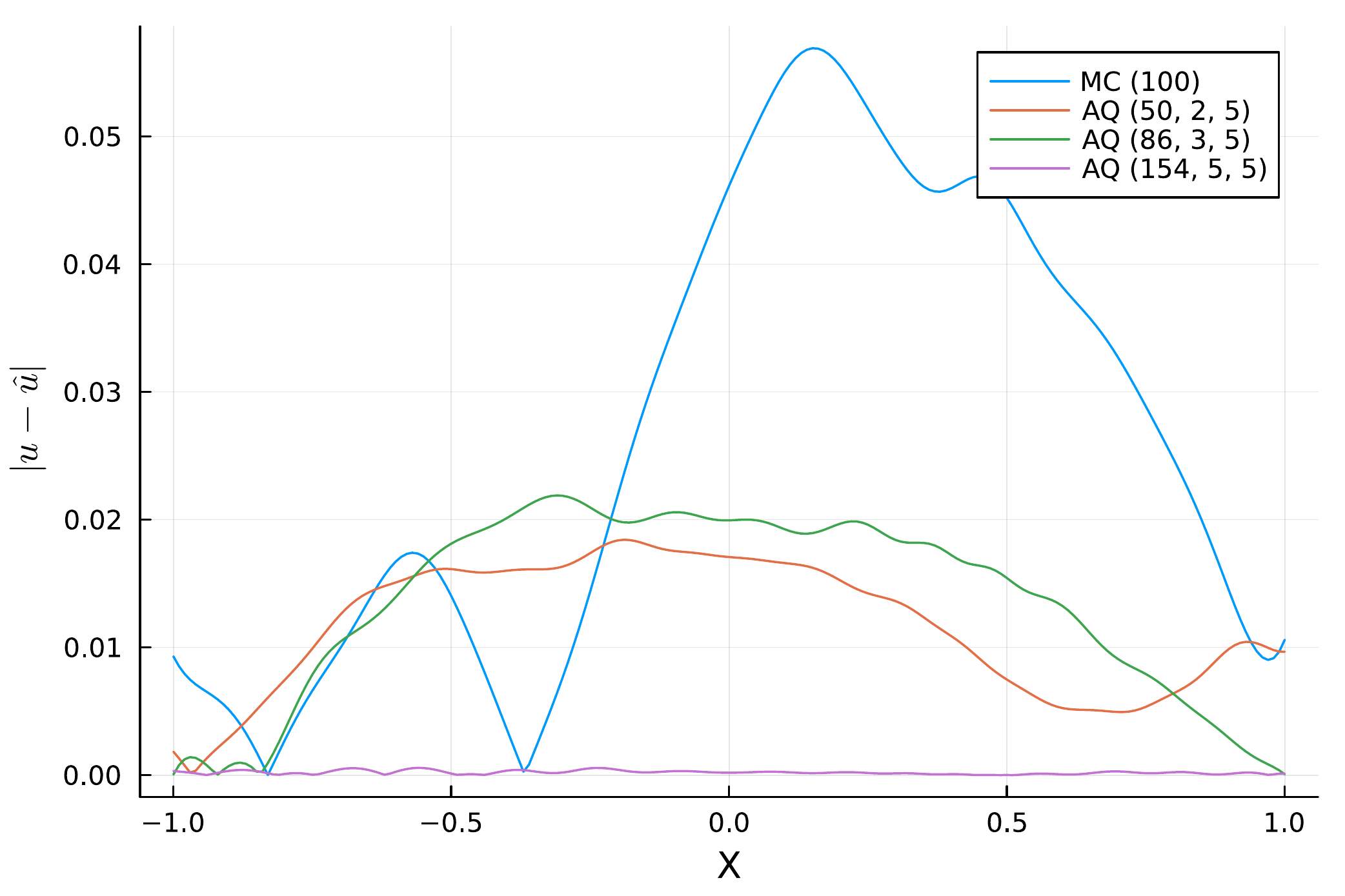}
    }
    \caption{Learning curve \protect\subref{fig:2_l2} and pointwise error \protect\subref{fig:2_rel} for the strong Poisson problem with the $\abse$ activation. The learning rate is $\eta = 10^{-2}$ and the points are resampled every $10$ epochs.}
    \label{fig:2_l2rel}
\end{figure}

%%%
\subsubsection{Weak Poisson in 1D}

We report the relative $L^2$ norm after training the network for the two one-dimensional weak Poisson problems in \tab{1D_poiw_abs} and \tab{1D_poiw_tanh}. We remark that the network does not converge to the expected solution when the integration points are not resampled (frequency of $5000$). Using more integration points alleviates this phenomenon, but we disregard this border case in the rest of this paragraph.

We observe that in many cases, and especially when the learning rate is set to $10^{-2}$, the error reached with \ac{aq} is one to two orders of magnitude lower than with \ac{mc}, even when \ac{aq} uses much fewer points than \ac{mc}. Provided that the number of integration points is sufficiently large, the final $L^2$ norm varies very little with the sampling frequency when the network is trained with \ac{aq}. When relying on \ac{mc}, the variance of the final error is much higher and the dependence between the variance and the number of integration points is unclear.

When the sampling frequency is high enough (every $100$ or fewer epochs), we observe that the number of integration points can very often be reduced while keeping the same levels of error. For instance, in the problem with the $\tanh$ function with a learning rate of $10^{-2}$ and a resampling frequency of $10$, shifting from $7$ pieces to $5$ or $3$ pieces while keeping order $2$ does not deteriorate the performance of \ac{aq}. However, reducing the number of \ac{mc} integration points increases the error in most cases.

To further compare the training phase with \ac{aq} and \ac{mc}, we plot the learning curve and pointwise error for the $\tanh$ activation in \fig{3_l2rel} and for the $\abse$ activation in \fig{4_l2rel}. In both cases, we remark that the convergence of the network is significantly smoother with \ac{aq} than with \ac{mc}, even with few integration points. Moreover, it is clear that the $L^2$ norm decreases in the case of \ac{aq}, whereas it seems to plateau with \ac{mc}. Finally, it is interesting to note that the learning curves follow the same trend at the beginning of the training. After a few hundred epochs, the learning curves corresponding to \ac{aq} keep on decreasing while the ones corresponding to \ac{mc} break away from this trend and start to suffer from high levels of noise.

In conclusion, solving the weak Poisson problem with \ac{aq} is computationally cheaper and shows a considerably faster and smoother convergence than with \ac{mc}.

\begin{table}
    \centering
    \resizebox{0.8\linewidth}{!}{%
        \begin{tabular}{ccccccccc}
            \toprule
            \multirow{2}{*}{$\eta$}     &                          & \multirow{2}{*}{$N_\Omega$ ($P$, $O$)} & \multicolumn{6}{c}{$\nu$}                                                                                                                         \\
            \cmidrule(l){4-9}
                                        &                          &                                        & $1$                       & $10$                  & $100$                 & $500$                 & $1000$                & $5000$                \\
            \midrule
            \multirow{12}{*}{$10^{-2}$} & \multirow{3}{*}{\ac{mc}} & $50$                                   & $4.82 \times 10^{-2}$     & $1.69 \times 10^{-1}$ & $1.41 \times 10^{-0}$ & $5.08 \times 10^{-0}$ & $8.96 \times 10^{-0}$ & $1.82 \times 10^{+3}$ \\
                                        &                          & $100$                                  & $4.29 \times 10^{-1}$     & $6.39 \times 10^{-1}$ & $1.68 \times 10^{-0}$ & $1.40 \times 10^{-0}$ & $2.11 \times 10^{+1}$ & $7.26 \times 10^{+3}$ \\
            \cmidrule(l){3-9}
                                        & \multirow{5}{*}{\ac{aq}} & $24$ ($2$, $2$)                        & $2.87 \times 10^{-2}$     & $2.05 \times 10^{-1}$ & $8.80 \times 10^{-2}$ & $1.14 \times 10^{-0}$ & $3.54 \times 10^{+2}$ & $4.43 \times 10^{+5}$ \\
                                        &                          & $43$ ($2$, $5$)                        & $1.04 \times 10^{-2}$     & $1.00 \times 10^{-2}$ & $4.90 \times 10^{-2}$ & $1.81 \times 10^{-0}$ & $2.04 \times 10^{+2}$ & $5.28 \times 10^{+5}$ \\
                                        &                          & $88$ ($2$, $10$)                       & $8.51 \times 10^{-3}$     & $8.24 \times 10^{-3}$ & $6.89 \times 10^{-2}$ & $2.05 \times 10^{-1}$ & $1.10 \times 10^{-0}$ & $3.11 \times 10^{+3}$ \\
                                        &                          & $61$ ($3$, $2$)                        & $8.12 \times 10^{-3}$     & $1.37 \times 10^{-2}$ & $1.59 \times 10^{-2}$ & $5.13 \times 10^{-2}$ & $4.83 \times 10^{-0}$ & $5.49 \times 10^{+3}$ \\
                                        &                          & $90$ ($3$, $5$)                        & $8.74 \times 10^{-3}$     & $1.03 \times 10^{-2}$ & $1.02 \times 10^{-2}$ & $8.71 \times 10^{-3}$ & $7.79 \times 10^{-3}$ & $3.47 \times 10^{+2}$ \\
                                        &                          & $91$ ($5$, $2$)                        & $9.67 \times 10^{-3}$     & $1.30 \times 10^{-2}$ & $6.35 \times 10^{-3}$ & $7.78 \times 10^{-3}$ & $1.49 \times 10^{-2}$ & $9.35 \times 10^{+1}$ \\
            \midrule
            \multirow{12}{*}{$10^{-3}$} & \multirow{3}{*}{\ac{mc}} & $50$                                   & $1.91 \times 10^{-1}$     & $1.70 \times 10^{-1}$ & $9.88 \times 10^{-1}$ & $4.93 \times 10^{-0}$ & $4.97 \times 10^{-0}$ & $1.00 \times 10^{+1}$ \\
                                        &                          & $100$                                  & $2.73 \times 10^{-1}$     & $2.32 \times 10^{-1}$ & $1.37 \times 10^{-0}$ & $8.35 \times 10^{-1}$ & $6.46 \times 10^{-1}$ & $6.55 \times 10^{+1}$ \\
            \cmidrule(l){3-9}
                                        & \multirow{5}{*}{\ac{aq}} & $26$ ($2$, $2$)                        & $5.67 \times 10^{-2}$     & $3.50 \times 10^{-2}$ & $5.15 \times 10^{-2}$ & $1.26 \times 10^{-0}$ & $3.02 \times 10^{-0}$ & $2.32 \times 10^{+3}$ \\
                                        &                          & $46$ ($2$, $5$)                        & $8.66 \times 10^{-3}$     & $7.24 \times 10^{-3}$ & $1.98 \times 10^{-2}$ & $4.13 \times 10^{-3}$ & $1.70 \times 10^{-1}$ & $1.78 \times 10^{+3}$ \\
                                        &                          & $93$ ($2$, $10$)                       & $5.87 \times 10^{-3}$     & $1.04 \times 10^{-2}$ & $5.71 \times 10^{-3}$ & $5.57 \times 10^{-3}$ & $5.17 \times 10^{-3}$ & $5.28 \times 10^{+2}$ \\
                                        &                          & $66$ ($3$, $2$)                        & $2.30 \times 10^{-2}$     & $7.99 \times 10^{-3}$ & $1.42 \times 10^{-2}$ & $1.19 \times 10^{-2}$ & $3.26 \times 10^{-2}$ & $3.90 \times 10^{-0}$ \\
                                        &                          & $96$ ($3$, $5$)                        & $1.97 \times 10^{-2}$     & $1.79 \times 10^{-0}$ & $3.78 \times 10^{-2}$ & $5.46 \times 10^{-3}$ & $5.53 \times 10^{-3}$ & $6.42 \times 10^{-3}$ \\
                                        &                          & $90$ ($5$, $2$)                        & $1.78 \times 10^{-2}$     & $6.80 \times 10^{-3}$ & $6.61 \times 10^{-3}$ & $1.16 \times 10^{-2}$ & $9.72 \times 10^{-3}$ & $1.83 \times 10^{-0}$ \\
            \bottomrule
        \end{tabular}
    }
    \caption{Comparison of the relative $L^2$ norm for the weak Poisson problem with the $\abse$ activation, depending on learning rate, resampling frequency, and integration method.}
    \label{tab:1D_poiw_abs}
\end{table}

\begin{table}
    \centering
    \resizebox{0.8\linewidth}{!}{%
        \begin{tabular}{ccccccccc}
            \toprule
            \multirow{2}{*}{$\eta$}     &                          & \multirow{2}{*}{$N_\Omega$ ($P$, $O$)} & \multicolumn{6}{c}{$\nu$}                                                                                                                         \\
            \cmidrule(l){4-9}
                                        &                          &                                        & $1$                       & $10$                  & $100$                 & $500$                 & $1000$                & $5000$                \\
            \midrule
            \multirow{12}{*}{$10^{-2}$} & \multirow{3}{*}{\ac{mc}} & $50$                                   & $7.67 \times 10^{-2}$     & $4.99 \times 10^{-1}$ & $4.64 \times 10^{-1}$ & $2.86 \times 10^{-0}$ & $8.05 \times 10^{-0}$ & $3.45 \times 10^{+2}$ \\
                                        &                          & $100$                                  & $1.75 \times 10^{-1}$     & $4.92 \times 10^{-1}$ & $9.31 \times 10^{-1}$ & $1.55 \times 10^{-0}$ & $1.64 \times 10^{-0}$ & $4.78 \times 10^{-0}$ \\
            \cmidrule(l){3-9}
                                        & \multirow{5}{*}{\ac{aq}} & $40$ ($3$, $2$)                        & $3.87 \times 10^{-2}$     & $1.49 \times 10^{-2}$ & $1.14 \times 10^{-1}$ & $1.34 \times 10^{-0}$ & $1.19 \times 10^{+1}$ & $4.21 \times 10^{+2}$ \\
                                        &                          & $46$ ($3$, $5$)                        & $8.40 \times 10^{-1}$     & $3.06 \times 10^{-2}$ & $3.47 \times 10^{-2}$ & $2.19 \times 10^{-2}$ & $1.77 \times 10^{-0}$ & $6.64 \times 10^{+1}$ \\
                                        &                          & $110$ ($3$, $10$)                      & $7.89 \times 10^{-3}$     & $4.56 \times 10^{-3}$ & $8.17 \times 10^{-3}$ & $5.09 \times 10^{-2}$ & $1.45 \times 10^{-0}$ & $3.92 \times 10^{+2}$ \\
                                        &                          & $84$ ($5$, $2$)                        & $7.41 \times 10^{-3}$     & $1.36 \times 10^{-2}$ & $3.41 \times 10^{-3}$ & $4.18 \times 10^{-2}$ & $4.27 \times 10^{-1}$ & $4.34 \times 10^{+2}$ \\
                                        &                          & $108$ ($5$, $5$)                       & $3.73 \times 10^{-3}$     & $4.77 \times 10^{-3}$ & $5.53 \times 10^{-3}$ & $9.31 \times 10^{-3}$ & $5.17 \times 10^{-1}$ & $3.05 \times 10^{+2}$ \\
                                        &                          & $98$ ($7$, $2$)                        & $3.75 \times 10^{-3}$     & $9.94 \times 10^{-3}$ & $3.71 \times 10^{-3}$ & $5.95 \times 10^{-2}$ & $1.67 \times 10^{-1}$ & $1.66 \times 10^{+2}$ \\
            \midrule
            \multirow{12}{*}{$10^{-3}$} & \multirow{3}{*}{\ac{mc}} & $50$                                   & $2.52 \times 10^{-1}$     & $2.54 \times 10^{-1}$ & $9.53 \times 10^{-1}$ & $2.76 \times 10^{-0}$ & $3.42 \times 10^{-0}$ & $1.47 \times 10^{+1}$ \\
                                        &                          & $100$                                  & $3.16 \times 10^{-1}$     & $3.77 \times 10^{-1}$ & $1.34 \times 10^{-0}$ & $2.56 \times 10^{-0}$ & $1.77 \times 10^{-0}$ & $1.87 \times 10^{-0}$ \\
            \cmidrule(l){3-9}
                                        & \multirow{5}{*}{\ac{aq}} & $21$ ($3$, $2$)                        & $1.59 \times 10^{-0}$     & $1.59 \times 10^{-0}$ & $1.21 \times 10^{-0}$ & $5.16 \times 10^{-1}$ & $5.55 \times 10^{-0}$ & $6.10 \times 10^{+1}$ \\
                                        &                          & $26$ ($3$, $5$)                        & $5.44 \times 10^{-1}$     & $1.66 \times 10^{-0}$ & $3.47 \times 10^{-1}$ & $3.43 \times 10^{-2}$ & $1.54 \times 10^{-1}$ & $1.86 \times 10^{-0}$ \\
                                        &                          & $57$ ($3$, $10$)                       & $6.52 \times 10^{-1}$     & $6.20 \times 10^{-1}$ & $5.57 \times 10^{-2}$ & $4.92 \times 10^{-2}$ & $5.12 \times 10^{-2}$ & $3.86 \times 10^{+1}$ \\
                                        &                          & $50$ ($5$, $2$)                        & $5.41 \times 10^{-2}$     & $3.90 \times 10^{-2}$ & $7.87 \times 10^{-2}$ & $5.77 \times 10^{-2}$ & $2.47 \times 10^{-2}$ & $5.02 \times 10^{+1}$ \\
                                        &                          & $71$ ($5$, $5$)                        & $2.32 \times 10^{-2}$     & $2.31 \times 10^{-2}$ & $4.37 \times 10^{-2}$ & $6.62 \times 10^{-2}$ & $5.95 \times 10^{-2}$ & $2.18 \times 10^{+1}$ \\
                                        &                          & $41$ ($7$, $2$)                        & $2.35 \times 10^{-1}$     & $1.81 \times 10^{-1}$ & $2.24 \times 10^{-1}$ & $1.08 \times 10^{-1}$ & $2.78 \times 10^{-1}$ & $1.50 \times 10^{-0}$ \\
            \bottomrule
        \end{tabular}
    }
    \caption{Comparison of the relative $L^2$ norm for the weak Poisson problem with the $\tanh$ activation, depending on learning rate, resampling frequency, and integration method.}
    \label{tab:1D_poiw_tanh}
\end{table}

\begin{figure}
    \centering
    \subfloat[Learning curve\label{fig:3_l2}]{
        \includegraphics[width=0.45\linewidth]{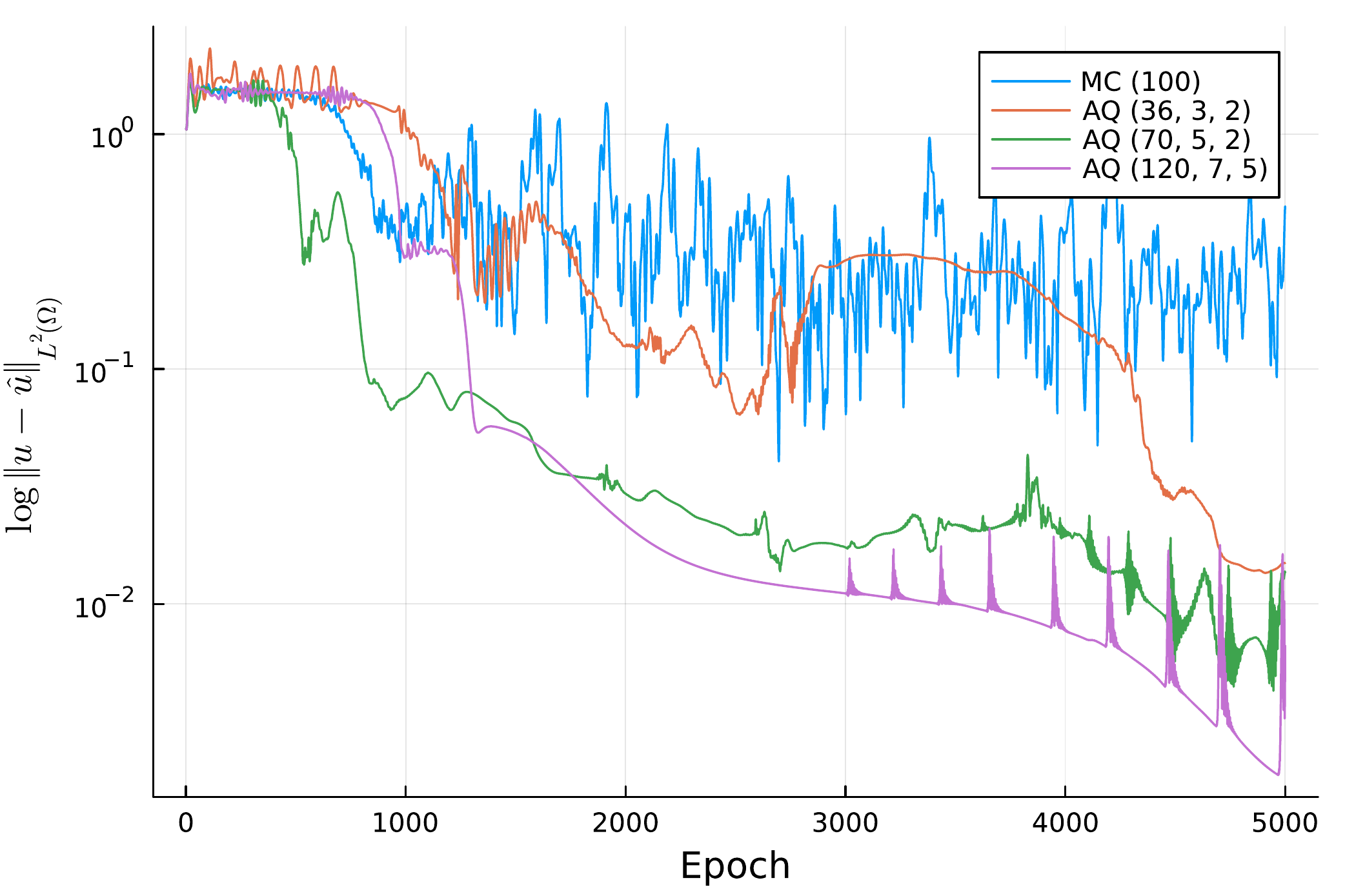}
    }
    \hfill
    \subfloat[Pointwise error\label{fig:3_rel}]{
        \includegraphics[width=0.45\linewidth]{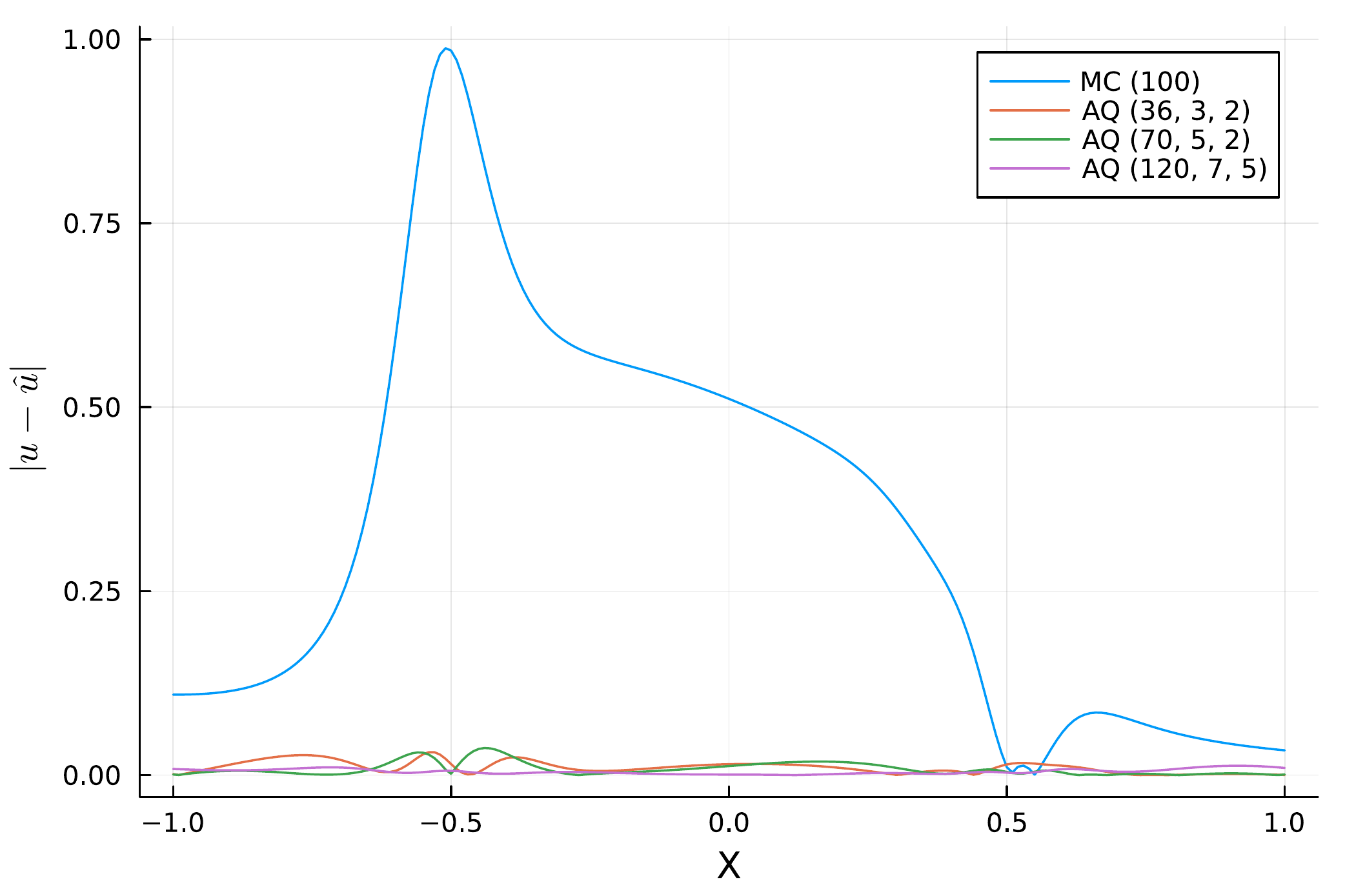}
    }
    \caption{Learning curve \protect\subref{fig:3_l2} and pointwise error \protect\subref{fig:3_rel} for the weak Poisson problem with the $\tanh$ activation. The learning rate is $\eta = 10^{-2}$ and the points are resampled every $10$ epochs.}
    \label{fig:3_l2rel}
\end{figure}

\begin{figure}
    \centering
    \subfloat[Learning curve\label{fig:4_l2}]{
        \includegraphics[width=0.45\linewidth]{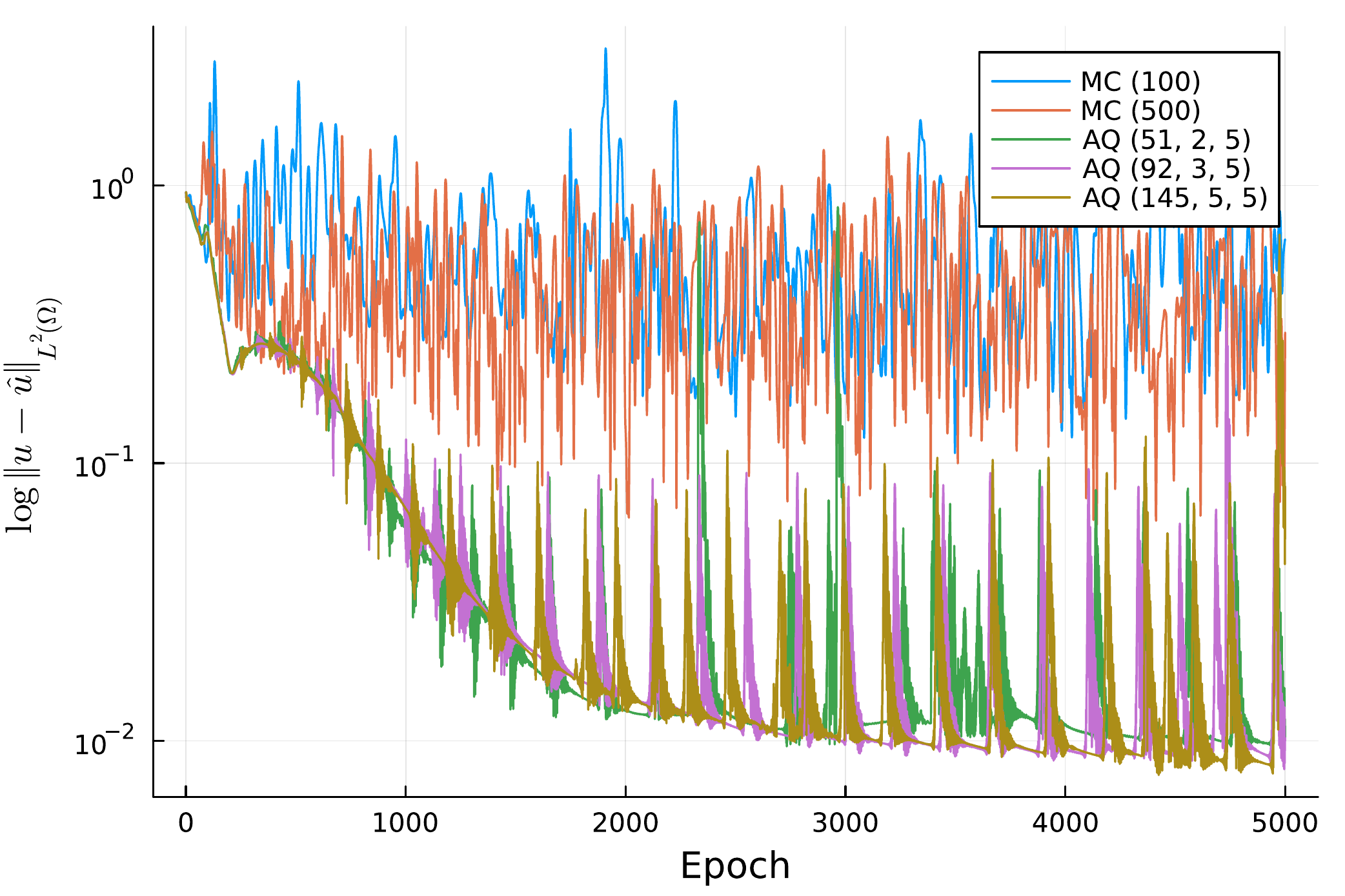}
    }
    \hfill
    \subfloat[Pointwise error\label{fig:4_rel}]{
        \includegraphics[width=0.45\linewidth]{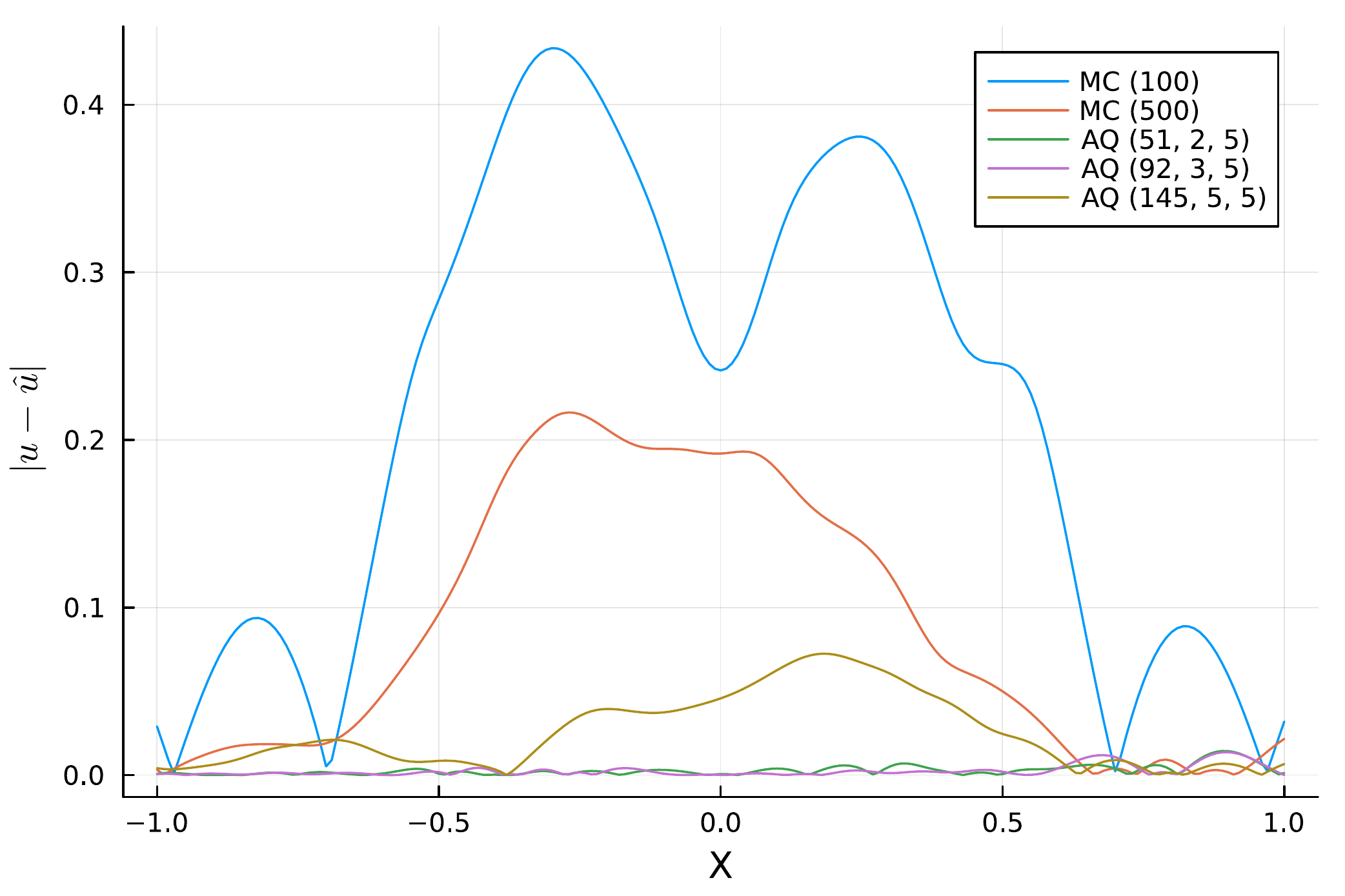}
    }
    \caption{Learning curve \protect\subref{fig:4_l2} and pointwise error \protect\subref{fig:4_rel} for the weak Poisson problem with the $\abse$ activation. The learning rate is set to $10^{-2}$ and the points are resampled every $10$ epochs.}
    \label{fig:4_l2rel}
\end{figure}

%%%
\subsubsection{Weak Poisson in 2D}

Based on the experiments in dimension one, we conclude that the integration points should be resampled around every $10$ epochs and the learning rate should be set to $10^{-2}$ to obtain optimal performances with both integration methods. We keep these hyperparameters in the rest of our experiments.

We find that \ac{aq} can reach similar or lower errors with less than half the number of points \ac{mc} needs. With the $\abse$ activation, \ac{aq} with fewer than $1000$ points and \ac{mc} with $10000$ points reach comparable performances. In the case of the $\tanh$ activation, \ac{aq} with around $1660$ points wins over \ac{mc} with $10000$ points. However, we notice that increasing the order of the quadrature or the number of pieces does not significantly reduce the relative $L^2$ norm.

\fig{5_rel} and \fig{6_rel} display the pointwise error between the ground truth and the solutions obtained with \ac{mc} and \ac{aq}. In the case of $\tanh$, we remark that \ac{mc} struggles to model the transition between the high and low regions. Similarly for $\abse$, \ac{mc} does not properly handle the region around the origin, where the solution shows larger variations. In both cases, the pointwise error for the model trained with \ac{aq} is much more homogeneous across the whole domain.

\begin{table}
    \centering
    \subfloat[$\abse$\label{tab:2D_poiw_abs}]{
        \resizebox{0.4\linewidth}{!}{%
            \begin{tabular}{ccc}
                \toprule
                Integration         & ($N_\Omega$, $N_\Gamma$) ($P$, $O$) & $L^2$ norm            \\
                \midrule
                \multirow{3}{*}{MC} & ($1000$, $100$)                     & $2.35 \times 10^{-1}$ \\
                                    & ($5000$, $500$)                     & $1.43 \times 10^{-1}$ \\
                                    & ($10000$, $1000$)                   & $1.10 \times 10^{-1}$ \\
                \cmidrule(l){2-3}
                \multirow{3}{*}{AQ} & ($927$, $85$) ($2$, $2$)            & $1.03 \times 10^{-1}$ \\
                                    & ($2681$, $144$) ($2$, $5$)          & $3.58 \times 10^{-2}$ \\
                                    & ($4041$, $160$) ($3$, $2$)          & $6.40 \times 10^{-2}$ \\
                \bottomrule
            \end{tabular}
        }
    }
    \qquad
    \subfloat[$\tanh$\label{tab:2D_poiw_tanh}]{
        \resizebox{0.4\linewidth}{!}{%
            \begin{tabular}{ccc}
                \toprule
                Integration         & ($N_\Omega$, $N_\Gamma$) ($P$, $O$) & $L^2$ norm            \\
                \midrule
                \multirow{3}{*}{MC} & ($1000$, $100$)                     & $1.33 \times 10^{-1}$ \\
                                    & ($5000$, $500$)                     & $8.51 \times 10^{-2}$ \\
                                    & ($10000$, $1000$)                   & $6.61 \times 10^{-2}$ \\
                \cmidrule(l){2-3}
                \multirow{3}{*}{AQ} & ($1660$, $111$) ($3$, $2$)          & $5.43 \times 10^{-2}$ \\
                                    & ($2964$, $154$) ($3$, $5$)          & $4.47 \times 10^{-2}$ \\
                                    & ($3465$, $175$) ($5$, $2$)          & $4.34 \times 10^{-2}$ \\
                \bottomrule
            \end{tabular}
        }
    }
    \caption{Relative $L^2$ norm for the weak Poisson problems in 2D with the $\abse$ \protect\subref{tab:2D_poiw_abs} and $\tanh$ \protect\subref{tab:2D_poiw_tanh} activations for various integration hyperparameters.}
    \label{tab:reduce_merge}
\end{table}

\begin{figure}
    \centering
    \includegraphics[width=0.8\linewidth]{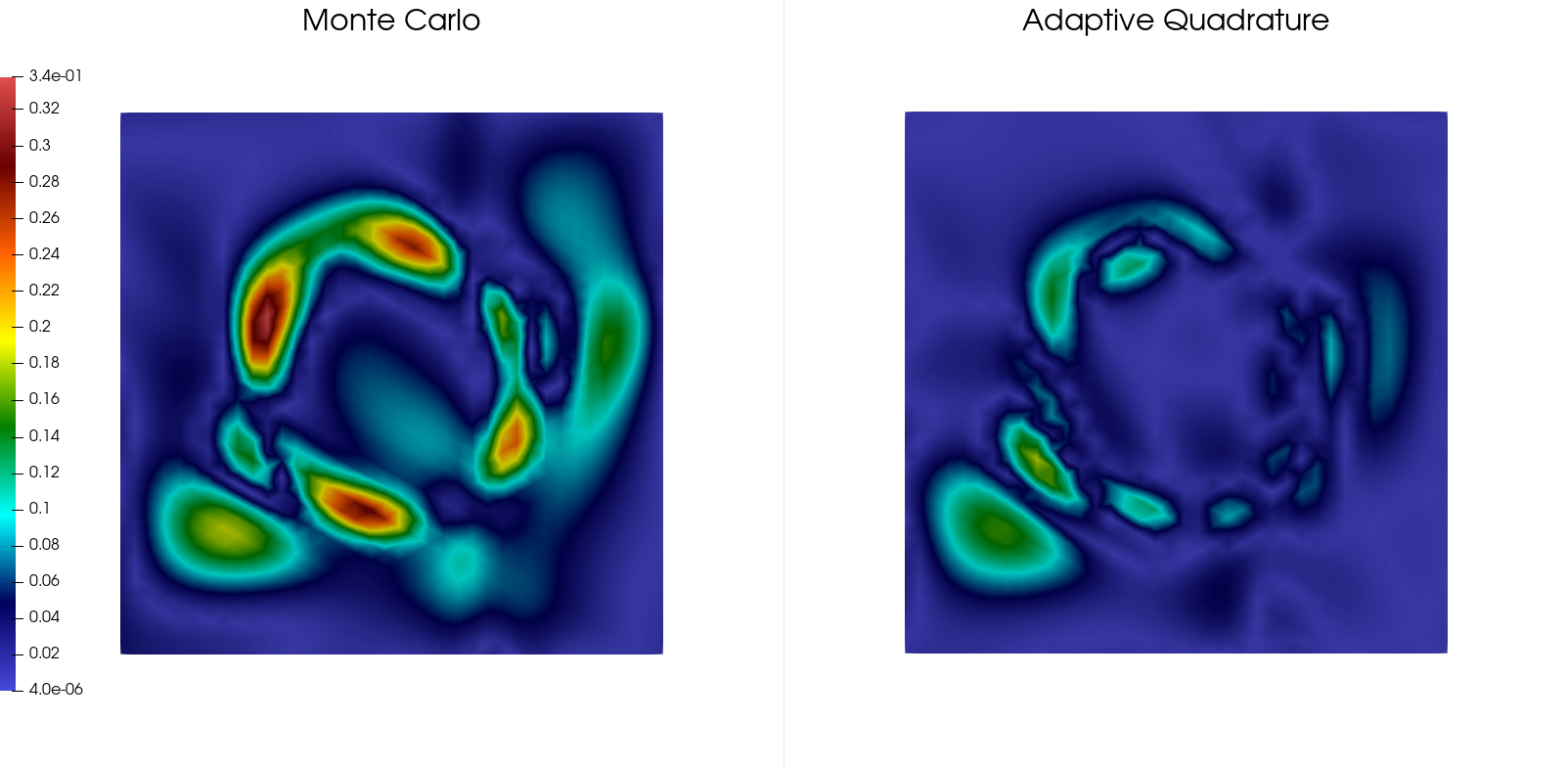}
    \caption{Comparison of the pointwise error for the weak Poisson problem with the $\tanh$ activation. The learning rate is $\eta = 10^{-2}$ and the points are resampled every $10$ epochs. \ac{mc} is trained with $5000$ points and \ac{aq} with $5$ pieces and order $3$ ($2964$ points).}
    \label{fig:5_rel}
\end{figure}

\begin{figure}
    \centering
    \includegraphics[width=0.8\linewidth]{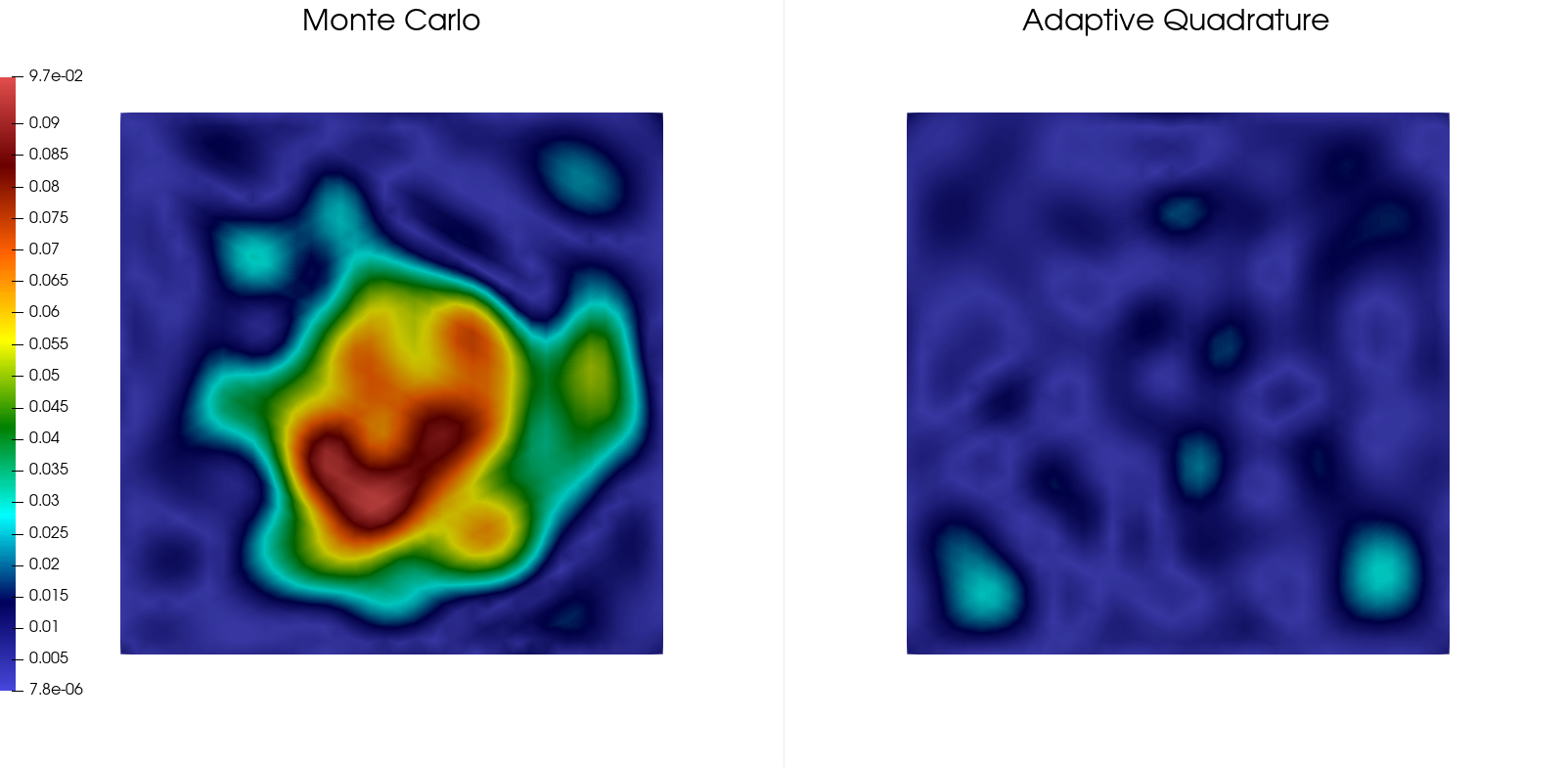}
    \caption{Comparison of the pointwise error for the weak Poisson problem with the $\abse$ activation. The learning rate is set to $10^{-2}$ and the points are resampled every $10$ epochs. \ac{mc} is trained with $5000$ points and \ac{aq} with $2$ pieces and order $5$ ($2681$ points).}
    \label{fig:6_rel}
\end{figure}

%%%
\subsection{Robustness to initialisation}
\label{subsect:initialisation}

In this second round of experiments, we study the robustness of our integration method to the initialisation of the network. We solve the weak Poisson problems in dimensions one and two. The learning rate is set to $10^{-2}$ and we resample the integration points every $10$ epochs. We report the minimum, maximum, average, and standard deviation of the final $L^2$ norm on $10$ random initialisations, as well as the average training time in \tab{init_poiw_abs} and \tab{init_poiw_tanh} for the $\abse$ and $\tanh$ activations respectively. We use the same seeds for \ac{mc} and \ac{aq} in both experiments. The number of integration points for \ac{mc} is chosen so that it is larger than the maximum number of points for the same experiments run with \ac{aq}. The exact choices of quadrature order, number of pieces, and integration points are reported in the corresponding tables.

We find that \ac{aq} has a consistently lower average $L^2$ norm than \ac{mc}. By comparing the standard deviation of the error, we infer that the solutions obtained by \ac{aq} are less dependent on the network initialisation compared to \ac{mc}. Besides, we observe that the distributions of the errors obtained with \ac{aq} and \ac{mc} do not overlap: the maximum error with \ac{aq} is lower than the minimum error with \ac{mc}. More importantly, we find that \ac{aq} enables a reduction of the number of integration points while keeping similar or shorter training times.

We conclude that our proposed integration method is more robust to initialisation than \ac{mc} and that it obtains higher accuracies than \ac{mc} with fewer integration points.

\begin{table}
    \centering
    \subfloat[$\abse$.\label{tab:init_poiw_abs}]{
        \resizebox{0.8\linewidth}{!}{%
            \begin{tabular}{cccccccc}
                \toprule
                Dim.               &         & ($N_\Omega$, $N_\Gamma$) ($P$, $O$) & Time (s) & min.                  & avg.                  & std.                  & max.                  \\
                \midrule
                \multirow{2}{*}{1} & \ac{mc} & ($100$, $2$)                        & $19.2$   & $1.98 \times 10^{-1}$ & $4.45 \times 10^{-1}$ & $2.45 \times 10^{-1}$ & $8.31 \times 10^{-1}$ \\
                                   & \ac{aq} & ($93$, $2$) ($3$, $5$)              & $19.5$   & $2.91 \times 10^{-3}$ & $7.88 \times 10^{-3}$ & $3.74 \times 10^{-3}$ & $1.44 \times 10^{-2}$ \\
                \midrule
                \multirow{2}{*}{2} & \ac{mc} & ($10000$, $1000$)                   & $486.0$  & $6.09 \times 10^{-2}$ & $1.09 \times 10^{-1}$ & $2.63 \times 10^{-2}$ & $1.44 \times 10^{-1}$ \\
                                   & \ac{aq} & ($7278$, $230$) ($3$, $5$)          & $413.5$  & $3.05 \times 10^{-2}$ & $6.59 \times 10^{-2}$ & $2.94 \times 10^{-2}$ & $1.07 \times 10^{-1}$ \\
                \bottomrule
            \end{tabular}
        }
    }
    \hfill
    \subfloat[$\tanh$.\label{tab:init_poiw_tanh}]{
        \resizebox{0.8\linewidth}{!}{%
            \begin{tabular}{cccccccc}
                \toprule
                Dim.               &         & ($N_\Omega$, $N_\Gamma$) ($P$, $O$) & Time (s) & min.                  & avg.                  & std.                  & max.                  \\
                \midrule
                \multirow{2}{*}{1} & \ac{mc} & ($200$, $2$)                        & $18.8$   & $1.12 \times 10^{-1}$ & $3.73 \times 10^{-1}$ & $3.96 \times 10^{-1}$ & $1.47 \times 10^{-0}$ \\
                                   & \ac{aq} & ($126$, $2$) ($5$, $10$)            & $18.2$   & $3.08 \times 10^{-3}$ & $9.00 \times 10^{-3}$ & $5.27 \times 10^{-3}$ & $2.19 \times 10^{-2}$ \\
                \midrule
                \multirow{2}{*}{2} & \ac{mc} & ($5000$, $500$)                     & $449.4$  & $6.44 \times 10^{-2}$ & $9.63 \times 10^{-2}$ & $3.77 \times 10^{-2}$ & $1.81 \times 10^{-1}$ \\
                                   & \ac{aq} & ($3224$, $167$) ($3$, $5$)          & $391.3$  & $3.50 \times 10^{-2}$ & $4.50 \times 10^{-2}$ & $6.59 \times 10^{-3}$ & $5.59 \times 10^{-2}$ \\
                \bottomrule
            \end{tabular}
        }
    }
    \caption{Distribution of the relative $L^2$ norm for the weak Poisson problem with the $\abse$ \protect\subref{tab:init_poiw_abs} and $\tanh$ \protect\subref{tab:init_poiw_tanh} activations on $10$ random initialisations. The learning rate is set to $10^{-2}$ and the points are resampled every $10$ epochs.}
    \label{tab:init_poiw}
\end{table}

%%%
\subsection{Reduction of the number of integration points}
\label{subsect:reduction}

In this set of experiments, we want to show that it is possible to further reduce the number of integration points generated by \ac{aq} by merging small regions with their neighbours, without sacrificing the accuracy of our method. We keep the same hyperparameters as in the previous paragraph and focus on the two-dimensional problems.

Every time the mesh is updated, we compute the median of the size of the cells. We then identify the cells that are smaller than a given fraction of this median and merge them one by one with their neighbours. We always merge a cell with the largest of its neighbours. In dimension two, after all the cells are merged in this way, the mesh needs an extra post-processing step in order to account for non-convex cells. Indeed, the union of two convex polygons may be non-convex. We split non-convex aggregated cells with the ear clipping algorithm \cite{meisters1975}, and apply our splitting strategy to all resulting cells to obtain triangles and convex quadrangles.

We experiment with several merging thresholds and report our findings in \tab{merging}. In dimension one, we find that merging regions up to $50\%$ of the median of the regions size does not significantly affect the error while allowing for a sizeable reduction of the number of integration points. In the two-dimensional case, raising the merging threshold above $25\%$ harms the performance of the $\abse$ network, but not that of the $\tanh$ network, as the number of integration points does not decrease.

It appears that increasing the merging threshold does not always reduce the number of integration points. Indeed, the aggregated regions may become less and less convex, so they will need to be split into several convex regions.

We illustrate the kind of meshes that a network generates and how they are merged in \fig{7_meshes}. These meshes are obtained from a trained $\tanh$ network solving the weak Poisson problem. We found that the meshes from the previous round of experiments were easier to interpret, so we selected one of the $10$ corresponding models. The boundary of the circular region can be easily identified. The lines that connect two sides of the square domain come from the first layer whereas the other lines come from the second layer. We observe that non-convex regions appear when small regions are merged.

\begin{table}
    \centering
    \subfloat[$\abse$. \ac{aq} has the following settings: $5$ pieces with order $10$ in 1D, $3$ pieces with order $2$ in 2D.\label{tab:merging_abse}]{
        \resizebox{0.45\linewidth}{!}{%
            \begin{tabular}{cccc}
                \toprule
                Dim.               &                          & ($N_\Omega$, $N_\Gamma$) (Threshold) & $L^2$ norm            \\
                \midrule
                \multirow{5}{*}{1} & \multirow{2}{*}{\ac{mc}} & ($200$, $2$)                         & $6.39 \times 10^{-1}$ \\
                                   &                          & ($100$, $2$)                         & $4.55 \times 10^{-1}$ \\
                \cmidrule(l){3-4}
                                   & \multirow{3}{*}{\ac{aq}} & ($290$, $2$) ($0\%$)                 & $8.25 \times 10^{-3}$ \\
                                   &                          & ($198$, $2$) ($50\%$)                & $7.82 \times 10^{-3}$ \\
                                   &                          & ($144$, $2$) ($100\%$)               & $4.41 \times 10^{-2}$ \\
                \midrule
                \multirow{5}{*}{2} & \multirow{2}{*}{\ac{mc}} & ($5000$, $500$)                      & $1.43 \times 10^{-1}$ \\
                                   &                          & ($1000$, $100$)                      & $2.35 \times 10^{-1}$ \\
                \cmidrule(l){3-4}
                                   & \multirow{3}{*}{\ac{aq}} & ($4041$, $158$) ($0\%$)              & $6.40 \times 10^{-2}$ \\
                                   &                          & ($3562$, $148$) ($10\%$)             & $7.27 \times 10^{-2}$ \\
                                   &                          & ($3217$, $132$) ($25\%$)             & $2.78 \times 10^{-2}$ \\
                                   &                          & ($2809$, $115$) ($50\%$)             & $1.41 \times 10^{-1}$ \\
                \bottomrule
            \end{tabular}
        }
    }
    \hfill
    \subfloat[$\tanh$. \ac{aq} has the following settings: $5$ pieces with order $10$ in 1D, $5$ pieces with order $5$ in 2D.\label{tab:merging_tanh}]{
        \resizebox{0.45\linewidth}{!}{%
            \begin{tabular}{cccc}
                \toprule
                Dim.               &                          & ($N_\Omega$, $N_\Gamma$) (Threshold) & $L^2$ norm            \\
                \midrule
                \multirow{5}{*}{1} & \multirow{2}{*}{\ac{mc}} & ($200$, $2$)                         & $2.34 \times 10^{-1}$ \\
                                   &                          & ($100$, $2$)                         & $4.92 \times 10^{-1}$ \\
                \cmidrule(l){3-4}
                                   & \multirow{3}{*}{\ac{aq}} & ($138$, $2$) ($0\%$)                 & $3.08 \times 10^{-3}$ \\
                                   &                          & ($94$, $2$) ($50\%$)                 & $4.18 \times 10^{-3}$ \\
                                   &                          & ($85$, $2$) ($100\%$)                & $2.72 \times 10^{-2}$ \\
                \midrule
                \multirow{5}{*}{2} & \multirow{2}{*}{\ac{mc}} & ($10000$, $1000$)                    & $6.61 \times 10^{-2}$ \\
                                   &                          & ($5000$, $500$)                      & $8.51 \times 10^{-2}$ \\
                \cmidrule(l){3-4}
                                   & \multirow{3}{*}{\ac{aq}} & ($7798$, $259$) ($0\%$)              & $5.02 \times 10^{-2}$ \\
                                   &                          & ($6895$, $218$) ($10\%$)             & $3.96 \times 10^{-2}$ \\
                                   &                          & ($8336$, $244$) ($25\%$)             & $5.26 \times 10^{-2}$ \\
                                   &                          & ($9088$, $204$) ($50\%$)             & $7.19 \times 10^{-2}$ \\
                \bottomrule
            \end{tabular}
        }
    }
    \caption{Relative $L^2$ norm for the weak Poisson problems with the $\abse$ \protect\subref{tab:merging_abse} and $\tanh$ \protect\subref{tab:merging_tanh} activations for various region merging thresholds.}
    \label{tab:merging}
\end{table}

\begin{figure}
    \centering
    \subfloat[$0\%$: $418$ regions. \label{fig:7_0.0}]{
        \includegraphics[width=0.45\linewidth]{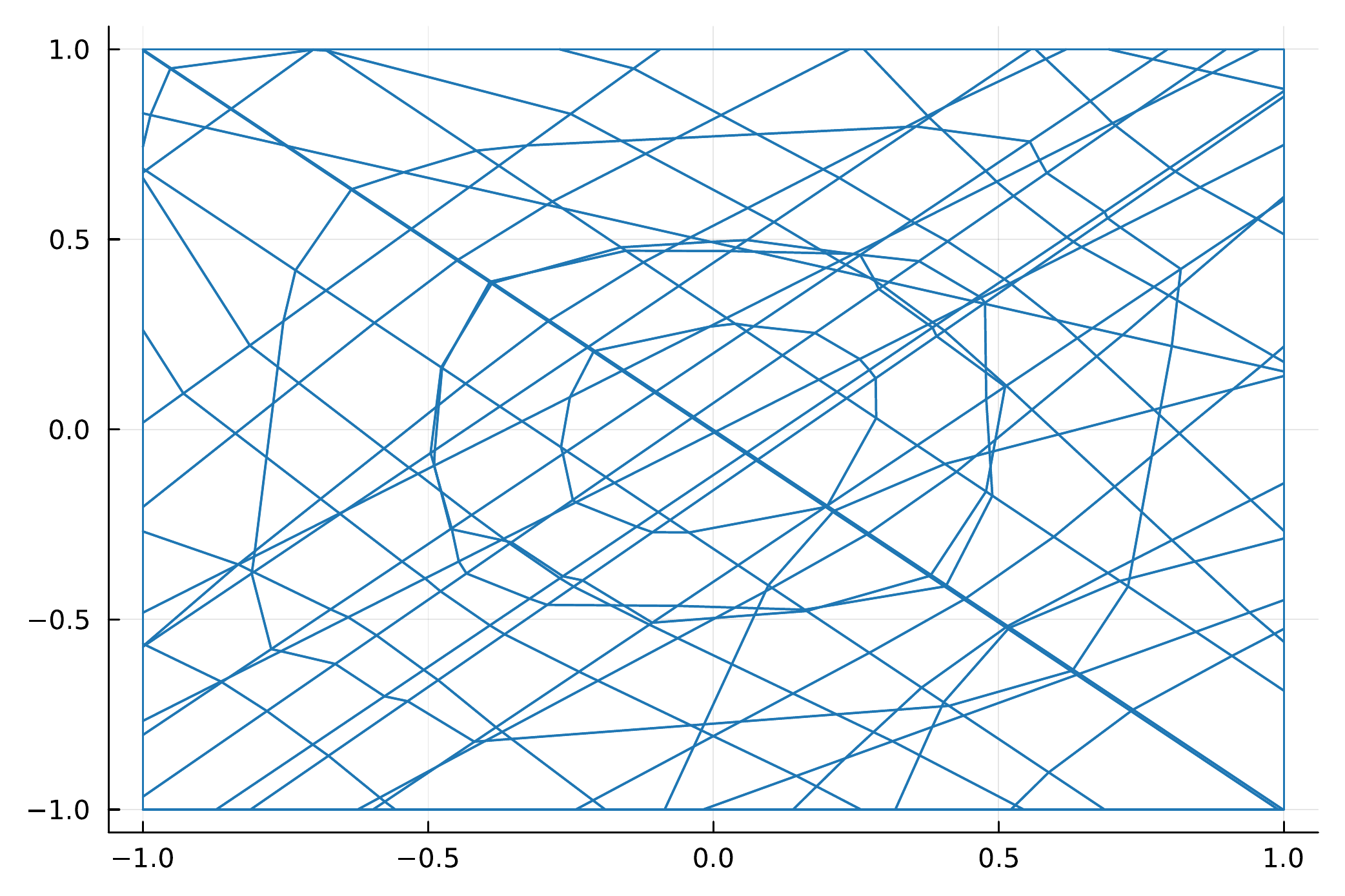}
    }
    \hfill
    \subfloat[$25\%$: $294$ regions. \label{fig:7_0.25}]{
        \includegraphics[width=0.45\linewidth]{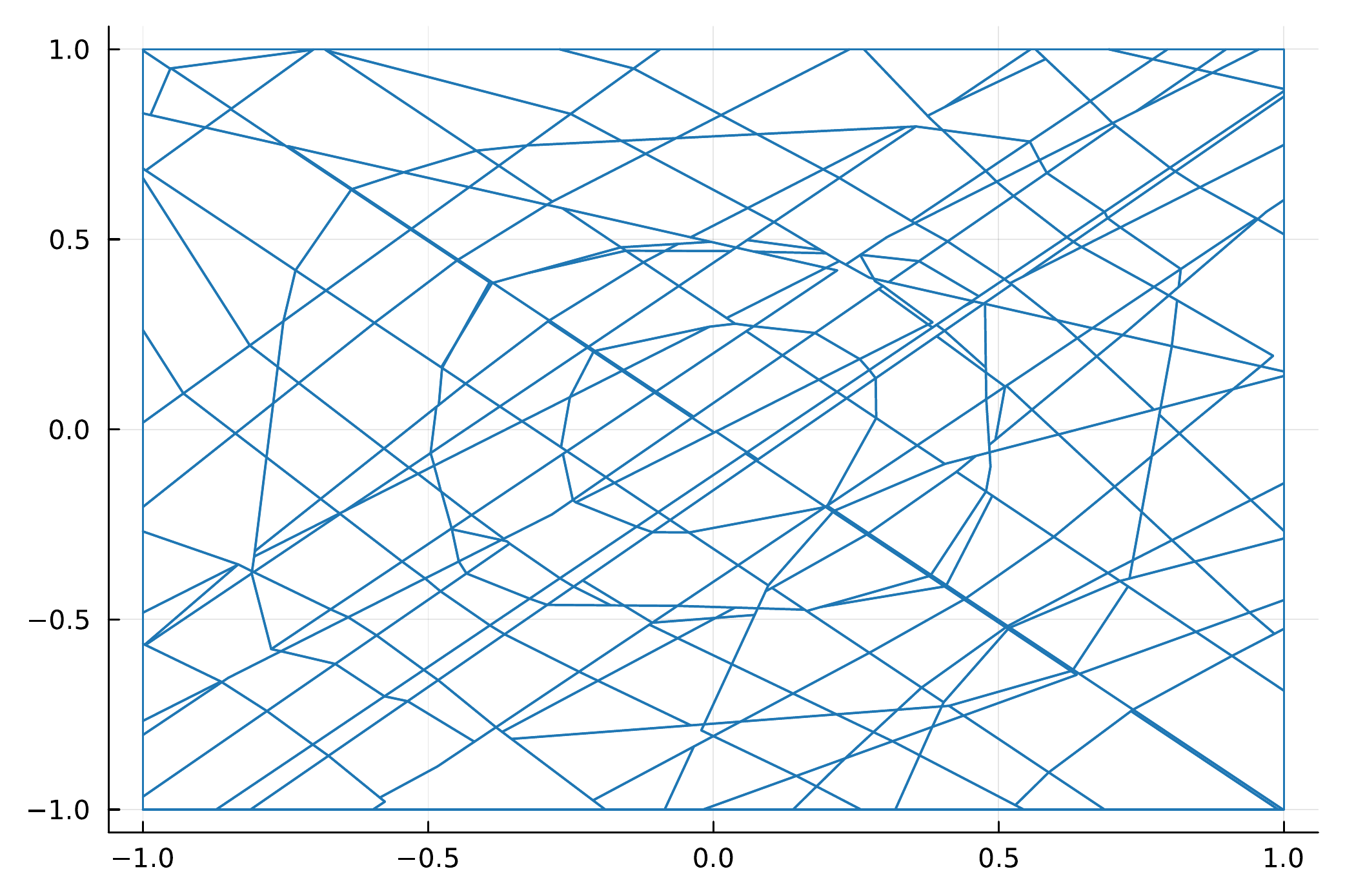}
    }
    \caption{Final mesh, number of cells, and number of integration points at several region merging thresholds for the weak Poisson problem and the $\tanh$ activation. The $\tanh$ activation is cut into $5$ pieces and the quadrature is of order $2$. The merging thresholds are $0\%$ \protect\subref{fig:7_0.0} and $25\%$ \protect\subref{fig:7_0.25}.}
    \label{fig:7_meshes}
\end{figure}

%%%
\subsection{More complex domains}
\label{subsect:rhombi}

We complete our numerical experiments by solving a weak Poisson equation on a more complex two-dimensional domain, defined as the union of two rhombi. We solve the following Poisson equation
\begin{equation}
    \left\{\begin{array}{rc}
        -\Delta u = x + y & \text{in } \Omega \\
        u = 0             & \text{on } \Gamma \\
    \end{array}\right.. \label{eq:rhombi}
\end{equation}

We obtain an approximation of the solution with the \ac{fem}. We rely on the following weak formulation for the \ac{fem} solution:
\[\left\{\begin{array}{l}
        \text{Find } u \in H^1_0(\Omega) \text{ such that for all } v \in H^1_0(\Omega), \\
        \int_\Omega \nabla u \cdot \nabla v \D \Omega = \int_\Omega (x + y) v \D \Omega
    \end{array}\right.,\]
where $H^1_0(\Omega)$ is the subset of functions in $H^1(\Omega)$ that vanish on $\Gamma$. We then obtain a mesh composed of $3894$ cells from GMSH and solve this \ac{fem} problem with the Gridap software \cite{badia2020}. We rely on first-order Lagrange finite elements, and the total number of degrees of freedom is $1852$. We use the \ac{fem} solution as the reference to compute the $L^2$ norm for the solution obtained by a \ac{nn}. The relative $L^2$ norm is computed with a quadrature of order $10$ on the cells of the mesh of the \ac{fem} solution.

We then approximate this equation with \acp{nn} using the same penalised weak form that we considered in the previous experiments. We increase the network architecture to $(2, 20, 20, 1)$ and the activation function is $\abse$. We train our network for $10000$ iterations using the ADAM optimiser with a learning rate of $10^{-2}$. We cut the $\abse$ activation into $3$ pieces, use a quadrature of order $2$ in each cell, and do not merge small regions. The integration points are resampled every $10$ epochs. On average, \ac{aq} involved $5894$ integration points in the domain and $292$ on the boundary. We match the computational budget of \ac{aq} for \ac{mc} by sampling $6000$ in-domain points and $300$ points on the boundary for \ac{mc}. To obtain a uniform distribution of integration points for \ac{mc}, we first perform a triangulation of the domain. Then for as many points as we need, we draw a triangle at random by weighting them by their measure and sample one point in the corresponding triangle.

We plot the solutions obtained by \ac{fem}, \ac{mc} and \ac{aq} in \fig{8_abs}. The pointwise difference between \ac{mc} and \ac{fem}, and between \ac{aq} and \ac{fem} is shown in \fig{8_rel}. The relative $L^2$ norm is $7.99 \times 10^{-2}$ for \ac{mc} and $4.68 \times 10^{-2}$ for \ac{aq}. We mention that the training times for the last experiment are $700$ seconds with \ac{aq} and $625$ seconds with \ac{mc}. We believe that this $12\%$ extension in time is well worth the $70\%$ decrease in $L^2$ norm. It is interesting to mention that the pointwise error is located around the same regions in the domain. However, when using \ac{aq} the magnitude of the local maxima of the pointwise error is much lower.

\begin{figure}
    \centering
    \subfloat[Approximation of \eq{rhombi} obtained with \ac{mc} (left), \ac{fem} (middle) and \ac{aq} (right).\label{fig:8_abs}]{
        \includegraphics[width=0.8\linewidth]{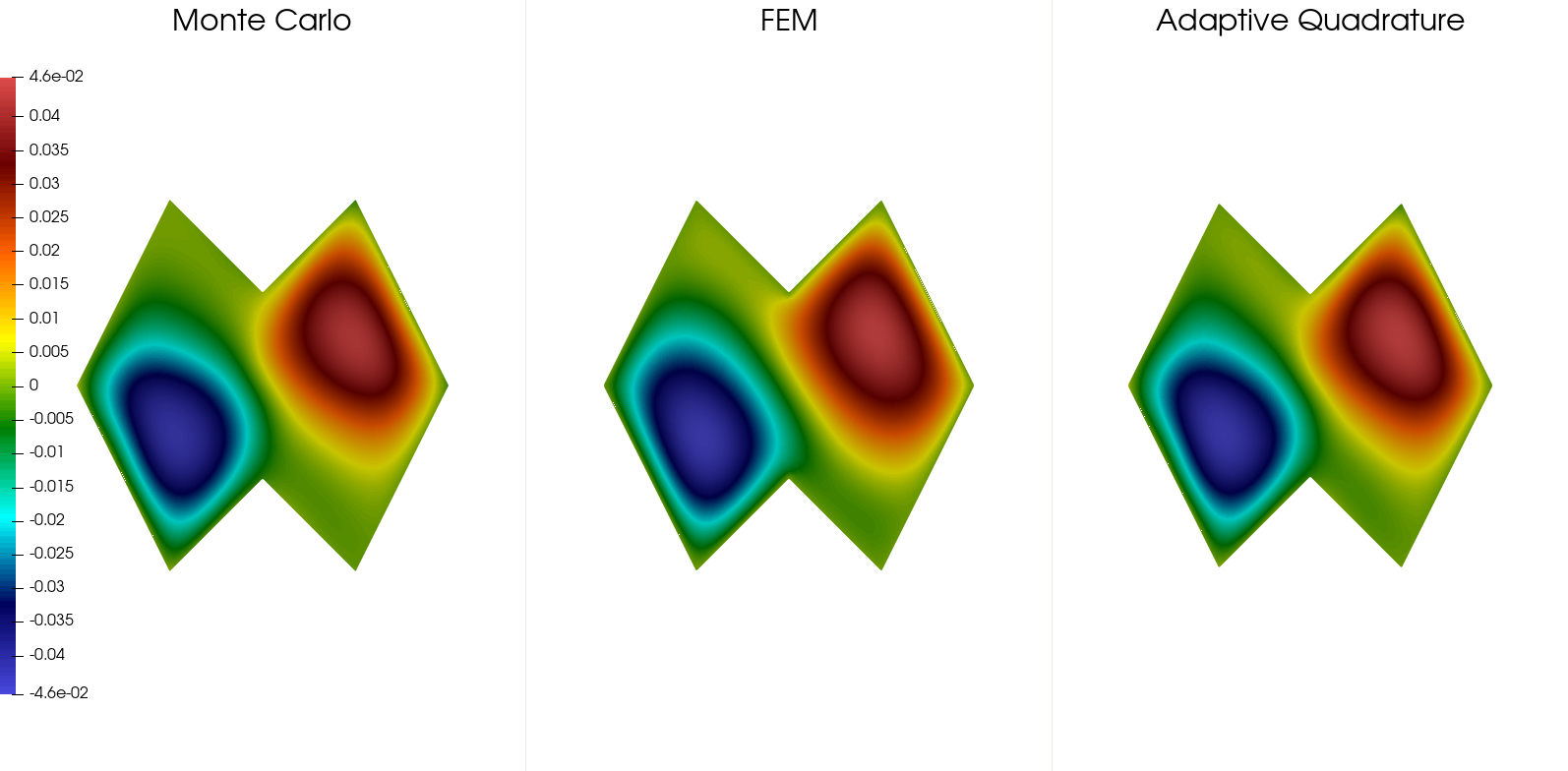}
    }

    \subfloat[Pointwise error between the \ac{fem} solution of \eq{rhombi} and the \ac{mc} solution (left), and the \ac{aq} solution (right). The relative $L^2$ norm is $7.99 \times 10^{-2}$ for \ac{mc} and $4.68 \times 10^{-2}$ for \ac{aq}. \label{fig:8_rel}]{
        \includegraphics[width=0.8\linewidth]{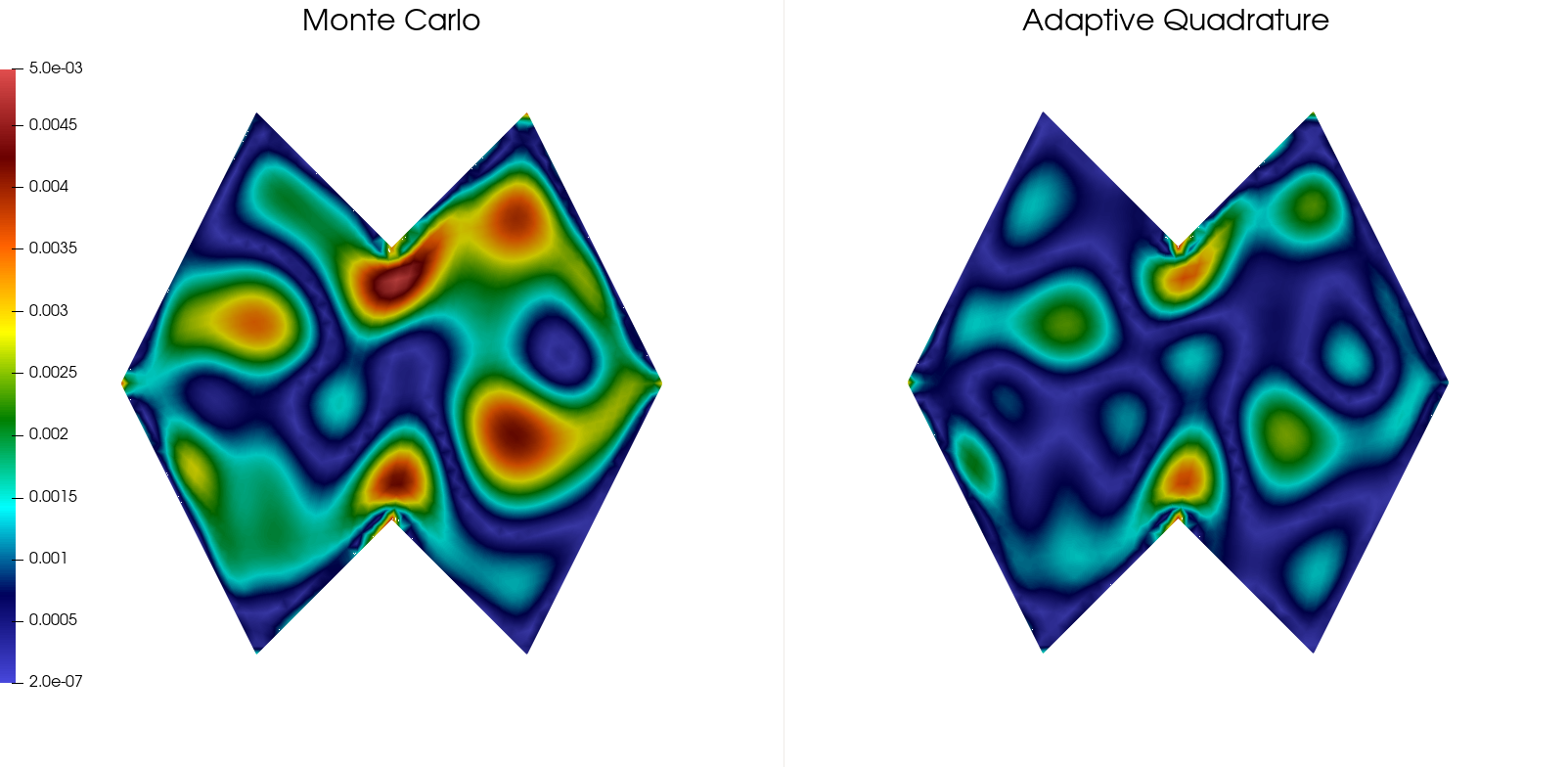}
    }
    \caption{Approximation \protect\subref{fig:8_abs} and pointwise error \protect\subref{fig:8_rel} for the solution of \eq{rhombi}.}
    \label{fig:8_rhombi}
\end{figure}

%%%
\subsection{Summary}
\label{subsect:summary}

Our numerical experiments confirmed the effectiveness of our proposed integration method in a wide range of scenarios. \ac{aq} enables to reach lower generalisation errors compared to \ac{mc} while using fewer integration points. We have shown that this new integration method is more robust to the initial parameters of the network. Moreover, the convergence of the network is less noisy and the generalisation error decays quicker, as \ac{aq} can sustain higher learning rates without introducing extra noise in the learning phase. We are convinced that our proposed integration method can help \acp{nn} become a more reliable and predictable tool when it comes to approximating \acp{pde}, so that they can be employed to solve real-world problems.

The most significant differences between \ac{mc} and \ac{aq} were observed for the weak Poisson problem. One reason could be that the loss of the strong Poisson problems can be understood in a pointwise sense. In this case, \ac{mc} minimises the pointwise distance between $\Delta u$ and $\Delta f$ at the location of the integration points. On the contrary, the loss of the weak Poisson problem is to be understood and minimised in a global sense. Our experiments also suggest that the benefits of our adaptive quadrature are more striking at the location of strong gradients.

%%%
\subsection{Discussion}
\label{subsect:discussion}

Our approach suffers from the curse of dimensionality and is intended to be used for low-dimensional \acp{pde} as a nonlinear alternative to the \ac{fem}. We noticed that the number of integration points in dimension one was of the order of a few dozen to a few hundred, when the usual number of integration points used for \ac{mc} in the literature is rather a few hundred to a few thousand. In dimension two, our method needs a few thousand integration points, whereas the models in the literature rely on a few tens of thousands of points. Even though our approach is still cheaper than \ac{mc} in the two-dimensional case, the reduction factor is much smaller compared to the one-dimensional scenario. We found that merging linear regions can help reduce the number of linear regions and thus integration points. This phenomenon should amplify as the dimension increases, but we are only ever concerned with problems arising from physics, which involve space and time only. We would like to extend our method to the three-dimensional case in the future.

We have observed that the time spent on extracting the mesh is offset by the reduction in the number of integration points, as the overall training very often takes less time with \ac{aq} than with \ac{mc}.

Our idea to regularise the activation function was motivated by the relationship between \ac{fem} and $\ReLU$ networks, as well as the need for a smooth activation so that the energy functional itself is smooth. We believe that it is more economical to use a smoothed \ac{cpwl} activation function than to directly use a smooth activation such as $\tanh$ because intuitively fewer pieces are necessary to reach a given tolerance in the former case. Whereas most of the current implementations of \ac{pinn} and its variants can only handle tensor-product domains or balls, our proposed linearisation algorithm can naturally handle any polygonal domain. This is an important step towards the use of \acp{nn} solvers for practical issues.

One direction for future work is tackling the variational setting in the general case when it is not possible to recast the problem as an energy minimisation. We point out that our domain decomposition strategy and adaptive quadrature can be extended to higher-order approximations (e.g. splines) of the activation function, which would result in curved boundaries for the regions.

%% file: src/7.conclusion.tex
In this work, we introduce a procedure to define adaptive quadratures for smooth \acp{nn} suitable for low-dimensional \ac{pde} discretisation. It relies on a decomposition of the domain into regions where the network is almost linear. We also show the importance of smoothness in the training process and propose a regularisation of the standard ReLU activation function.

We carry out the numerical analysis of the proposed framework to obtain upper bounds for the integration error. Numerical experimentation shows that our integration method helps make the convergence less noisy and quicker compared to Monte Carlo integration. We observe that the number of integration points needed to compute the loss function can be significantly reduced with our method and that our integration strategy is more robust to the initialisation of the network. We illustrated the benefit of our adaptive integration on the Poisson equation, but it can be extended to any \ac{pde} involving a symmetric and coercive bilinear form.

While \ac{nn} solvers are applied to increasingly intricate problems, they still lack reliability and theoretical guarantees before being used on real-world problems. Besides, most \ac{pinn} frameworks can only handle tensor-product domains, while the \ac{fem} can be applied to much more complex domains. The adaptive quadratures we propose, the error bounds we prove, and the seamless possibility to handle polygonal domains of our method represent a few steps further in bridging these gaps.

For the sake of reproducibility, our implementation of the proposed adaptive quadrature is openly available at \cite{magueresse2023}.

%% file: src/8.acknowledgements.tex
This research was partially funded by the Australian Government through the Australian Research Council (project number DP220103160).

%% file: src/9.proofs.tex
\subsection{Proof of Lemma \ref{lem:integrability}}
\label{proof:lem:integrability}

\begin{proof}
    Let $\alpha > 0$ and $\rho \in \mathcal{A}^\alpha$. We only show the lemma at $+\infty$, as it is obtained by symmetry at $-\infty$ by considering $\tilde{\rho}: x \mapsto \rho(-x)$. Let $M > 0$ and define $I = \oo{M}{+\infty}$. Since $x \mapsto \abs{x}^{2 + \alpha} \rho''(x)$ is bounded on $\RR$, the $p$-test for improper integrals implies that $\rho''$ is integrable on $I$. Besides, for all $x \geq M$, one can write
    $$\rho'(x) = \rho'(x) - \rho'(M) + \rho'(M) = \int_M^x \rho''(t) \D{t} + \rho'(M).$$
    This shows that $\rho'$ has a finite limit at $+\infty$, that we write $A$. Furthermore, it holds
    $$\abs{\rho'(x) - A} = \abs{\int_x^\infty \rho''(t) \D{t}} \leq \int_x^\infty \abs{\rho''(t)} \D{t} \lesssim \int_x^\infty \frac{1}{t^{2 + \alpha}} \D{t} = \frac{1}{(1 + \alpha) x^{1 + \alpha}}.$$
    This proves that $x \mapsto \abs{x}^{1 + \alpha} \abs{\rho'(x) - A}$ is bounded on $I$. We introduce the function $\eta: x \mapsto \rho(x) - A x$, which is such that $\eta'(x) = \rho'(x) - A$ so $x \mapsto \abs{x}^{1 + \alpha} \abs{\eta'(x)}$ is bounded. Applying the same arguments as above to $\eta$ shows that it has a finite limit $B$ at $+\infty$ and that $x \mapsto \abs{x}^\alpha \abs{\eta(x) - B} = \abs{x}^\alpha \abs{\rho(x) - A x - B}$ is bounded. We conclude that $\rho$ behaves linearly at $+\infty$, as it has a slant asymptote $\rho_{+\infty}: x \mapsto A x + B$ at $+\infty$.

    Using the $p$-test for improper integrals again, $x \mapsto \abs{x}^\alpha \abs{\rho(x) - \rho_{+\infty}(x)}$ being bounded on $I$ implies that $\rho - \rho_{+\infty}$ belongs to $L^p(I)$ for all $p > 1/\alpha$. Since $\rho$ and $\rho_{+\infty}$ are continuous on $\cc{0}{M}$, $\rho - \rho_{+\infty}$ is bounded on $\cc{0}{M}$ and the $L^p$ integrability of $\rho - \rho_{+\infty}$ extends to $\RR_+$ as a whole. Similar arguments show that $\rho' - \rho_{+\infty}' \in L^q(\RR_+)$ for all $q > 1/(1 + \alpha)$. Finally, $\rho''$ is bounded on $\RR$ (by definition of $\mathcal{A}^\alpha$) so in particular it is bounded on $\cc{0}{M}$. Together with the boundedness of $x \mapsto \abs{x}^{2 + \alpha} \rho''(x)$, we conclude that $\rho'' \in L^r(\RR_+)$ for all $r > 1/(2 + \alpha)$.
\end{proof}

\subsection{Proof of Lemma \ref{lem:bound}}
\label{proof:lem:bound}

\begin{proof}
    Let $\alpha > 0$, $\rho \in \mathcal{A}^\alpha$, $1/\alpha < p < \infty$ and $I \in P(\rho)$ be a compatible interval for the function $\rho$. If $\rho$ is linear on $I$, then for all $x < y \in I$, $T[\rho, x]$ coincides with $\rho$ on $\oo{x}{y}$, and so do $T[\rho, y]$ and $T[\rho, x, y]$. Besides, $\rho'' = 0$ so the proposition holds trivially. We thus suppose that $\rho$ is either strictly convex or strictly concave on $I$, which means that $\rho''$ is zero at the two ends of $I$ and nowhere else in $I$.

    For $\delta \geq 0$, we define the set $E_\delta = \{(x, y) \in I^2, y - x \geq \delta\}$. We also write $E = E_0$. Let now $\delta > 0$ be fixed. For all $(x, y) \in E_\delta$, the quantity $\aabs{\rho''}_{L^{p/(2p+1)}(\oo{x}{y})}$ is not zero since $y > x$ and the integrand only vanishes on a set of zero measure. Besides, if $\oo{x}{y}$ is infinite, $T[\rho, x, y]$ coincides with $\rho_{\pm \infty}$ and since $p > 1/\alpha$, \lem{integrability} ensures that $\rho - T[\rho, x, y]$ belongs to $L^p(\oo{x}{y})$. This enables us to consider the function
    \[r: E_\delta \ni (x, y) \mapsto \frac{\aabs{\rho - T[\rho, x, y]}_{L^p(\oo{x}{y})}}{\aabs{\rho''}_{L^{p/(2p+1)}(\oo{x}{y})}}.\]
    We want to show that $r$ can be defined on $E$ as a whole, and that $r$ is bounded on $E$.

        {\itshape (i) Continuity and differentiability of $c$ on $E_\delta$.}
    By definition of $I$, for all $x < y \in I$, the tangents to $\rho$ at $x$ and $y$ intersect at an abscissa that we write $c(x, y)$, with $x < c(x, y) < y$. More precisely, for finite $x < y \in I$, the function $c$ has the following expression
    \[c(x, y) = -\frac{[\rho(y) - y \rho'(y)] - [\rho(x) - x \rho'(x)]}{\rho'(y) - \rho'(x)},\]
    in which the denominator is not zero because $\rho'$ is injective on $I$. We extend $c$ to possible infinite values of $x$ or $y$ by setting $c(x, +\infty) = \lim_{y \to +\infty} c(x, y)$ and $c(-\infty, y) = \lim_{x \to -\infty} c(x, y)$. Since $\rho$ and $\rho'$ are continuous on $I$, and since $\rho'$ is injective on $I$, we conclude that $c$ is continuous on $E_\delta$. By using the differentiation rule for quotients, similar arguments show that $c$ is differentiable on $E_\delta$.

        {\itshape (ii) Continuity of $r$ on $E_\delta$.}
    Since $\rho$ and $\rho'$ are continuous on $\RR$, the map $x \mapsto T[\rho, x]$ is continuous. Besides, $c$ is continuous so the map $(x, y) \to T[\rho, x, y]$ is also continuous. Finally, any map $(x, y) \to \int_x^y \phi(t, x, y) \D t$ is continuous whenever $\phi$ is continuous. Here, $(x, y) \mapsto \rho - T[\rho, x, y]$ and $\rho''$ are continuous, so we conclude that the numerator and the denominator of $r$ are continuous. Moreover, since $y - x \geq \delta > 0$, the denominator of $r$ is not zero, which makes $r$ continuous on $E_\delta$.

    \lem{integrability} ensures that the numerator of $r$ is bounded. The denominator is bounded from below because $y - x > \delta$ and $\rho''$ only vanishes on a set of zero measure. It is bounded from above because when $p > 1/\alpha$, we verify that $p/(2p+1) > 1/(\alpha+2)$ so \lem{integrability} ensures that $\rho''$ belongs to $L^{p/(2p+1)}(\RR)$. This proves that $r(x, y)$ is bounded on $E_\delta$. Thus the only case where $r$ could be unbounded is when $x$ and $y$ get arbitrarily close ($\delta$ goes to zero).

        {\itshape (iii) Continuity and differentiability of $c$ on $E$.}
    We now show that $c$ can be continuously extended to $E$ and that this extension is differentiable on $E$. To do so, we prove that for all element $x \in I$, $\lim_{y \to x} c(x, y) = x$. One can prove that for all $y \in I$, $\lim_{x \to y} c(x, y) = y$ using a similar argument. Let $\epsilon \leq y - x$, so that $x + \epsilon \in I$. We start by rearranging the terms in the expression of $c(x, x + \epsilon)$ and obtain the following
    \[c(x, x + \epsilon) = x + \frac{\rho'(x + \epsilon) - \epsilon^{-1} (\rho(x + \epsilon) - \rho(x))}{\rho'(x + \epsilon) - \rho'(x)} \epsilon.\]
    Now, let $k \geq 2$ be the smallest integer such that the $k$-th derivative of $\rho$ at $x$ is not zero. By definition of a compatible interval, $k$ will be equal to $2$ for all interior points of $I$. Since $\rho$ is analytic in the neighbourhood of the zeros of $\rho''$, there also exists such a $k > 2$ for the ends of $I$. We can always take $\epsilon$ as small as needed so that $\rho$ is analytic around $x$ in a neighbourhood of radius $\epsilon$. We obtain a Taylor expansion of the fraction above by expanding its numerator and denominator to order $k-1$ around $x$, expressed at $x + \epsilon$, and arrive at
    \[c(x, x + \epsilon) = x + \frac{k - 1}{k} \epsilon + O(\epsilon^2).\]
    This shows that the function $\hat{c}$ defined by $\hat{c}(x, y) = c(x, y)$ if $x < y$ and $\hat{c}(x, x) = x$ is continuous and differentiable on $E$.

        {\itshape (iv) Continuity of $r$ on $E$.}
    Let $x \in I$. We show that $\epsilon \mapsto r(x, x + \epsilon)$ has a finite limit when $\epsilon$ goes to zero by finding an equivalent of the numerator and the denominator of $r$. We write the numerator as $(I_1(\epsilon)^p + I_2(\epsilon)^p)^{1/p}$, where $I_1$ and $I_2$ are the integrals on $\oo{x}{c(x, x+\epsilon)}$ and $\oo{c(x, x+\epsilon)}{x + \epsilon}$ respectively. The denominator of $r$ is written $I_3(\epsilon)$. Here again, let $k \geq 2$ be the smallest integer such that the $k$-th derivative of $\rho$ at $x$ is not zero.

    We use the mean-value form of the remainder in the $k$-th-order Taylor expansion of $\rho$ to obtain that for all $t \in \oo{x}{c(x, x + \epsilon)}$, there exists $\eta_t \in \oo{x}{t}$ such that $\rho(t) - T[\rho, x](t) = \frac{1}{k!} \rho^{(k)}(\eta_t) (t - x)^k$. Then, owing to the mean-value theorem for integrals since $\rho^{(k)}$ is continuous, there exists $\eta \in \oo{x}{c(x, x + \epsilon)}$ such that
    \[I_1(\epsilon)^p = \int_{x}^{c(x, x + \epsilon)} \abs{\frac{1}{k!} \rho^{(k)}(\eta_t) (t - x)^k}^p \D t = \frac{1}{k!^p (kp+1)} |\rho^{(k)}(\eta)|^p (c(x, x + \epsilon) - x)^{kp+1}.\]
    We can now use the previous Taylor expansion of $c$ around $(x, x)$ and also expand $\rho^{(k)}(\eta)$ around $x$ since $\eta$ goes to $x$ as $\epsilon$ approaches zero. We finally obtain the following expansion for $I_1(\epsilon)^p$
    \[I_1(\epsilon)^p = \frac{1}{k!^p (kp+1)} \left(\frac{k-1}{k}\right)^{kp+1} |\rho^{(k)}(x)|^p \epsilon^{kp+1} + o(\epsilon^{kp+1}).\]
    We obtain a $k$-th order Taylor expansion of $\rho - T[\rho, x + \epsilon]$ as follows: there exist $\eta_1, \eta_2 \in \oo{x}{x + \epsilon}$, such that for all $t \in \oo{c(x, x + \epsilon)}{x + \epsilon}$, there exists $\eta_t$ between $x$ and $t$ such that
    \[\rho(t) - T[\rho, x + \epsilon](t) = \frac{1}{k!} \left[\rho^{(k)}(\eta_t) (t - x)^k - k \rho^{(k)}(\eta_2) \epsilon^{k-1} (t - x) + k \rho^{(k)}(\eta_2) \epsilon^k - \rho^{(k)}(\eta_1) \epsilon^k\right] + o(\epsilon^k).\]
    We observe that $\eta_1$, $\eta_2$ and $\eta_t$ go to $x$ as $\epsilon$ approaches zero. This means that the Taylor expansion of $\rho^{(k)}(\eta_t)$ is $\rho^{(k)}(x) + O(\epsilon)$, and similarly at $\eta_1$ and $\eta_2$. Besides, using the expansion of $c(x, x+ \epsilon)$, we show that $x + \frac{k-1}{k} \epsilon + o(\epsilon) \leq t \leq x + \epsilon$, which means that $t - x$ is of the order of $\epsilon$. After keeping the terms of order $\epsilon^k$, factoring by $\rho^{(k)}(x) \epsilon^k$ and introducing the change of variable $u = \frac{t - x}{\epsilon}$, we obtain
    \[\rho(t) - T[\rho, x + \epsilon](t) = \frac{\rho^{(k)}(x)}{k!} \left[u^k - k u + k - 1\right] \epsilon^k + o(\epsilon^k).\]
    The bounds of the integral become $\frac{k-1}{k} + o(1)$ and $1$, and we decompose the integration domain into $\oo{\frac{k-1}{k}}{1}$ and $\oo{\frac{k-1}{k} + o(1)}{\frac{k-1}{k}}$. The second integral is going to be negligible against the first because its integrand is bounded and its integration domain has a length of $o(1)$. Putting everything together, we obtain the following expansion for $I_2(\epsilon)^p$
    \[I_2(\epsilon)^p = \frac{1}{k!^p} \left(\int_{\frac{k-1}{k}}^{1} \abs{u^k - 1 - k (u - 1)}^p \D u\right) |\rho^{(k)}(x)|^p \epsilon^{kp+1} + o(\epsilon^{kp+1}).\]
    Using a similar argument, the Taylor expansion of the integrand of the denominator of $r$ is $\frac{1}{(k-2)!} \rho^{(k)}(x)(t - x)^{k-2} + o((t - x)^{k-2})$ and taking the $L^{p/(2p+1)}$ norm of this relationship, we find that the expansion of $I_3(\epsilon)^p$ is
    \[I_3(\epsilon)^p = \frac{1}{(k-2)!^p} \left(\frac{2p+1}{kp+1}\right)^{2p+1} |\rho^{(k)}(x)|^p \epsilon^{kp+1} + o(\epsilon^{kp+1}).\]
    Altogether, in the expression of $r(x, x + \epsilon)^p$, we can factor the numerator and the denominator by $|\rho^{(k)}(x)|^p \epsilon^{kp+1}$ and taking the limit as $\epsilon$ goes to zero, we obtain
    \[\lim_{\epsilon \to 0} r(x, x + \epsilon)^p = \frac{1}{(k(k-1))^p} \left(\frac{kp+1}{2p+1}\right)^{2p+1} \left(\frac{1}{kp+1} \left(\frac{k-1}{k}\right)^{kp+1} + \int_{\frac{k-1}{k}}^{1} \abs{u^k - 1 - k (u - 1)}^p \D u\right).\]
    This shows that the function $\hat{r}$ defined by $\hat{r}(x, y) = r(x, y)$ if $x < y$ and $\hat{r}(x, x)$ as the quantity above is continuous on $E$ as a whole. Besides, $\hat{r}$ has finite limits on $\partial E$, so we conclude that it is bounded on $E$.

        {\itshape (v) Proof of the other inequalities.}
    The same proof can be adapted for the bound on the derivatives, and the bounds in $L^\infty$ norm (\lem{integrability} guarantees that $\rho'' \in L^{1/2}(\RR)$). There is one important change when bounding the derivative: all instances of $k$ must be replaced by $k-1$ (since $\rho^{(k)} = (\rho')^{(k-1)}$), so the appropriate norm of $\rho''$ needed to cancel $\epsilon^{(k-1)p + 1}$ is $p/(p+1)$. The other arguments are the same everywhere, up to different constants in the equivalent of $r$ at points of the type $(x, x$). It is still possible to define the corresponding function $r$ on $E$ as a whole and to show that it is bounded on $E$.
\end{proof}

\subsection{Proof of Proposition \ref{prop:cpwl}}
\label{proof:prop:cpwl}

\begin{proof}
    This proof is inspired by the heuristics given in \cite{berjon2015}. We keep the notations of the definition in \subsect{approximation_space}. Let $\alpha > 0$, $\rho \in \mathcal{A}^\alpha$, $1/\alpha < p \leq \infty$, $q \geq p$ and $I$ be a compatible interval for $\rho$. If $\rho$ is linear on $I$, then $\rho$ coincides with its tangent at either end of $I$ so it is enough to define the restriction of $\pi_{n, p}[\rho]$ to $I$ as this tangent. We thus decompose the integral norms of $\rho - \pi_{n, p}[\rho]$ and $\rho' - \pi_{n, p}[\rho]'$ on the $\kappa(\rho)$ compatible intervals for $\rho$ where $\rho$ is not linear. We decide to place $n' = \floor{n/\kappa(\rho)}$ points in each interval.

        {\itshape (i) Bound in one compatible interval for $q < \infty$.}
    Since $\pi_{n, p}[\rho]$ coincides with $\rho_{\pm \infty}$ on the first and last compatible intervals for $\rho$, \lem{integrability} ensures that when $q \geq p > 1/\alpha$, the functions $\rho - \pi_{n, p}[\rho]$ and $\rho' - \pi_{n, p}[\rho]'$ belong to $L^q(\RR_{\pm \infty})$.

    Let $I = \oo{a}{b}$ be a compatible interval for $\rho$ such that $\rho$ is not linear on $I$. Let $a < \xi_1 < \ldots < \xi_{n'} < b$ the free points in $I$. For convenience we introduce $\xi_0 = a$ and $\xi_{n' + 1} = b$. For $0 \leq k \leq n'$, we write $I_k = \oo{\xi_k}{\xi_{k+1}}$. By construction of $\pi_{n, p}[\rho]$, the restriction of $\pi_{n, p}[\rho]$ to $I_k$ coincides with $T[\rho, \xi_k, \xi_{k+1}]$. We apply \lem{bound} to obtain the upper bounds
    \begin{align}
        \aabs{\rho - \pi_{n, p}[\rho]}_{L^q(I)}^q   & = \sum_{0 \leq k \leq n'} \aabs{\rho - T[\rho, \xi_k, \xi_{k+1}]}_{L^q(I_k)}^q \leq A_q(\rho)^q \sum_{0 \leq k \leq n'} \aabs{\rho''}_{L^{q/(2q+1)}(I_k)}^q, \label{eq:bound_P}   \\
        \aabs{\rho' - \pi_{n, p}[\rho]'}_{L^q(I)}^q & = \sum_{0 \leq k \leq n'} \aabs{\rho' - T[\rho, \xi_k, \xi_{k+1}]'}_{L^q(I_k)}^q \leq B_q(\rho)^q \sum_{0 \leq k \leq n'} \aabs{\rho''}_{L^{q/(q+1)}(I_k)}^q. \label{eq:bound_DP}
    \end{align}
    We introduce the notations $q_1 = q/(q+1)$ and $q_2 = q/(2q+1)$ to make the rest of the proof more readable. We are now looking for a choice of free points that provides a satisfying upper bound for both inequalities. We do so by equidistributing the $L^{q_2}$ norm of $\rho''$ on the subsegments $I_k$. Let $\eta$ be the function defined on $I$ by
    \[\eta(\xi) = \int_{\xi_0}^{\xi} \abs{\rho''(t)}^{q_2} \D{t}.\]
    Since $\rho''$ may only vanish on a set of zero measure, it is clear that $\eta$ is strictly increasing. Besides, we check that $\eta(\xi_0) = 0$ and $\eta(\xi_{n'+1}) = \aabs{\rho''}_{L^{q_2}(I)}^{q_2}$. Therefore, by the intermediate value theorem, we can select $\xi_k$ such that $\eta(\xi_k) = k (n'+1)^{-1} \aabs{\rho''}_{L^{q_2}(I)}^{q_2}$ for all $1 \leq k \leq n'$.

    Now, the terms in the sum of \eq{bound_P} are equal to $\left(\aabs{\rho''}_{L^{q_2}(I_k)}^{q_2}\right)^{q/q_2} = (\eta(\xi_{k+1}) - \eta(\xi_k))^{2q+1}$, and with this choice of $(\xi_k)$ these terms are all constant to $(n'+1)^{-(2q+1)} \aabs{\rho''}_{L^{q_2}(I)}^{q}$. Therefore, \eq{bound_P} becomes
    \begin{align}
        \aabs{\rho - \pi_{n, p}[\rho]}_{L^q(I)}^q \leq A_q(\rho)^q (n'+1)^{-2q} \aabs{\rho''}_{L^{q_2}(\RR)}^{q}, \label{eq:bound_P2}
    \end{align}
    where we used the fact that $\rho''$ is $L^{q_2}$ integrable on $\RR$ as a whole to bound the $L^{q_2}$ norm on $I$.

    The partition of $I$ above does not equidistribute $\aabs{\rho''}_{L^{q_1}(I_k)}^{q_1}$ so we cannot use the same method to bound \eq{bound_DP}. Instead, we break down $\abs{\rho''}^{q_1}$ into $\abs{\rho''}^{q_1-q_2} \abs{\rho''}^{q_2}$ and use the fact that $\rho''$ is bounded to move the first factor out of the integral:
    \[
        \aabs{\rho''}_{L^{q_1}(I_k)}^{q_1} \leq \aabs{\rho''}_{L^\infty(I_k)}^{q_1 - q_2} \aabs{\rho''}_{L^{q_2}(I_k)}^{q_2} \leq (n' + 1)^{-1} \aabs{\rho''}_{L^\infty(\RR)}^{q_1 - q_2} \aabs{\rho''}_{L^{q_2}(\RR)}^{q_2}.
    \]
    We make use of this result to bound \eq{bound_DP}. After simplification, we have
    \begin{align}
        \aabs{\rho' - \pi_{n, p}[\rho]'}_{L^q(I)}^q \leq B_q(\rho)^q (n' + 1)^{-q} \aabs{\rho''}_{L^\infty(\RR)}^{q^2/(2q+1)} \aabs{\rho''}_{L^{q_2}(\RR)}^{q(q+1)/(2q+1)}. \label{eq:bound_DP2}
    \end{align}

    {\itshape (ii) Bound on $\RR$ as a whole for $q < \infty$.}
    Let $\mathcal{I}(\rho)$ denote the set of compatible intervals for $\rho$. We apply the strategy presented in {\itshape (i)} to all $I \in \mathcal{I}(\rho)$ and find an upper bound for the $L^q$ norm on $\RR$ as a whole. Applying \eq{bound_P2}, we finally reach
    \begin{align*}
        \aabs{\rho - \pi_{n, p}[\rho]}_{L^q(\RR)}^q & = \sum_{I \in \mathcal{I}(\rho)} \aabs{\rho - \pi_{n, p}[\rho]}_{L^q(I)}^q \leq \sum_{I \in \mathcal{I}(\rho)} A_q(\rho)^q (n'+1)^{-2q} \aabs{\rho''}_{L^{q_2}(\RR)}^{q} \\
                                                    & \leq \kappa(\rho) A_q(\rho)^q (n'+1)^{-2q} \aabs{\rho''}_{L^{q_2}(\RR)}^{q}                                                                                              \\
                                                    & \leq A_q(\rho)^q \kappa(\rho)^{2q+1} \aabs{\rho''}_{L^{q_2}(\RR)}^{q} n^{-2q}.
    \end{align*}
    The last inequality comes from the fact that for all $m, n > 0$, $m \floor{n/m} \geq n - m$. We conclude by raising both sides to the power $1/q$: the $L^q$ norm of $\rho - \pi_{n, p}[\rho]$ decays as $n^{-2}$. We apply the same method to $\rho' - \pi_{n, p}[\rho]$: using \eq{bound_DP2} yields
    \begin{align*}
        \aabs{\rho' - \pi_{n, p}[\rho]'}_{L^q(\RR)}^q & \leq B_q(\rho)^q \kappa(\rho)^{q+1} \aabs{\rho''}_{L^\infty(\RR)}^{q^2/(2q+1)} \aabs{\rho''}_{L^{q_2}(\RR)}^{q(q+1)/(2q+1)} n^{-q}.
    \end{align*}
    Here again, we conclude by raising both sides to the power $1/q$: the $L^q$ norm of $\rho' - \pi_{n, p}[\rho]'$ decays as $n^{-1}$.

        {\itshape (iii) Extension to $p = \infty$.}
    The same bounds are obtained for the $L^\infty$ norm, by considering the maximum $L^\infty$ norm over the $n'+1$ subsegments and sampling the free points so as to equidistribute $\abs{\rho''}^{1/2}$.
\end{proof}

\subsection{Proof of Lemma \ref{lem:convexity}}
\label{proof:lem:convexity}

\begin{proof}
    We know that the cells in $\tau_{\RR^d}(\bm{\vartheta})$ are all convex since they are defined as the interior of intersections of half-planes. The final mesh $\tau_\Omega(\bm{\vartheta})$ is the intersection of $\tau_{\RR^d}(\bm{\vartheta})$ against $\Omega$. Let $K \in \tau_\Omega(\bm{\vartheta})$. There exists $K' \in \tau_{\RR^d}(\bm{\vartheta})$ such that $K = K' \cap \Omega$. If $K' \subseteq \Omega$, then $K = K'$ is convex. By the contraposition principle, if $K$ is not convex, then $K' \nsubseteq  \Omega$. This shows that if $K$ is a non-convex cell in $\tau_\Omega(\bm{\vartheta})$, then $K$ intersects with both $\Omega$ and $\RR^d \backslash \Omega$, which means that $K$ intersects with the boundary of $\Omega$. In the case where $\Omega$ is convex, the intersection of two convex sets being convex implies that $K = K' \cap \Omega$ is always convex.
\end{proof}

\subsection{Proof of Proposition \ref{prop:nns}}
\label{proof:prop:nns}

\begin{proof}
    Let $\Omega \subset \RR^d$ be a bounded domain. Let $\rho$ and $\sigma$ be two continuous functions almost everywhere differentiable such that $\rho - \sigma$, $\rho'$ and $\sigma'$ are bounded on $\RR$, and $\rho$ and $\rho'$ are Lipschitz continuous. We recall that we define the $L^\infty$ norm of a vector- or matrix-valued function as the maximum $L^\infty$ norm of its coordinates.

    We prove the proposition by induction, starting from the identity and showing that the inequalities still hold after composition by a linear map or an activation function. As will be made clear in {\itshape (iii)}, we also need to show that $\nabla u_\sigma$ is bounded in $\Omega$. In this proof, the new constants after the induction are denoted with a prime symbol.

        {\itshape (i) Induction hypotheses.}
    Suppose that $u_\rho$ and $u_\sigma$ take values in $\RR^n$ and that the following inequalities hold for some constants $C_1$, $C_2$, $C_3$, $C_{\bm{\vartheta}}$ and integer $\alpha$
    \begin{align*}
        \aabs{u_\rho - u_\sigma}_{L^\infty(\Omega, \RR^n)}                          & \leq C_1 \aabs{\rho - \sigma}_{L^\infty(\RR)},                                              \\
        \aabs{\nabla u_\rho - \nabla u_\sigma}_{L^\infty(\Omega, \RR^{d \times n})} & \leq C_2 \aabs{\rho - \sigma}_{L^\infty(\RR)} + C_3 \aabs{\rho' - \sigma'}_{L^\infty(\RR)}, \\
        \aabs{\nabla u_\sigma}_{L^\infty(\Omega, \RR^{d \times n})}                 & \leq C_{\bm{\vartheta}} \aabs{\sigma'}_{L^\infty(\RR)}^\alpha.
    \end{align*}

    {\itshape (ii) Composition with a linear map.}
    Let $\bm{\Theta}: \bm{x} \to \bm{W} \bm{x} + \bm{b}$, where $\bm{W} \in \RR^{m \times n}$ and $\bm{b} \in \RR^{m}$. Let $1 \leq i \leq m$. The triangle inequality yields
    \begin{align*}
        \aabs{(\bm{\Theta} \circ u_\rho)_i - (\bm{\Theta} \circ u_\sigma)_i}_{L^\infty(\Omega)} & \leq \sum_{1 \leq j \leq n} \abs{\bm{W}_{ij}} \aabs{(u_\rho - u_\sigma)_j}_{L^\infty(\Omega)} \leq \aabs{\pmb{\Theta}}_{\infty, 1} \aabs{u_\rho - u_\sigma}_{L^\infty(\Omega, \RR^n)},
    \end{align*}
    where we have introduced the norm $\aabs{\bm{\Theta}}_{\infty, 1} = \max_{1 \leq i \leq m} \sum_{1 \leq j \leq n} \abs{\bm{W}_{ij}}$. Let $1 \leq k \leq d$. We follow the same method for the norm of $\partial_k (\bm{\Theta} \circ u_\rho) - \partial_k (\bm{\Theta} \circ u_\sigma)$: it holds
    \begin{align*}
        \aabs{\partial_k (\bm{\Theta} \circ u_\rho)_i - \partial_k (\bm{\Theta} \circ u_\sigma)_i}_{L^\infty(\Omega)} & \leq \aabs{\bm{\Theta}}_{\infty, 1} \aabs{\nabla u_\rho - \nabla u_\sigma}_{L^\infty(\Omega, \RR^{d \times n})}.
    \end{align*}
    The same argument shows that $\aabs{\partial_k (\bm{\Theta} \circ u_\sigma)_i}_{L^\infty(\Omega)} \leq \aabs{\bm{\Theta}}_{\infty, 1} \aabs{\nabla u_\sigma}_{L^\infty(\Omega, \RR^{d \times n})}$. We conclude by taking the maximum on $1 \leq i \leq m$ and $1 \leq k \leq d$: the three bounds still hold for constants $C_1'$, $C_2'$, $C_3'$ and $C_{\bm{\vartheta}}'$ obtained by multiplying the original constants $C_1$, $C_2$, $C_3$ and $C_{\bm{\vartheta}}$ by $\aabs{\bm{\Theta}}_{\infty, 1}$, and taking $\alpha' = \alpha$.

        {\itshape (iii) Composition with an activation function.}
    We now compose $u_\rho$ by $\rho$ and $u_\sigma$ by $\sigma$. Let $1 \leq i \leq n$. We use the triangle inequality again and obtain
    \begin{align*}
        \aabs{(\rho \circ u_\rho)_i - (\sigma \circ u_\sigma)_i}_{L^\infty(\Omega)} & \leq \aabs{\rho \circ (u_\rho)_i - \rho \circ (u_\sigma)_i}_{L^\infty(\Omega)} + \aabs{\rho \circ (u_\sigma)_i - \sigma \circ (u_\sigma)_i}_{L^\infty(\Omega)}.
    \end{align*}
    We bound the first term using the Lipschitz continuity of $\rho$. The second term is bounded by the $L^\infty$ norm of $\rho - \sigma$. Altogether, applying the induction hypothesis, we reach
    \begin{align*}
        \aabs{(\rho \circ u_\rho)_i - (\sigma \circ u_\sigma)_i}_{L^p(\Omega)} & \leq \underbrace{\left[C_1 \lip{\rho} + 1\right]}_{= C_1'} \aabs{\rho - \sigma}_{L^\infty(\RR)}.
    \end{align*}
    We now turn to the bound for the gradient. By the chain rule, $\partial_k (\rho \circ (u_\rho)_i) = \rho' \circ (u_\rho)_i \partial_k (u_\rho)_i$. We use the triangle inequality to decompose the distance between $\partial_k (\rho \circ u_\rho)_i$ and $\partial_k (\rho \circ u_\rho)_i$ into three terms:
    \begin{align*}
        \aabs{\partial_k (\rho \circ u_\rho)_i - \partial_k (\sigma \circ u_\sigma)_i}_{L^\infty(\Omega)} & \leq \aabs{\rho' \circ (u_\rho)_i \partial_k (u_\rho)_i - \rho' \circ (u_\rho)_i \partial_k (u_\sigma)_i}_{L^\infty(\Omega)}       \\
                                                                                                          & + \aabs{\rho' \circ (u_\rho)_i \partial_k (u_\sigma)_i - \rho' \circ (u_\sigma)_i \partial_k (u_\sigma)_i}_{L^\infty(\Omega)}      \\
                                                                                                          & + \aabs{\rho' \circ (u_\sigma)_i \partial_k (u_\sigma)_i - \sigma' \circ (u_\sigma)_i \partial_k (u_\sigma)_i}_{L^\infty(\Omega)}.
    \end{align*}
    The first term is easily bounded using the fact that $\rho'$ is bounded on $\RR$. The second term is bounded owing to $\rho'$ being Lipschitz continuous and $\nabla u_\sigma$ being bounded. We bound the last term by using the $L^\infty$ norm of $\rho' - \sigma'$ and the fact that $\nabla u_\sigma$ is bounded. We finally obtain the following upper bound
    \begin{align*}
        \aabs{\partial_k (\rho \circ u_\rho)_i - \partial_k (\sigma \circ u_\sigma)_i}_{L^p(\Omega)} & \leq  \underbrace{\left[C_2 \aabs{\rho'}_{L^\infty(\RR)} + C_1 C_{\bm{\vartheta}} \lip{\rho'} \aabs{\sigma'}_{L^\infty(\RR)}^\alpha\right]}_{= C_2'} \aabs{\rho - \sigma}_{L^p(\RR)} \\
                                                                                                     & + \underbrace{\left[C_3 \aabs{\rho'}_{L^\infty(\RR)} + C_{\bm{\vartheta}} \aabs{\sigma'}_{L^\infty(\RR)}^\alpha\right]}_{= C_3'} \aabs{\rho' - \sigma'}_{L^p(\RR)}.
    \end{align*}
    Since $\sigma'$ is bounded on $\RR$ it holds $\aabs{\partial_k (\sigma \circ u_\sigma)_i}_{L^\infty(\Omega)} \leq \aabs{\sigma'}_{L^\infty(\RR)} \aabs{\nabla u_\sigma}_{L^\infty(\Omega, \RR^{d \times n})}$. We conclude by taking the maximum on $1 \leq i \leq m$ and $1 \leq k \leq d$: the three bounds still hold for constants $C_1'$, $C_2'$ and $C_3'$ indicated above, $C_{\bm{\vartheta}}' = C_{\bm{\vartheta}}$ and $\alpha' = \alpha + 1$.
\end{proof}

%% file: src/10.algorithms.tex
\subsection{Mesh extraction}
\label{app:algo_mesh}

In this section, we provide more detail on the way to construct an adapted mesh to a neural network based on a \ac{cpwl} approximation of its activation function, as described in \alg{adaptiveMesh}. We recall the notations introduced in \sect{adaptivity}.

Suppose that $\pi[u_k]$ is a \ac{cpwl} map and that its restriction to $R \subset \RR^d$ is $\pi[u_k]_{|R}(\bm{x}) = \bm{W}_R \bm{x} + \bm{b}_R$. For $1 \leq i \leq m$, we write $\bm{w}_i$ and $b_i$ the $i$-th row of $\bm{W}_R \in \RR^{m \times d}$ and $\bm{b}_R \in \RR^m$. Let $\pi[\rho]$ be a \ac{cpwl} function, and $(\xi_j)_{1 \leq j \leq K}$ denote its breakpoints. We also introduce the local coefficients $(\alpha_j, \beta_j)$ of $\pi[\rho]$ such that the restriction of $\pi[\rho]$ to $\cc{\xi_j}{\xi_{j+1}}$ has the expression $x \mapsto \alpha_j x + \beta_j$.

To identify the regions where $\pi[\rho] \circ \pi[u_k]_{|R}$ is linear, we need to consider the equations $\bm{w}_i \cdot \bm{x} + b_i = \xi_j$ for all $1 \leq i \leq m$ and $1 \leq j \leq K$. These hyperplanes intersect in $R$ and define subregions adapted to $\pi[\rho] \circ \pi[u_k]_{|R}$. This is the step that corresponds to $\mathtt{cutRegion}$ in \alg{adaptiveMesh}. We now present specific algorithms to define these subregions in dimensions zero, one, and two.

\subsubsection{Dimensions zero and one}
\label{app:algo_mesh_cut_01}

In dimension zero, a region is reduced to one single point so nothing needs to be done. In the one-dimensional case, $R$ is a segment that we write $\cc{p_-}{p_+}$. We first compute the range of each linear map $h_i: x \mapsto \bm{w}_i \cdot x + b_i$ on $R$ by computing their values at $p_-$ and $p_+$. We then find which breakpoints $\xi_j$ belong to that range via a binary search, and we compute the corresponding antecedents $p_{i, j}$ of $\xi_j$ by $h_i$ with normalised coordinates as follows
\[p_{i, j} = \frac{1}{2} (p_- + p_+) + \frac{t_{i, j}}{2} (p_+ - p_-), \qquad t_{i, j} = \frac{2 \xi_j - (h_i(p_+) + h_i(p_-))}{h_i(p_+) - h_i(p_-)} \in \cc{-1}{+1}.\]
We found this formulation more numerically stable compared to expressing $p_{i, j}$ in terms of $w_i$, $b_i$, and $\xi_j$. We sort the normalised coordinates $t_{i, j}$ and obtain the corresponding sorted abscissae $(p_l)$. For each segment $\cc{p_l}{p_{l+1}}$, we evaluate the linear maps $h_i$ at $\frac{1}{2} (p_l + p_{l+1})$ and retrieve the coefficients $(\alpha_{l, i}, \beta_{l, i})$ of $\pi[\rho]$ at this location to update the local expression $(\bm{W}_R, \bm{b}_R)$ into $(\mathrm{diag}(\bm{\alpha}_l) \bm{W}_R, \mathrm{diag}(\bm{\alpha}_l) \bm{b}_R + \bm{\beta}_l)$.

\subsubsection{Dimension two}
\label{app:algo_mesh_cut_2}

The two-dimensional case is much more involved as the hyperplanes are lines that can intersect with the boundary of $R$ and among one another. We consider four steps.

    {\itshape (i) Intersection of the hyperplanes with $\partial R$.} Since $R$ is a convex polygon, each hyperplane $h_i: \bm{x} \mapsto \bm{w}_i \cdot \bm{x} + b_i$ can intersect with $\partial R$ at two points at most. If the number of intersections is zero or one, then $R$ lies entirely on one side of the hyperplane thus we only need to handle the case when there are two intersection points. One notable corner case occurs when the hyperplane coincides with an edge of the polygon. In this degenerate case, the region is not cut.

The clipping of a line by a convex polygon can be done in logarithmic time, using a method akin to binary search. We refer the reader to \cite{skala1994} for the full algorithm. We write $(P_{i, j}^-, P_{i, j}^+)$ the clipping of $h_i$ by $R$, and $(P_l)$ the collection of these vertices.

    {\itshape (ii) Intersection of the hyperplanes among themselves.} Next, we check the intersection of pairwise combinations of the lines obtained above. It is important to note that for a fixed row $i$, the lines $(P_{i, j}^-, P_{i, j}^+)$ are parallel. Indeed, they are solutions to $\bm{w}_i \cdot \bm{x} + b_i = \xi$ for two different values of $\xi$. This means that we can skip the tests corresponding to lines arising from the same row. Let $i_1 \neq i_2$ be two different row indices, $j_1$ and $j_2$ be two breakpoint indices, and suppose that $(P_{i_1, j_1}^-, P_{i_1, j_1}^+)$ and $(P_{i_2, j_2}^-, P_{i_2, j_2}^+)$ are two non-parallel lines. We compute their intersection $Q$ via the following parametrisation and solve for $s$ and $t$:
\[Q = \overline{P}_{i_1, j_1} + s \Delta P_{i_1, j_1} = \overline{P}_{i_2, j_2} + t \Delta P_{i_2, j_2}, \qquad \begin{bmatrix} \Delta P_{i_1, j_1} & -\Delta P_{i_2, j_2} \end{bmatrix} \begin{bmatrix} s \\ t \end{bmatrix} = \begin{bmatrix} \overline{P}_1 - \overline{P}_2 \end{bmatrix},\]
where $\overline{P}_{i, j} = \frac{1}{2}(P_{i, j}^- + P_{i, j}^+)$ is the middle of the segment and $\Delta P_{i, j} = \frac{1}{2}(P_{i, j}^+ - P_{i, j}^-)$ is half the direction vector. We need to check that $s$ and $t$ are between $-1$ and $+1$.

    {\itshape (iii) Representation of these intersections as an adjacency graph.} We initialise a connectivity graph with the edges of the region $R$. During step {\itshape (i)}, for each line $(P_{i, j}^-, P_{i, j}^+)$ that cuts the polygon, we add the points $P_{i, j}^-$ and $P_{i, j}^+$ to the edges they belong to, and add an arc in the graph, that connects these two points. During step {\itshape (ii)}, we only need to add new vertices to the graph.

    {\itshape (iv) Extraction of the smallest cycles in this graph.} Finding the partition of $R$ where the composition by $\pi[\rho]$ is \ac{cpwl} amounts to finding the fundamental cycle basis of the undirected graph defined above. We refer to \cite{kavitha2009} for a definition of the fundamental cycle basis of a planar graph and for an algorithm to build it.

\begin{algorithm}
    \caption{$\mathtt{adaptive\char`_mesh}(u)$}
    \label{alg:adaptiveMesh}
    \begin{algorithmic}
        \State $\tau \gets [(\Omega, \bm{I}_d, \bm{0}_d)]$
        \For{$\ell \in \mathtt{layers}(u)$}
        \State $\bm{W} \gets \mathtt{weight}(\ell), \quad \bm{b} \gets \mathtt{bias}(\ell), \quad \tau^{\text{new}} \gets \varnothing$
        \For{$(R, \bm{W}_R, \bm{b}_R) \in \tau$}
        \State $\bm{W}_R^{\text{new}} \gets \bm{W} \times \bm{W}_R, \quad \bm{b}_R^{\text{new}} \gets \bm{W} \times\bm{b}_R + \bm{b}$ \Comment{Compose by linear map}
        \For{$(S, \bm{W}_S, \bm{b}_S) \in \mathtt{cut\char`_region}(\pi[\rho], R, \bm{W}_R^{\text{new}}, \bm{b}_R^{\text{new}})$} \Comment{See \app{algo_mesh_cut_01} and \app{algo_mesh_cut_2}.}
        \State $(\bm{\alpha}, \bm{\beta}) \gets \mathtt{coefficients}(\pi[\rho], S)$ \Comment{Retrieve local coefficients of $\pi[\rho]$ in $S$}
        \State $\bm{W}_S^{\text{new}} \gets \mathtt{diag}(\bm{\alpha}) \times \bm{W}_S, \quad \bm{b}_S^{\text{new}} \gets \mathtt{diag}(\bm{\alpha}) \times \bm{b}_S + \bm{\beta}$ \Comment{Compose by $\pi[\rho]$}
        \State $\tau^{\text{new}} \gets \tau^{\text{new}} \cup \{(S, \bm{W}_S^{\text{new}}, \bm{b}_S^{\text{new}})\}$
        \EndFor
        \EndFor
        \State $\tau \gets \tau^{\text{new}}$
        \EndFor
        \State \Return $\tau$
    \end{algorithmic}
\end{algorithm}

\subsubsection{Notes on the algorithm}

In \alg{adaptiveMesh}, $\mathtt{weight}(\ell)$ and $\mathtt{bias}(\ell)$ denote the parameters of the network at layer $\ell$. The instruction $\mathtt{coefficients}$ is just a lookup operation to retrieve the local coefficients of $\pi[\rho]$. Here its outputs are vectors because the local map on a region is vector-valued. Finally if $\bm{\alpha}$ is a vector, $\mathtt{diag}(\bm{\alpha})$ is the diagonal matrix whose diagonal coefficients are $\bm{\alpha}$.

\subsection{Decomposition of convex polygons}
\label{app:algo_polygon}

In the general two-dimensional case, the cells of the mesh adapted to $u$ are arbitrary convex polygons. Numerical quadrature rules are known for triangles and convex quadrangles, so we decide to split the cells into a collection of these two reference shapes. We encode hard rules to decompose convex polygons with up to ten vertices into triangles and convex quadrangles. For example, an octagon is split into one triangle and two quadrangles by considering the three cycles $(1, 2, 3, 4)$, $(4, 5, 6, 7)$, $(7, 8, 1, 4)$. Larger polygons are recursively split in half until they have fewer than ten vertices.